\documentclass[12pt,a4paper]{amsart}
\usepackage[english]{babel} 
\usepackage[utf8]{inputenc}
\usepackage{amssymb}
\usepackage{float}
\usepackage{amscd}
\usepackage{amsthm}
\usepackage{mathtools}
\usepackage[top=2cm, bottom=1.9cm, left=2cm, right=2cm,twoside=false]{geometry}
\usepackage{array}
\usepackage{longtable}
\usepackage{graphicx}
\usepackage{url} 
\usepackage{wrapfig}
\usepackage{color}
\usepackage[usenames,x11names]{xcolor}
\usepackage[all]{xy}
\usepackage{enumerate}
\usepackage{multirow}
\usepackage{units} 
\usepackage{hyperref} 
\hypersetup{linkcolor  =DodgerBlue3, citecolor  = teal, urlcolor   = teal, colorlinks = true, hyperfootnotes =false}
\usepackage{yfonts}
\usepackage{mathrsfs}
\usepackage{caption}
\usepackage{subcaption}
\usepackage{multicol}

\usepackage{lscape}

\renewcommand{\to}{\longrightarrow}

\newcommand{\into}{\hookrightarrow}

\def\Spec{\operatorname{Spec}}

\newcommand{\F}{\mathbb{F}}
\newcommand{\N}{\mathbb{N}}

\newcommand{\C}{\mathbb{C}}

\newcommand{\Q}{\mathbb{Q}}

\renewcommand{\P}{\mathbb{P}}

\renewcommand{\O}{\mathcal{O}}
\renewcommand{\tilde}{\widetilde}

\renewcommand{\epsilon}{\varepsilon}
\renewcommand{\phi}{\varphi}
\renewcommand{\theta}{\vartheta}
\renewcommand{\d}{\partial}
\newcommand{\id}{\mathrm{id}}

\newcommand{\ftip}[1]{\mathrm{tip}^{+}(#1)}
\newcommand{\ltip}[1]{\mathrm{tip}^{-}(#1)}

\newcommand{\floor}[1]{\lfloor #1 \rfloor}
\newcommand{\de}{\coloneqq}

\def\gcd{\operatorname{gcd}}
\def\NS{\operatorname{NS}}
\def\Exc{\operatorname{Exc}}

\newcommand{\redd}{_{\mathrm{red}}}
\newcommand{\rev}[1]{#1^{t}}
\newcommand{\s}{^{\dagger}\ \!\!}

\renewcommand{\leq}{\leqslant}
\renewcommand{\geq}{\geqslant}

\newcommand{\FZa}{\mathcal{FZ}_{1}}  
\newcommand{\cA}{\mathcal{A}} 
\newcommand{\cB}{\mathcal{B}} 
\newcommand{\cC}{\mathcal{C}} 
\newcommand{\cD}{\mathcal{D}} 
\newcommand{\cE}{\mathcal{E}}
\newcommand{\cF}{\mathcal{F}}
\newcommand{\cG}{\mathcal{G}}
\newcommand{\ORa}{\mathcal{OR}_{1}}
\newcommand{\ORb}{\mathcal{OR}_{2}}

\theoremstyle{plain}
\newtheorem{tw}{Theorem}[section]
 \makeatletter \@addtoreset{step}{tw}\makeatother

\theoremstyle{definition}

\newtheorem{lem}[tw]{Lemma}
\newtheorem{prop}[tw]{Proposition}
\newtheorem{wn}[tw]{Corollary}
\newtheorem{prz}[tw]{Example}

\newtheorem{conjecture}[tw]{Conjecture}
\newtheorem{notation}[tw]{Notation}
\newtheorem{rem}[tw]{Remark} 

\theoremstyle{remark}
\newtheorem*{uw}{Remark} 
\newtheorem{claim}{Claim}
\newtheorem*{claim*}{Claim}

\makeatletter \def\subsection{\@startsection{subsection}{3}%
   \z@{.5\linespacing\@plus.7\linespacing}{.5\linespacing}%
   {\bfseries\itshape}} \makeatother  

\makeatletter \def\@tocline#1#2#3#4#5#6#7{\relax \ifnum #1>\c@tocdepth \else \par \addpenalty\@secpenalty\addvspace{#2} \begingroup \hyphenpenalty\@M \@ifempty{#4}{\@tempdima\csname r@tocindent\number#1\endcsname\relax}{\@tempdima#4\relax} \parindent\z@ \leftskip#3\relax \advance\leftskip\@tempdima\relax \rightskip\@pnumwidth plus4em \parfillskip-\@pnumwidth #5\leavevmode\hskip-\@tempdima \ifcase #1 \or\or \hskip 1em \or \hskip 2em \else \hskip 3em \fi #6\nobreak\relax \dotfill\hbox to\@pnumwidth{\@tocpagenum{#7}}\par \nobreak \endgroup  \fi}
\makeatother

\makeatletter \renewenvironment{proof}[1][\proofname]{
  \par\pushQED{\qed}\normalfont
  \topsep6\p@\@plus6\p@\relax
  \trivlist\item[\hskip\labelsep\bfseries#1\@addpunct{.}]
  \ignorespaces}{
  \popQED\endtrivlist\@endpefalse} \makeatother

\def\:{\colon}

\numberwithin{equation}{section}
\def\8{\infty}

\begin{document}
\title[Planar rational cuspidal curves\\ I. $\C^{**}$-fibrations]{Classification of planar rational cuspidal curves\\ I. $\C^{**}$-fibrations}
\author{Karol Palka}
\author{Tomasz Pełka}
\address{Institute of Mathematics, Polish Academy of Sciences, ul. Śniadeckich 8, 00-656 Warsaw, Poland}
\email{palka@impan.pl}
\email{tpelka@impan.pl}
\thanks{This research project has begun as a part of the Homing Plus programme of the Foundation for Polish Science, cofinanced from the EU, and has been finalized as a part of the research project 2015/18/E/ST1/00562, National Science Centre, Poland}

\begin{abstract}{To classify complex rational cuspidal curves $E\subseteq \mathbb{P}^2$ it remains to classify the ones with complement of log general type, i.e. the ones for which $\kappa(K_X+D)=2$, where $(X,D)$ is a log resolution of $(\mathbb{P}^2,E)$. It is conjectured that $\kappa(K_X+\frac{1}{2}D)=-\infty$ and hence $\mathbb{P}^2\setminus E$ is $\mathbb{C}^{**}$-fibered, where $\mathbb{C}^{**}=\mathbb{C}^1\setminus\{0,1\}$, or $-(K_X+\frac{1}{2}D)$ is ample on some minimal model of $(X,\frac{1}{2}D)$. Here we classify, up to a projective equivalence, those rational cuspidal curves for which the complement is $\mathbb{C}^{**}$-fibered. From the rich list of known examples only very few are not of this type. We also discover a new infinite family of bicuspidal curves with unusual properties.}

\end{abstract}

\maketitle

\section{Main result}

All varieties considered are complex algebraic. By a \emph{curve} we mean an irreducible and reduced variety of dimension $1$. Let $\bar{E}\subseteq \P^{2}$ be a rational curve which is \emph{cuspidal}, that is, its singularities are locally analytically irreducible. One can equivalently view $\bar{E}$ as the image of an injective morphism $\P^{1}\to \P^{2}$.  The literature on such curves is rich (see \cite{FLMN_cusps_and_open_surfaces} and \cite{Moe-cuspidal_MSc} for a summary). A basic invariant of $\bar{E}\subseteq \P^{2}$ is the logarithmic Kodaira-Iitaka dimension $\kappa\de \kappa(\P^{2}\setminus \bar{E})$. Cases with $\kappa\neq 2$ have already been classified (see Lemma \ref{lem:kappa<=1}). In particular, their complement is $\C^{1}$- or $\C^{*}$-fibered, where $\C^*=\C^1\setminus\{0\}$. 

We will therefore assume that $\kappa=2$, that is, $\P^{2}\setminus \bar{E}$ is \emph{of log general type}. There are many examples, including several infinite series depending on up to three discrete parameters constructed by various authors, but there is no classification. In fact not many general properties have been proved so far, mostly because the existing theory of log surfaces does not give efficient methods in case $\kappa=2$. However, recently M. Koras and the first author proved the Coolidge-Nagata conjecture \cite{Palka-Coolidge_Nagata1}, \cite{KoPa-CooligeNagata2}, which asserts that all rational cuspidal curves are Cremona equivalent to a line. The proof uses, among others, a modification of the Log Minimal Model Program with half-integral coefficients, as developed in \cite{Palka-minimal_models}, and which is based on the generalization of the notion of \emph{almost minimality} (cf.\ \cite[\S 2.3.11]{Miyan-OpenSurf}). We continue this approach. It turns out that the geometry of $\bar E\subseteq \P^2$ is closely related to the geometry of the pair $(\P^2,\frac{1}{2}\bar E)$. In particular, we have the following conjecture which strengthens the Coolidge-Nagata conjecture and the Flenner-Zaidenberg Weak Rigidity Conjecture (see\ Sec. \ref{subsection:QHP}). Put $\C^{**}=\C^1\setminus\{0,1\}$.

\begin{conjecture}[Negativity Conjecture, \cite{Palka-minimal_models} 4.7]\label{conj}
If $(X,D)$ is a smooth completion of a $\Q$-acyclic surface then $\kappa(K_{X}+\frac{1}{2}D)=-\infty$. In particular, if $X\setminus D$ is of log general type then it has a $\C^{**}$-fibration or the pushforward of $-(K_X+\frac{1}{2}D)$ is ample on some minimal model of $(X,\frac{1}{2}D)$.
\end{conjecture}

The $\C^{**}$-fibration comes as a restriction of a log Mori fibration over a curve of the minimal model of $(X,\frac{1}{2}D)$. Note also that $\P^2\setminus\bar E$ is $\Q$-acyclic. The goal of the current and a forthcoming article is to obtain a classification of rational cuspidal curves (with complements of log general type) up to a projective equivalence, under the assumption that Conjecture \ref{conj} holds. 

Here we work out the most technical part, that is, we classify possible singularities of these $\bar E\subseteq \P^2$ for which $\P^2\setminus \bar E$ admits a $\C^{**}$-fibration. A vast majority of known examples turns out to share this property. It is then no wonder that to achieve a classification we need to go through a detailed combinatorial analysis, even if the general ideas are clear. Let us call an injective morphism with a singular image a \emph{singular embedding}. We obtain the following geometric description.

\begin{tw}[Rational cuspidal curves with $\C^{**}$-fibered complements]\label{thm:geometric}
Let $\bar{E}\subseteq \P^{2}$ be a rational cuspidal curve with complement of log general type. If $\P^{2}\setminus \bar{E}$ is $\C^{**}$-fibered then one of the following holds (see Fig.\ \ref{fig:thm_geometric}):
\begin{enumerate}[(a)]
\item $\bar{E}$ is projectively equivalent to one of the Orevkov unicuspidal curves $C_{4k}$, $C^*_{4k}$ for some $k\geq 1$ \cite[Thm.\ C]{OrevkovCurves}.

\item $\bar E$ has two cusps and the line $\ell$ joining them has no more common points with $\bar E$, hence $\bar{E}\setminus \ell \subseteq \P^{2}\setminus \ell$ is a closed embedding of $\C^*$ into $\C^2$.

\item $\bar E$ has two cusps and the line $\ell$ tangent to one of them meets $\bar E$ in one more point, transversally, hence $\bar{E}\setminus \ell \subseteq \P^{2}\setminus \ell$ is a singular embedding of $\C^*$ into $\C^2$. The cusp of $\bar E\setminus \ell$ has multiplicity $2$.

\item $\bar{E}$ is projectively equivalent to one of the Flenner-Zaidenberg tricuspidal curves of the first kind \cite[3.5]{FLZa-_class_of_cusp}.
The line $\ell$ through two cusps, including the one with highest multiplicity, has no more common points with $\bar E$, hence $\bar{E}\setminus \ell \subseteq \P^{2}\setminus \ell$ is a singular embedding of $\C^*$ into $\C^2$. The cusp of $\bar E\setminus \ell$ has multiplicity $2$. 
\end{enumerate}
\end{tw}

\begin{figure}[H]
 \begin{subfigure}{0.3\textwidth}
 \includegraphics[scale=0.15]{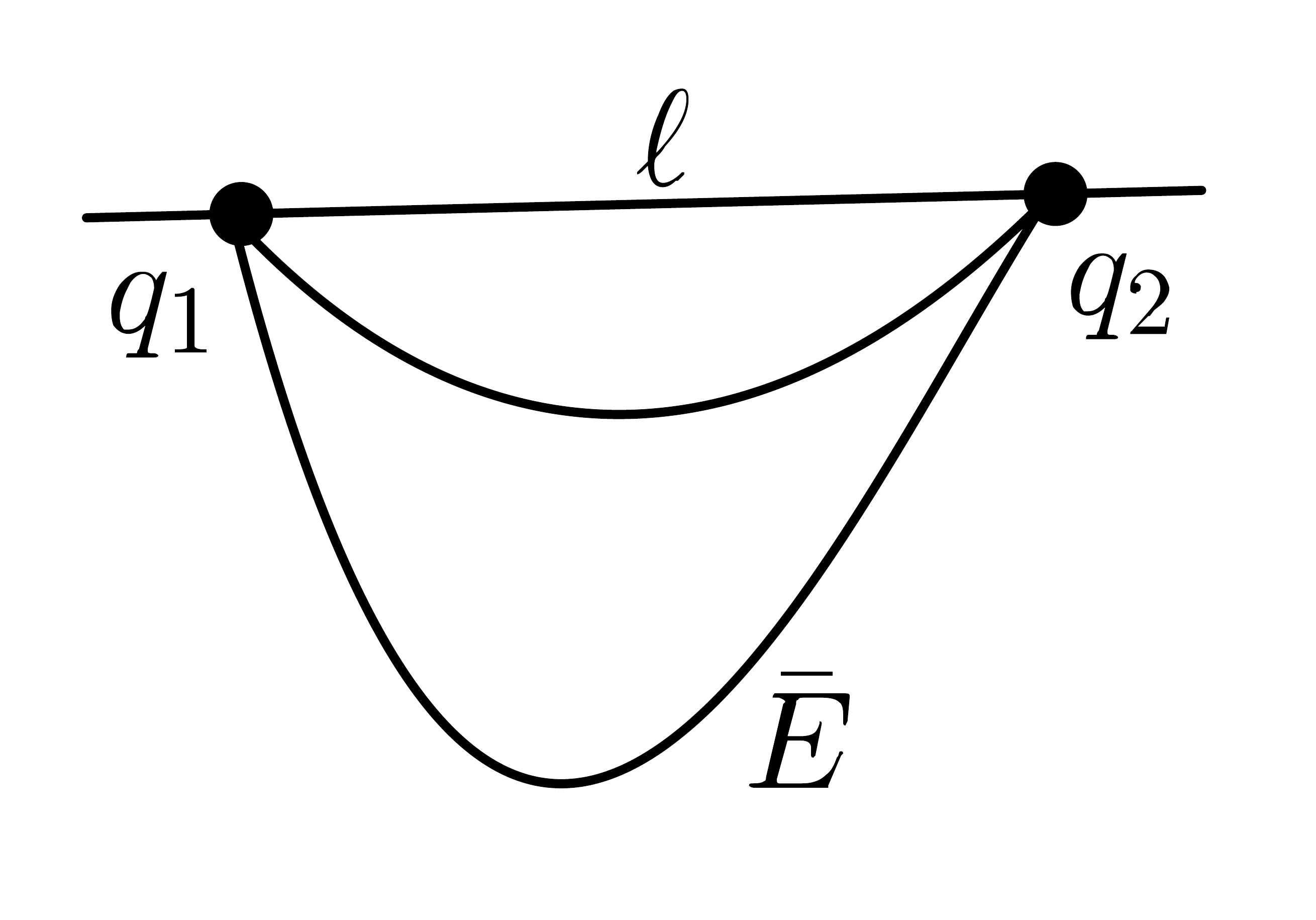}
 \caption{\ref{thm:geometric}(b)}
 \end{subfigure}
 ~
 \begin{subfigure}{0.3\textwidth}
 \includegraphics[scale=0.15]{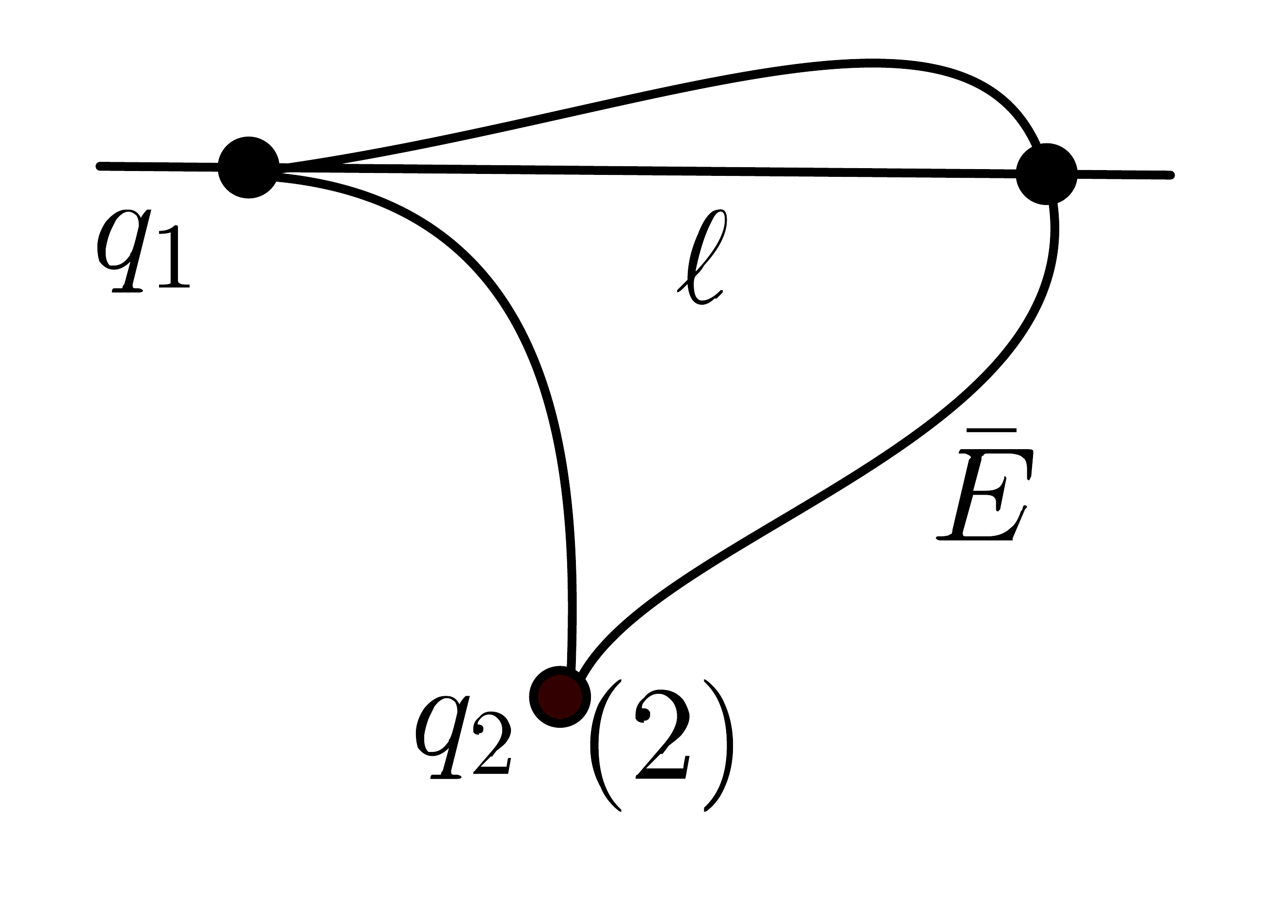}
 \caption{\ref{thm:geometric}(c)}
 \end{subfigure}
 ~
 \begin{subfigure}{0.3\textwidth}
 \includegraphics[scale=0.15]{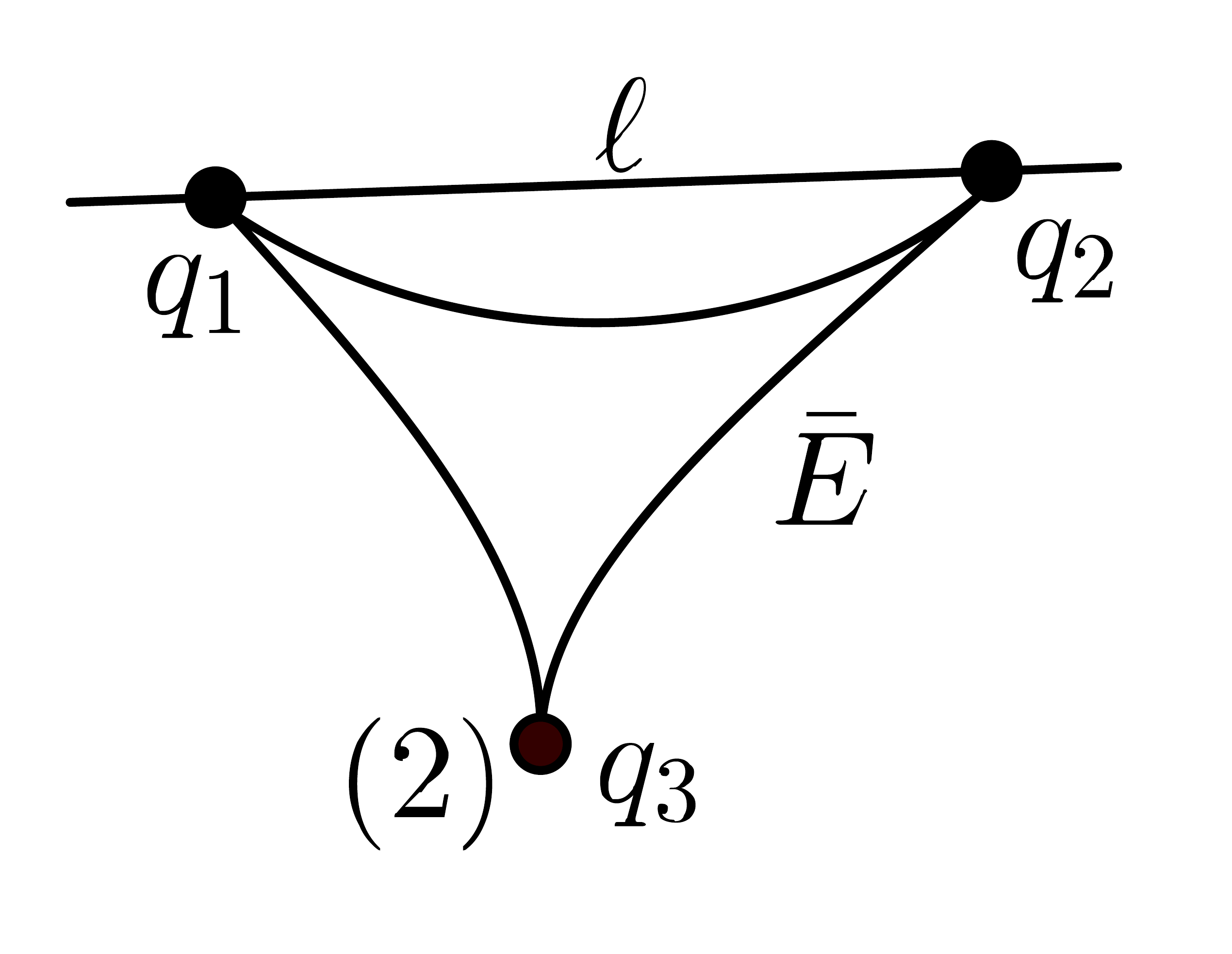}
 \caption{\ref{thm:geometric}(d)}
 \end{subfigure}
 \caption{Configuration $(\P^{2}, \bar{E}+\ell)$ from Theorem \ref{thm:geometric}.}
 \label{fig:thm_geometric}
\end{figure}

A detailed classification is given in Theorem \ref{thm:possible_HN-types} and we obtain Theorem \ref{thm:geometric} as a corollary. For the definition of Hamburger-Noether pairs, which give a convenient way of describing singularities, see Section \ref{sec:cusps}. Recall that the Fibonacci sequence is defined by $F_{0}=0$, $F_{1}=1$ and $F_{j+2}=F_{j+1}+F_{j}$. 

\begin{tw}[Classification of rational cuspidal curves with $\C^{**}$-fibered complements]\label{thm:possible_HN-types}
Let $\bar{E}\subseteq \P^{2}$ be a rational cuspidal curve. Denote by $c$ the number of cusps of $\bar E$ and by $E$ the proper transform of $\bar{E}$ under the minimal log resolution of $(\P^{2},\bar{E})$. Then $\P^2\setminus \bar E$ is a $\C^{**}$-fibered surface of log general type if and only if the cusps of $\bar E$ have one of the following sequences of Hamburger-Noether pairs (the corresponding multiplicity sequences are listed in Table \ref{table:fibrations}): 

\smallskip \textbf{Case $c=3$:}
\begin{enumerate}
\item[($\FZa$)] 
$\tbinom{2k+1}{2}$, $\tbinom{d-1}{d-2}$ and $\tbinom{2(d-2-k)+1}{2}$ for some integers $d-3\geq k \geq \tfrac{1}{2}d-1 \geq 1$.  In particular, $\deg\bar{E}=d$ and $E^{2}=-(d-2)$.
\end{enumerate}

\textbf{Case $c=2$:}
\begin{enumerate}
\item[($\cA$)] $\tbinom{ps(\gamma +1)}{ps\gamma }\tbinom{ps}{p}\tbinom{p}{1}$ and $\tbinom{\gamma (ps+1)+p(s-1)+1}{ps+1}$ 
for some integers $\gamma ,p,s$ satisfying $p\geq 2$, $\gamma ,s\geq 1$ and $(\gamma,p)\neq (1,2)$. 
In particular, $\deg\bar{E}=(\gamma +1)ps+1$ and $E^2=-\gamma$.

\item[($\cB$)]  $\tbinom{(ps-1)(\gamma+1)}{(ps-1)\gamma}\tbinom{ps-1}{p}$ and $\tbinom{p(\gamma s+s -1)}{ps}\tbinom{p}{1}$ 
for some integers $\gamma ,p,s$ satisfying $p,s\geq 2$, $\gamma\geq 1$ and $(\gamma,p)\neq (1,2)$. 
In particular, $\deg\bar{E}=(\gamma+1)ps-\gamma$ and $E^2=-\gamma$.

\item[($\cC$)] $\tbinom{p(\gamma s+s+1)}{p(\gamma s+1)}\tbinom{p}{1}$ and $\tbinom{(\gamma +1)(ps+1)+p}{ps+1}$ 
for some integers $p\geq 2$ and $\gamma ,s\geq 1$. 
In particular, $\deg\bar{E}=(\gamma s+s+1)p+1$ and $E^2=-\gamma$. 

\item[($\cD$)] $\tbinom{(\gamma +1)(ps-1)+p}{\gamma (ps-1)+p}$ and $\tbinom{p(\gamma s+s+1)}{ps}\tbinom{p}{1}$ 
for some integers $p\geq 2$ and $\gamma ,s\geq 1$. 
In particular, $\deg\bar{E}=(\gamma s+s+1)p-\gamma $ and $E^2=-\gamma$.

\item[($\cE$)] $\tbinom{8k+8}{4k+2}\tbinom{2}{1}$ and $\tbinom{8k+4}{4k+4}\tbinom{4}{1}$ 
for some integer $k\geq 1$. 
In particular, $\deg\bar{E}=8k+6$ and $E^{2}=-2$.

\item[($\cF$)] $\tbinom{8k}{4k+2}\tbinom{2}{1}$ and $\tbinom{8k+4}{4k}\tbinom{4}{1}$ 
for some integer $k\geq 1$.
In particular, $\deg\bar{E}=8k+2$ and $E^{2}=-2$.

\item[($\cG$)] $\tbinom{4\gamma-3}{\gamma-1}$ and $\tbinom{2\gamma-1}{2}$ 
for some integer $\gamma\geq 3$. 
In particular, $\deg\bar{E}=2\gamma-1$ and $E^2=-\gamma$.
\end{enumerate}

\textbf{Case $c=1$:}
\begin{enumerate}
\item[($\ORa$)] $\tbinom{F_{4k+4}}{F_{4k}}\tbinom{3}{1}$ 
for some integer $k\geq 1$.
In particular, $\deg\bar{E}=F_{4k+2}$ and $E^{2}=-2$.

\item[($\ORb$)] $\tbinom{2F_{4k+4}}{2F_{4k}}\tbinom{6}{1}$ 
for some integer $k\geq 1$. 
In particular, $\deg\bar{E}=2F_{4k+2}$ and $E^{2}=-2$.
\end{enumerate}

\vspace{0.5em}
\noindent
All the above HN-types are different from each other and each (for every admissible value of $d$, $k$, $\gamma$, $p$, $s$) is realized by a rational cuspidal planar curve, which is unique up to a projective equivalence.
\end{tw}

The HN-type $\FZa$ with parameters $d$, $k$ will be denoted by $\FZa(d,k)$. The  HN-type $\cA$ with parameters $\gamma$, $p$, $s$ will be denoted as $\cA(\gamma,p,s)$, similarly for $\cB$, $\cC$ and $\cD$. The HN-type $\cE$ with parameter $k$ will be denoted by $\cE(k)$, similarly for the remaining HN-types. All families except $\cG$ have already appeared in literature. For a detailed discussion see Section \ref{sec:existence_of_fibrations}. The bicuspidal curves $\cG$ are new (see however the discussion in Sec.\ \ref{sec:construction_(h)}). As we will see, these are exactly the curves in Theorem \ref{thm:geometric}(c). They are geometrically different than other bicuspidal curves listed in Theorem \ref{thm:possible_HN-types} in the sense that the line joining the cusps meets $\bar E\setminus\{q_1, q_2\}$. Like all bicuspidal curves above, they are closures of (singular or smooth) $\C^*$-embeddings into $\C^2$, namely the embeddings \cite[(b)]{BoZo-annuli} with parameters $m,k$ both equal to $\gamma-1$. However, opposite to other cases, one needs to make a specific choice of coordinates on $\C^2$, different than the one in loc.\ cit., to obtain them. We explain the construction in Lemma \ref{lem:type_G}, see also Section \ref{rem:BoZo}. 

We will address the problem of classification of the remaining rational cuspidal curves (the log del Pezzo case) in a forthcoming article.

\begin{rem}[Non-standard HN-types]\label{rem:special_HN_cases} The HN-types listed above are standard (see Sec.\ \ref{sec:cusps} for definitions) except for $\cA(\gamma\geq 2,p,1)$, $\cD(\gamma,p,1)$, $\cA(1,p,s)$, $\cB(1,p,s)$, $\cF(1)$, $\ORa(1)$ and $\ORb(1)$. We did not convert them to standard HN-pairs to keep some homogeneity of notation. In the first two cases the standard HN-pairs are $\binom{p(\gamma+1)}{p\gamma}\binom{p}{p+1}$  for the first cusp in $\cA(\gamma,p,1)$, $\binom{(\gamma +2)p+1}{p}$ for the second cusp in $\cD(\gamma,p,1)$, $\binom{2ps+p}{ps}\binom{p}{1}$ for the first cusp in $\cA(1,p,s)$ provided $s\neq 1$ and $\binom{3p+1}{p}$ otherwise, $\binom{2(ps-1)+p}{ps-1}$ for the first cusp in $\cB(1,p,s)$, $\binom{13}{4}$ for the second cusp in $\cF(1)$, $\binom{22}{3}$ for $\ORa(1)$ and $\binom{43}{6}$ for $\ORb(1)$.
\end{rem}	 

As explained above, our approach is based on the Log MMP run for $(X,\frac{1}{2}D)$, which strongly motivates the study of $\C^{**}$-fibrations, as they are related to the log Mori Fiber Space structures over curves (see the remarks after Lemma \ref{lem:Qhp_has_no_lines}). But in the current article the proofs are more elementary in spirit. The list is obtained based mostly on the study of degenerate fibers of the minimal completing $\P^{1}$-fibration and the interplay with the geometry of the contractible divisor $D-E$. The key part, classification of singularities, is done in Section \ref{sec:classification_of_sing}. First, in Proposition \ref{prop:some_Cstst_fibration_extends} we prove in a constructive way that one may choose a $\C^{**}$-fibration of $\P^2\setminus\bar E$ which has no base points on the minimal log resolution. Interestingly, this property holds for all $\Q$-homology planes of log general type which admit some $\C^{**}$-fibration (an analogue for $\C^*$-fibrations holds too, see Proposition \ref{prop:some_Cst_fibration_extends}). We note here that many results on fibers of $\C^{**}$-fibrations of $\Q$-homology planes have been proved in \cite{MiySu-Cstst_fibrations_on_Qhp}. However, we need a much more detailed analysis. Then, given the extension of the fibration to $X$, we study the specific cases when $D$ contains a full fiber (Lemma \ref{lem:case_nu=1}) or a multiple section (Lemma \ref{lem:case_h=2}, Corollary \ref{cor:case_h=2}). The case into which most of the curves fall is when $D$ contains three sections but no full fiber. To obtain a classification in this situation we show that there are at most three degenerate fibers and all of them are rational chains (Lemma \ref{lem:fibers_are_chains}), and then that one may assume $E$ is vertical (Lemmas  \ref{lem:E_is_vert_if_c=1} and \ref{lem:E_is_vert_for_c>=2}). We use also the results of Tono, who analyzed some special cases ($c=1$, $E^2=-2$ and $c=2$, $E^2=-1$, see Lemma \ref{lem:Tono_E2=-1,-2}). 

In Section \ref{ssec:existence_of_fibrations} we show that, conversely, if $\bar{E}\subseteq \P^{2}$ is of one of the HN-types listed in Theorem \ref{thm:possible_HN-types} then its complement is $\C^{**}$-fibered and of log general type. Finally, in Sections \ref{ssec:AD} - \ref{sec:construction_(h)} we prove the existence and uniqueness of curves of each HN-type listed in Theorem \ref{thm:possible_HN-types}. For families $\FZa$ and $\ORa$, $\ORb$ this result follows from the original constructions \cite{FlZa-rational_curves_and_singularities,OrevkovCurves}. We prove that the curves of HN-types $\cA$ - $\cF$ arise as closures of embeddings of $\C^{*}$ into $\C^{2}$ and we show how to deduce their uniqueness from the classification of such embeddings obtained in  \cite{CKR-Cstar_good_asymptote, KoPa-SporadicCstar2}. The existence and uniqueness of the new curves of HN-types $\cG(\gamma)$ is proven in Lemma \ref{lem:type_G}.

The proof of the implication \ref{thm:possible_HN-types}$\implies$\ref{thm:geometric} is given in Section \ref{sec:1.2=>1.1}. In the Appendix (Section \ref{sec:appendix}) we show how various numerical invariants  of cusps, including the sequences of HN-pairs, can be computed one from another.

\tableofcontents

\section{Preliminaries}\label{sec:preliminaries}

\subsection{Log surfaces.}

For a smooth projective surface $X$ we denote by $\NS(X)$ the Néron-Severi group of Cartier divisors on $X$ modulo numerical equivalence. Put $\NS_{\Q}(X)=\NS(X)\otimes \Q$ and $\rho(X)=\dim_{\Q}(\NS_{\Q}(X))$. A smooth rational projective curve $C$ on $X$ with $C^{2}=n$ is called an \emph{$n$-curve}. By the adjunction formula $K_X\cdot C=-2-n$ for such a curve, where $K_X$ denotes the canonical divisor.

Let $D$ be an effective divisor on $X$. By a \emph{component} of $D$ we always mean an irreducible component of $D$ (not a connected component). The number of components of $D$ is denoted by $\#D$. We call $D$ a \emph{simple normal crossing} divisor (\emph{snc} for short) if its components are smooth  and its only singularities are nodes. We call $(X,D)$ a \emph{smooth pair} if $X$ is a smooth projective surface and $D$ is a reduced snc-divisor on $X$. For every effective divisor $D$ on $X$ there is a sequence of blow-ups $(X',D')\to (X,D)$, such that the total transform $D'$ of $D$ is snc: such a process is called a \emph{log resolution of singularities}. For every smooth surface $V$ there exists a smooth pair $(X,D)$, such that $X\setminus D\cong V$. Every such pair is called a \emph{smooth completion} of $V$. If $(X,D)$ is a smooth pair with $D\neq 0$ then a blowup $X'\to X$ at a point $p\in D$ is called \emph{outer} for $D$ if $p$ is a smooth point of $D$, otherwise it is \emph{inner} for $D$.

Assume that $D$ is a reduced snc-divisor. A \emph{dual graph} of $D$ has vertices representing components of $D$ with weights being their self-intersection numbers, and an edge between two vertices represents a point of intersection of the corresponding components. If $D_{0}\leq D$ is an effective subdivisor of $D$ we define its \emph{branching number} as:
\begin{equation*}
\beta_{D}(D_{0})=D_{0}\cdot (D-D_{0}).
\end{equation*}
If $L$ is a component of $D$ with $\beta_{D}(L)\leq 1$ we say that $L$ is a \emph{tip} of $D$. If $\beta_{D}(L)\geq 3$ we say that $L$ is \emph{branching} in $D$. We say that an effective snc-divisor $T$ is a \emph{tree} if the dual graph of $T\redd$ is connected and contains no loop. Further, $T$ is a \emph{rational} tree if additionally all components of $T$ are rational curves.
 
A rational tree $T$ with no branching component is called a \emph{chain}. It is \emph{ordered} if it has a distinguished \emph{first} tip. If $T$ is ordered and reduced we can write $T=L_{1}+\dots+L_{r}$, where $L_{i}$'s are irreducible and $L_{i}\cdot L_{i+1}=1$ (hence $L_{i}\cdot L_{j}=0$ for $|i-j|>1$). Assuming $L_{1}$ is the distinguished tip of $T$ we write also $T=[a_{1},\dots, a_{r}]$, where $a_{i}=-L_{i}^{2}$. If $A=[a_{1},\dots, a_{r}]$ we write $\rev{A}=[a_{r},\dots, a_{1}]$ for the same chain with opposite order. If a rational tree $T$ has only one branching component $B$ and $\beta_{T\redd}(B)=3$ then $T$ is called a \emph{fork}.

An effective subdivisor $T_{0}$ of an effective divisor $T$ is called a \emph{twig} of $T$ if it is a chain with components $L_{1},\dots, L_{r}$, such that $\beta_{T\redd}(L_{1})=1$ and $\beta_{T\redd}(L_{r})=2$. By convention we assume that $T_{0}$ is ordered so that the tip $L_{1}$ is its first component. A twig is \emph{maximal} if it is maximal in the set of twigs ordered by inclusion of supports.

A $(-1)$-curve $L$ in $T$ is called \emph{superfluous} if it intersects at most two  components of $T\redd-L$, each in one point and transversally. Equivalently, after its contraction the image of $T$ is an snc-divisor. An snc-divisor is \emph{snc-minimal} if it has no superfluous $(-1)$-curves.

We put $d(0)=1$ and for a nonzero reduced divisor $T$ with components $L_{1},\dots, L_{r}$ we put $d(T)=\det[-L_{i}\cdot L_{j}]_{1\leq i,j\leq r}$. Elementary properties of determinants imply the following:
\begin{lem}[discriminants]
\label{lem:discriminants}
Let $(X,T)$ be a smooth pair.
\begin{enumerate}[(a)]
\item Let $(X',T')\to(X,T)$ be a blowup of a point on $T$, where $T'$ is the reduced preimage of $T$. Then $d(T')=d(T)$. In particular, if $T$ contracts to a smooth point then $d(T)=1$.
\item Let $T_{1}$ and $T_{2}$ be reduced snc-divisors on $X$ for which $T_{1}\cdot T_{2}=1$ and let $C_{1}$, $C_{2}$ be the unique components of $T_{1}$ and $T_{2}$ which meet. Then:
\begin{equation*}
d(T_{1} + T_{2}) = d(T_{1})d(T_{2})-d(T_{1}-C_{1})d(T_{2}-C_{2}).
\end{equation*}
\end{enumerate}
\end{lem} 

Following \cite{Tono_nie_bicuspidal}, for two chains $A=[a_{1},\dots a_{r}]$, $B=[b_{1},\dots, b_{s}]$ we put $A*B\de [a_{1},\dots, a_{r-1},a_{r}+b_{1}-1,b_{2},\dots, b_{s}]$. We have $(A*B)*C=A*(B*C)$ for  chains $A,B,C$. If $A=[a_{1},\dots, a_{r}]$ is a chain we define $A^{*}=[(2)_{a_{r}-1}]*\dots * [(2)_{a_{1}-1}]$. It is easy to see that a chain $[A,1,A']$ can be contracted to a $0$-curve if and only if $A'=A^{*}$. Moreover, in the latter case $d(A)=d(A')$ (see \cite[4.8]{Fujita-noncomplete_surfaces}).

\smallskip
Let $\phi\colon X \to X'$ be a birational morphism between smooth projective surfaces. Then we can write $\phi$ as a composition of blowups $\phi=\phi_{z}\circ \dots \circ \phi_{1}$ for some $z\geq 0$. If $L$ is a curve on $X$ then we say that a blow-up does not \emph{touch} $L$ if it is an isomorphism in some neighborhood of $L$. We say that $\phi$ \emph{touches $L$ $n$ times} for some $n\geq 0$ if exactly $n$ of the blowups $\pi_{1},\dots, \pi_{z}$ touch $L$ and each exceptional divisor meets the respective proper transform of $L$ transversally in one point. Note that in this case $(\pi_{*}L)^{2}=L^{2}+n$.

\subsection{\texorpdfstring{$\Q$}{Q}-acyclic surfaces and related conjectures.}\label{subsection:QHP}

If $\bar{E}\subseteq \P^{2}$ is a rational cuspidal curve then by the Poincar\'{e}-Lefschetz duality \cite[8.3]{Bredon_top_and_geom} the surface $S=\P^{2}\setminus \bar{E}$ is \emph{$\Q$-acyclic}, that is, the Betti numbers $b_{i}(\P^{2}\setminus \bar{E})=\dim H_{i}(\P^{2}\setminus \bar{E};\Q)$ vanish for $i>0$. All smooth $\Q$-acyclic surfaces are affine \cite[2.5]{Fujita-noncomplete_surfaces} and rational \cite{GuPrad-rationality_3(smooth_II)}. By the work of many authors $\Q$-acyclic surfaces which are not of log general are classified (see \cite[\S3.4]{Miyan-OpenSurf}). The case of log general type is much more complicated. We need the following result.

\begin{lem}[No affine lines, \cite{MiTs-lines_on_qhp}, \cite{Miyan-OpenSurf} II.3.1]\label{lem:Qhp_has_no_lines}
A smooth $\Q$-acyclic surface of log general type contains no curves isomorphic to $\C^{1}$. In particular, every its smooth completion $(X,D)$ for which $D$ is snc-minimal is almost minimal.
\end{lem}

We note that this result uses the logarithmic version of the Bogomolov-Miyaoka-Yau inequality. Almost minimality means that when passing to the minimal model of $(X,D)$ in the sense of Mori we contract only curves contained in $D$ and its images. Unfortunately, even in case $S$ is of type $\P^2\setminus \bar E$ this fact gives rather weak restrictions on the singularities of $\bar E\subseteq \P^2$. 

In this situation, the first author proposed studying the minimal model of $(X,\frac{1}{2}D)$. In \cite[\S 3]{Palka-minimal_models} he described the creation of the latter and made the Negativity Conjecture \ref{conj} for $\Q$-acyclic surfaces, which is in fact the main motivation for the current article. This conjecture implies the Flenner-Zaidenberg Weak Rigidity Conjecture \cite{FZ-deformations} (equivalent to the vanishing of $h^0(2K_X+D)$; see \cite[2.5(2)]{Palka-minimal_models} for a discussion), and hence the Coolidge-Nagata conjecture. It is better suited for the log MMP approach. For instance, once it is proven, the log MMP for $(X,\frac{1}{2}D)$ implies that some minimal model of $(X,\frac{1}{2}D)$ has a structure of a Log Mori Fiber Space. If the base is a point then the pushforward of $-(K_X+\frac{1}{2}D)$ is ample on it. If the base is a curve then, since the general fiber $f$ is isomorphic to $\P^1$, we get that $f\cdot (2K_X+D)<0$, so $f\cdot D\leq 3$, and hence $f\cdot D=3$, as otherwise $X\setminus D$ would not be of log general type. The latter fibration induces a $\C^{**}$-fibration on $X\setminus D$, where $\C^{**}=\C^1\setminus \{0,1\}$.

\subsection{\texorpdfstring{$\P^{1}$}{P1}-fibrations.}\label{ssec:fibrations}

A \emph{$\P^{1}$-fibration} of a smooth projective surface $X$ is a dominating morphism $p\colon X\to B$ with general (scheme-theoretic) fibers isomorphic to $\P^{1}$. A fiber non-isomorphic to $\P^{1}$ is called a \emph{degenerate fiber}. All fibers of $p$ are numerically equivalent and have self-intersection $0$. Conversely (see \cite[V.4.3]{BHPV_complex_surfaces}), for any $0$-curve $F$ on $X$ or, more generally, for any effective divisor $F$ with $F^{2}=0$ and $p_{a}(F)=0$, there is a $\P^{1}$-fibration of $X$ with $F$ as one of the fibers.

For a given $\P^{1}$-fibration with general fiber $f$, we say that an (irreducible) curve $C$ is \emph{vertical} (resp. \emph{horizontal}) if $C\cdot f=0$ (resp. $C\cdot f\neq 0$). Every divisor $T$ can be uniquely decomposed as $T=T_{v}+T_{h}$, where all components of $T_{v}$ are vertical and all components of $T_{h}$ are horizontal. A horizontal curve $C$ with $C\cdot f=n$ is called an \emph{$n$-section}. If $C$ is vertical and irreducible then by $\mu(C)$ we denote the multiplicity of $C$ in the fiber containing $C$; such a fiber will be usually denoted by $F_{C}$. 

Every degenerate fiber of a $\P^{1}$-fibration can be contracted to a $0$-curve by iterated contractions of $(-1)$-curves. Therefore, by induction one easily gets the following (see \cite[4.14]{Fujita-noncomplete_surfaces}):
\begin{lem}[Degenerate fibers]\label{lem:singular_P1-fibers} Let $F$ be a degenerate fiber of a $\P^{1}$-fibration of a smooth projective surface. Then $F$ is a rational tree and its $(-1)$-curves are non-branching in $F\redd$. Furthermore:
\begin{enumerate}[(a)]
\item If a $(-1)$-curve $L$ is a component of $F$ and $\mu(L)=1$ then $\beta_{F\redd}(L)=1$ and $F$ contains another $(-1)$-curve.
\item 
Assume that $F$ has a unique $(-1)$-curve $L$ and let $U,V$ be the connected components of $F\redd -L$. Let $U_{0}$, $V_{0}$ be some fixed components of $U$ and $V+L$ respectively. Then:
\begin{enumerate}[($\textrm{b}_1$)]
\item If $\mu(U_{0})=1$ and $\mu(V_{0})=1$ then $U_{0}, V_{0}$ are tips of $F$ and $F$ is a chain $[U,1,U^{*}]$. Moreover, $d(U)=d(U^{*})$.
\item If $\mu(U_{0})=1$ and $\mu(V_{0})=2$ then either $F$ is a chain like above, or $F$ is a fork with maximal twigs $U_{0}=[2]$, $U_{0}'=[2]$ and $T$, such that $V_{0}$ is the first tip of $T$ and $T$ contracts to a smooth point (see Fig.\ \ref{fig:fork}).
\end{enumerate}
\item Assume that $F$ contains only two $(-1)$-curves $L_{1}$ and $L_{2}$, $\mu(L_{1})=1$ and both connected components of $F\redd - (L_{1}+L_{2})$ contain components of multiplicity $1$. Then $F$ is a chain of type $[1,(2)_{\alpha}*U,1,U^{*}]$ for some $\alpha \geq 1$.

\begin{figure}[h]
\centering
\includegraphics[scale=0.45]{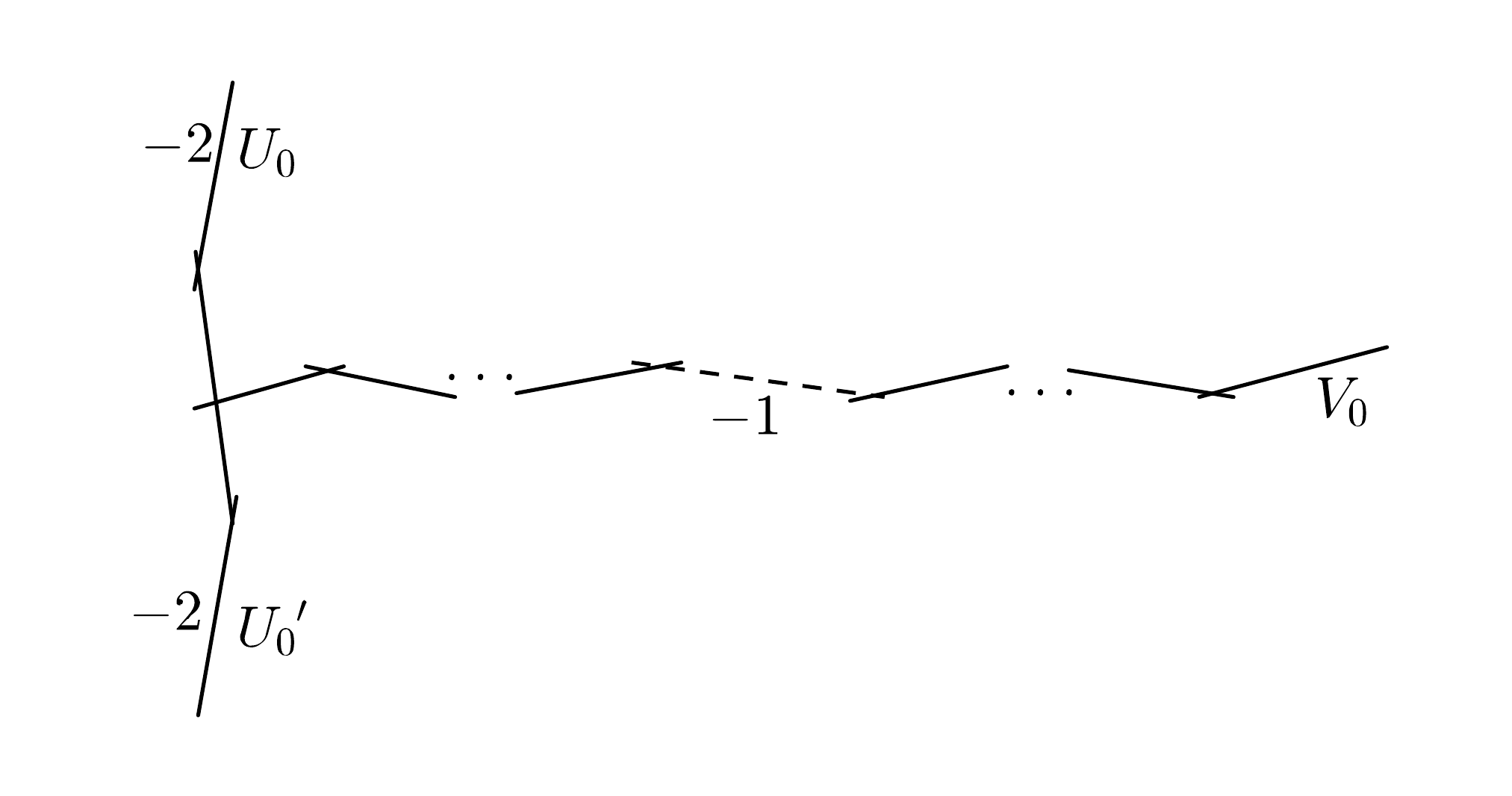}
\caption{A \emph{special fork} from Lemma \ref{lem:singular_P1-fibers}$(\mathrm{b_{2}})$. The tips $U_{0},U_{0}'$ are the only components of multiplicity $1$; $\mu(V_{0})=2$.} \label{fig:fork}
\end{figure}

\end{enumerate}
\end{lem}
A fiber like in Lemma \ref{lem:singular_P1-fibers}($\mathrm{b_{2}}$) will be called a \emph{special fork}. 

\smallskip Assume that $V$ is a smooth quasi-projective surface. A morphism to a smooth curve $p\colon V\to B$ is called a \emph{$\C^{(n*)}$-fibration} if there exists a nonempty Zariski-open subset $U\subseteq B$, such that for every $b\in U$ there exist $n$ points $p_{1},\dots, p_{n}\in \C^{1}$, such that the scheme-theoretic fiber $p^{*}(b)$ is isomorphic to $\C^{1}\setminus \{p_{1},\dots, p_{n}\}$. Call $U_{0}$ the maximal subset of $B$ with the above property. A fiber $F_{b}=p^{-1}(b)$ for $b\not \in U_{0}$ is called a \emph{degenerate fiber} of $p$.

For a $\C^{(n*)}$-fibration $p\colon V\to B$ there exist smooth completions $(X,D)$ of $V$ and $\bar{B}$ of $B$, such that $p$ extends to a $\P^{1}$-fibration $\bar{p}\colon X\to \bar{B}$ with $F\cdot D=n$ for every fiber $F$ of $\bar{p}$. Any triple $(X,D,\bar{p})$ as above will be called a \emph{completion} of a fibration $p\colon V\to B$. We will use the following:

\begin{notation}\label{not:fibrations_h_and_nu}
Let $(X,D)$ be a smooth pair and $\bar{p}\colon X\to \bar{B}$ a $\P^{1}$-fibration. Then $h=\#D_{h}$ is the number of horizontal components of $D$, $\nu$ is the number of fibers contained in $D$ and $\sigma(F_{b})$ for $b\in \bar{B}$ denotes the number of components of $F_{b}=\bar{p}^{-1}(b)$ not contained in $D$.
\end{notation}

\begin{lem}[{\cite[9.2]{Suzuki-chi_for_fibrations}}, {\cite[3.2]{Zaid_isotrivial_curves_on_surfaces}}, cf. {\cite[III.1.8]{Miyan-OpenSurf}}]\label{lem:Suzuki}
Let $V$ be a smooth affine surface and let $p\colon V\to B$ be a $\C^{(n*)}$-fibration with general fiber $f$. Then:
\begin{enumerate}[(a)]
\item For every degenerate fiber $F$ of $p$ we have $\chi(F)\geq \chi(f)$ and if the equality holds then $F_{\mathrm{red}}\cong f \cong \C^{1}$ or $F_{\mathrm{red}}\cong f \cong \C^{*}$.
\item
\begin{equation*}
\chi(V)= \chi(B)\cdot \chi(f)+\sum_{b\in B}(\chi(F_{b})-\chi(f)).
\end{equation*}
\end{enumerate}
\end{lem}

\begin{lem}[{\cite[4.16]{Fujita-noncomplete_surfaces}}, cf. {\cite[Lemma 2]{Palka-AMS_LZ}}]\label{lem:fibrations-Sigma-chi}
Let $X,D,\bar{p}$ be as in Notation \ref{not:fibrations_h_and_nu}. Then:
\begin{equation*}h+\nu+\rho(X)=\#D+2+\sum_{b\in B}(\sigma(F_{b})-1).\end{equation*}
In particular, if the components of $D$ generate $\NS_{\Q}(X)$ freely, e.g.\ when $X\setminus D$ is $\Q$-acyclic, then
\begin{equation*}h+\nu-2=\sum_{b\in B}(\sigma(F_{b})-1)\geq 0.\end{equation*}
\end{lem}

In order to apply Lemma \ref{lem:Suzuki} to $\C^{**}$-fibrations we need the following:
\begin{lem}[Curves with $\chi\geq 0$] \label{lem:when_chi(fiber)>=0}
Let $F$ be a, possibly reducible, affine curve such that $\chi(F)\geq 0$. Assume that $F$ contains no component $C$ homeomorphic (in the complex topology) to $\C^{1}$, such that $C\cap (F-C)$ consists of at most one point. Then $F$ is a disjoint union of curves homeomorphic to $\C^{*}$.
\end{lem}
\begin{proof}
We argue by induction on $\#F$. The case $\#F=1$ is obvious. Let $C$ be a component of $F$ intersecting $F-C$ in $d$ points. We have $C\cong \bar{C} \setminus \{r \text{\ points} \}$ for some $r\geq 1$, where $\bar{C}$ is a projective curve. Let $C^{\nu}\to \bar C$ be its normalization. We have $$\chi(C)=\chi(\bar{C})-r\leq\chi(C^\nu)-r=2-2g(C^\nu)-r\leq 1,$$ hence $\chi(C)\leq 1$. Moreover, the inequality is an equality if and only if $\chi(\bar C)=\chi(C^\nu)$, $g(C^\nu)=0$ and $r=1$, that is, $C$ is homeomorphic to $\C^{1}$, so in this case $\chi(C)<d$ by assumption. It follows that in every case $$0\leq \chi(F)=\chi(F-C)+\chi(C)-d\leq \chi(F-C).$$ Thus by the inductive hypothesis $F-C$ is a disjoint union of curves homeomorphic to $\C^{*}$. Then $\chi(F-C)=0$, so $d=\chi(C)\leq 1$ and the argument above implies that the inequality is strict. Thus $d=0$, $C$ is homeomorphic to $\C^{*}$ and is disjoint from $F-C$, as claimed.
\end{proof}

\begin{wn}[$\C^{**}$-fibrations]\label{cor:Cstst-fibers} Let $p\colon V\to B$ be a $\C^{**}$-fibration of a smooth $\Q$-acyclic surface of log general type. Then, with the notation as in Notation \ref{not:fibrations_h_and_nu}:
\begin{enumerate}[(a)]
\item $p$ has $3-\nu$ degenerate fibers, each of them is disjoint union of curves isomorphic to $\C^{*}$,
\item $\sum_{F}\sigma(F)=h+1$, where the sum runs over all degenerate fibers of $p$,
\item $\nu\leq 1$.
\end{enumerate}\end{wn}
\begin{proof}
(a) Let $(X,D,\bar{p})$ be a minimal completion of $p$. Since $X$ is rational, $B$ is isomorphic to an open subset of $\P^{1}$ with $\#(\P^{1}\setminus B)=\nu$. The closure of a  fiber $F$ in $X$ is a rational tree. By Lemma \ref{lem:Qhp_has_no_lines} $F$ has no component isomorphic to $\C^{1}$ and by Lemma \ref{lem:Suzuki}(a), $\chi(F)\geq \chi(\C^{**})+1=0$. Then Lemma \ref{lem:when_chi(fiber)>=0} shows that every degenerate fiber is a disjoint union of $\C^*$'s. By Lemma \ref{lem:Suzuki}(b) $1=(2-\nu)\cdot (-1)+\sum_{F}(\chi(F)+1)$, so $\sum_{F}1=3-\nu$.

(b) Because the components of $D$ generate $\NS_{\Q}(X)$ freely, from Lemma \ref{lem:fibrations-Sigma-chi} and part (a) we have $h+1=3-\nu+\sum_{F}(\sigma(F)-1)=\sum_{F}\sigma(F)$.

(c) If $F_{1}\neq F_{2}$ are fibers contained in $D$ then the linear equivalence $F_{1}\sim F_{2}$ implies that the components of $D$ are not linearly independent in $\NS_{\Q}(X)$, hence the natural homomorphism $H_2(D,\Q)\to H_2(X,\Q)$ is not surjective, so by the Poincaré-Lefschetz duality $H^2(X\setminus D)=H_2(X,D)\neq 0$; a contradiction.
\end{proof}

Note that if $(X,D,\bar p)$ is a minimal completion of $p$ then the closure of a non-degenerate fiber of $p$ is a non-degenerate fiber of $\bar p$, which follows easily from the fact that none of the components of $D$ contained in the completed fiber can be a $(-1)$-curve. On the other hand, a completion of a degenerate fiber of $p$ may very well be a non-degenerate fiber of $\bar p$.

\subsection{Log resolutions and Hamburger-Noether pairs.} \label{sec:cusps}

The are many ways of describing the geometry of cusps and exceptional divisors of their resolutions. We use the \emph{Hamburger-Noether pairs} (\emph{HN} for short), because on one hand they are a compact way of keeping track of sums of multiplicities and squares of multiplicities of singular points on successive proper transforms of the cusp and on the other hand they are directly related to weights of dual graphs (see Lemma \ref{lem:HN_for_chains}). For more details we refer the reader to \cite[Appendix]{KR-C*_actions_on_C3} and \cite{Russell_HN_pairs}. The relation with the  multiplicity sequence is explained in Lemma \ref{lem:Hn_gives_multiplicities}. Relations with other numerical characteristics (Puiseux pairs, Zariski pairs) are explained in Section \ref{sec:appendix}. 

Let $(\chi,q)$ be an analytically irreducible germ (a cusp) of a planar curve singular at the point $q$. Let $Q$ be the exceptional divisor of its minimal log resolution and let $C$ be the unique $(-1)$-curve in $Q$ intersecting the proper transform of $\chi$. Since $Q$ is created from a smooth point by a sequence of blowups, its structure is easy to understand. For instance, every component of $Q$ meets at most three others and if it meets three then it meets a maximal twig of $Q$ (see Fig.\ \ref{fig:treeQ}). Since the log resolution is minimal, $C$ is not a tip of $Q$.
\begin{figure}[h]
\setlength{\unitlength}{0.5cm}
\begin{picture}(16,5)
\put(0,1){\circle{0.2}}
\put(0.2,1){\line(1,0){0.6}}
\put(1.0,0.9){$\dots$}
\put(2.0,1){\line(1,0){0.6}}
\put(2.8,1){\circle{0.2}}
\put(3.8,1.2){\line(0,1){0.6}}
\put(3.8,2){\circle{0.2}}
\put(3.8,2.2){\line(0,1){0.6}}
\put(3.7,3.0){$\vdots$}
\put(3.8,3.6){\line(0,1){0.6}}
\put(3.8,4.4){\circle{0.2}}
\put(3.0,1){\line(1,0){0.6}}
\put(3.8,1){\circle{0.2}}
\put(4,1){\line(1,0){0.6}}
\put(4.8,0.9){$\dots$}
\put(5.8,1){\line(1,0){0.6}}
\put(6.6,1){\circle{0.2}}
\put(6.8,1){\line(1,0){0.6}}
\put(7.6,1){\circle{0.2}}
\put(7.8,1){\line(1,0){0.6}}
\put(8.6,0.9){$\dots\dots$}
\put(7.6,1.2){\line(0,1){0.6}}
\put(7.6,2){\circle{0.2}}
\put(7.6,2.2){\line(0,1){0.6}}
\put(7.5,3.0){$\vdots$}
\put(7.6,3.6){\line(0,1){0.6}}
\put(7.6,4.4){\circle{0.2}}
\put(10.5,1){\line(1,0){0.6}}
\put(11.3,1){\circle{0.2}}
\put(11.5,1){\line(1,0){0.6}}
\put(12.3,0.9){$\dots$}
\put(13.3,1){\line(1,0){0.6}}
\put(14.1,1){\circle{0.2}}
\put(11.3,1.2){\line(0,1){0.6}}
\put(11.3,2){\circle{0.2}}
\put(11.3,2.2){\line(0,1){0.6}}
\put(11.2,3.0){$\vdots$}
\put(11.3,3.6){\line(0,1){0.6}}
\put(11.3,4.4){\circle{0.2}}
\put(14.3,1){\line(1,0){0.6}}
\put(14.95,0.7){$*_{\ C}$}
\put(15.1,1.2){\line(0,1){0.6}}
\put(15.1,2){\circle{0.2}}
\put(15.1,2.2){\line(0,1){0.6}}
\put(15,3.0){$\vdots$}
\put(15.1,3.6){\line(0,1){0.6}}
\put(15.1,4.4){\circle{0.2}}
\end{picture}
\caption{The graph of $Q$. The star represents the unique $(-1)$-curve $C$.}
\label{fig:treeQ}
\end{figure}
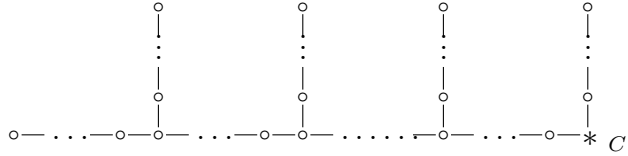
Letting $Z$ vary through smooth germs at $q$, we note that the intersection numbers $(Z\cdot \chi)_{q}$ at $q$ are bounded from above. Any smooth germ $Z$ for which the maximum is reached is said to be \emph{maximally tangent} to $\chi$ at $q$. We denote the multiplicity of the singular point $q\in \chi$ by $\mu(q)$.

For the beginning of the induction let $(Z_{1},q_1)$ be a smooth germ maximally tangent to $(\chi_1,q_1):=(\chi,q)$. For the inductive step assume we are given a triple $(\chi_j, Z_j,q_j)$ such that $(\chi_j,q_j)$ and $(Z_j,q_j)$ are respectively an analytically irreducible and a smooth germ of planar curves. If $(\chi_j,q_j)$ is smooth and meets $Z_{j}$ transversally then we stop, otherwise we put $c_j=(Z_j\cdot \chi_j)_{q_j}$ and we choose a smooth germ $(Y_j,q_j)$ transversal to $Z_j$ at $q_j$ and such that 
\begin{equation}\label{HN:convention}
\mbox{ if } c_j=\mu(q_j) \mbox{  then  } Y_j \mbox{  is maximally tangent to  } (\chi_j,q_j).
\end{equation}
We put $p_j=(Y_j\cdot \chi_j)_{q_j}$. We now blow up over $q_j$ until the proper transform $(\chi_{j+1},q_{j+1})$ of $(\chi_{j},q_j)$ intersects the total transform of $Z_{j}+Y_{j}$ not in a node and we denote by $Z_{j+1}$ the last exceptional curve. This way we obtain a sequence $$\binom{c_{1}}{p_{1}}\ldots \binom{c_{h}}{p_{h}}$$ of the \emph{standard HN-pairs} of the cusp $(\chi,q)$. By the Euclidean algorithm and by the transversality of $Z_j, Y_j$ we get for $j\geq 1$ 
\begin{equation}\label{eq:HN_c_j+1}c_{j+1}=\gcd(c_{j},p_{j}) \text{\ \ \ and\ \ } \min\{c_j,p_j\}=\mu(q_j)
\end{equation} 
respectively. Our choices of $Z_1$ and $Y_j$ give respectively 
\begin{equation}\label{eq:HN_our_convention} p_1\nmid c_1,\  p_1\leq c_1 \text{\ \ and\ \ } c_j\neq p_j.
\end{equation} 
With these choices the sequence of HN-pairs is unique and $Q$ has exactly $h-1$ branching components. One can also show that with the convention \eqref{HN:convention} $(\chi,q)$ has exactly $h$ Puiseux pairs, see Lemma \ref{lem:HN_vs_Puiseux}.

\begin{prz}[Semi-ordinary cusps]\label{ex:ordinary_cusp}A cusp of HN-type $\binom{2m+1}{2}$ for $m\geq 1$ is called \emph{semi-ordinary}. It has multiplicity $2$ and is locally analytically isomorphic to the singular point of $x^{2}=y^{2m+1}$ at $0\in \Spec \C[x,y]$. The exceptional divisor of its minimal log resolution is a chain $[3,1,(2)_{m-1}]$. It is often called a cusp of \emph{type $A_{2m}$}. A semi-ordinary cusp for $m=1$ (type $A_2$) is called \emph{ordinary}.
\end{prz}

\begin{rem}[Non-standard HN pairs]\label{HN:discussion}
It is possible, and in some situations more convenient to allow more general choices of $Z_1$ or $Y_j$ for $j\geq 1$ which do not satisfy \eqref{eq:HN_our_convention} (running inductive arguments more smoothly is one of the reasons). For instance, instead of \eqref{eq:HN_our_convention} we could ask $Y_j$ to satisfy $p_j\leq c_j$ (as it is done in \cite{KoPaRa-SporadicCstar1}), which results in replacing all pairs $\binom{c_j}{p_j}$ where $p_j>c_j$ with sequences $\binom{c_j}{c_j}_{r_j}\binom{c_j}{p_j-r_j c_j}$, where $r_j$ is the integer part of $p_j/c_j$ and $\binom{c}{c}_r$ means $\binom{c}{c}$ repeated $r$ times. Additionally, we could allow a more general choice of $Z_1$, which will relax the condition $p_1\nmid c_1$. Converting the so-computed non-standard HN-type to the standard one results with replacing the initial part \begin{equation}\label{eq:HN-equivalence}\binom{c_1}{p_1}\binom{p_1}{p_1}_r\binom{p_1}{p_2}\ p_2\nmid p_1 \text{\ \ with\ \ } \binom{c_1+rp_1+p_2}{p_1}.\end{equation}
The reader can check that this does not affect the basic equations listed in Lemma \ref{lem:HN-equations}. In particular, it is sometimes convenient to replace the initial germ $Z_{1}$ with the germ of the line tangent to $\bar{E}\subseteq \P^{2}$ at the cusp. We we will use such \emph{tangent HN-types} in future articles. Note however, that although they carry more information than the local topology of the singularity, they makes sense only for planar projective curves.
\end{rem}

The multiplicities of the preimage of $q$ on the proper transforms of $\chi$ under consecutive blowups form the \emph{multiplicity sequence} of $q\in \chi$. Since we rely on HN-pairs, we need a recipe to compute one from another. For examples see Table \ref{table:fibrations}.

\begin{lem}[HN-pairs vs the multiplicity sequence]\label{lem:Hn_gives_multiplicities}
Let $(\chi,q)$ be an analytically irreducible germ of a planar curve.
\begin{enumerate}[(a)]
\item Let $\binom{c}{p}=\binom{c_j}{p_j}$ be some HN-pair which is an element of some (not necessarily standard) HN-sequence of $(\chi,q)$. Then the part of the multiplicity sequence of $(\chi,q)$ corresponding to this HN-pair is
\begin{equation*}
((\min\{p,c\})_{u_1},(m_2)_{u_2}, \dots, (m_v)_{u_v}),
\end{equation*} 
where, putting $m_0=\max\{p,c\}$, the numbers $u_{k}$ and $m_{k+1}$ for $k\geq 1$ are respectively the quotient and remainder of dividing $m_{k-1}$ by $m_k$. The index $v\geq 1$ is the first $k$ such that $m_k|m_{k-1}$.

\item Let $((\mu_{1})_{u_1},\ldots,(\mu_{g})_{u_g})$ with $\mu_1>\ldots>\mu_g=1$ be the multiplicity sequence of $(\chi,q)$. Then $(\chi,q)$ is of standard HN-type
$$\binom{u_{1}\mu_{1}+\mu_{2}}{\mu_{1}}\binom{c_2}{p_2}\ldots \binom{c_{h}}{p_{h}},$$ where $c_{j},p_{j}$ for $j\geq 2$ are defined inductively as follows: $c_j=\gcd (c_{j-1},p_{j-1})$; if $c_j=1$ then we put $h=j-1$ and we stop. Otherwise we put $p_{j}=\mu_{k_{j}+1}+u_{k_{j}}\mu_{k_{j}}-\mu_{k_{j}-1}$, where $k_{j}$ is the unique integer such that $\mu_{k_{j}}=c_j$.
\end{enumerate}
\end{lem}

\begin{proof}
(a) Blow up once on $q$. If $c=p$ this is the end of the sequence of blowups described by $\binom{c}{p}$, so the claim is clear. We may assume, say, $p<c$. The remaining sequence is described by the HN-pair $\binom{c-p}{p}$, which is a part of some sequence of HN-pairs for the proper transform of $\chi$. The result follows by the Euclidean algorithm.

(b) Let $\binom{c_{1}}{p_{1}}\dots\binom{c_{h}}{p_{h}}$ be the standard HN-type of $(\chi,q)$. By assumption $\mu_1=p_1<c_1$ and $p_1\nmid c_1$. Blow up $u_{1}$ times over $q$. The proper transform of $(\chi,q)$ has (not necessarily standard) HN-type $\binom{p_{1}}{c_{1}-u_{1}p_{1}}\binom{c_{2}}{p_{2}}\dots \binom{c_{h}}{p_{h}}$, so $\mu_2=\min\{p_1,c_1-u_1p_1\}=c_1-u_1p_1$.

The formula for $c_j$, $j\geq 2$ holds by \eqref{eq:HN_c_j+1}. Fix $j\geq 2$. We perform the sequence of blowups described by $(\binom{c_j'}{p_j'})_{j'<j-1}$ and part of $\binom{c_{j-1}}{p_{j-1}}$ until the multiplicity of the cusp drops to $\mu_k=c_j$. We have $\mu_{k}|\mu_{k-1}$ and the intersection of the proper transform of $\chi$ with the last exceptional curve equals $\mu_{k-1}$. Then we do the $\mu_{k-1}/\mu_{k}$ remaining blowups, each time the multiplicity of the center is $\mu_k$, until the proper transform $\chi_{j}$ of $\chi$ does not meet the total exceptional divisor in a node. At this point, by the definition of HN pairs, we have a smooth germ $(Y_j,q_j)$ transversal to the last exceptional component $Z_j$,  and $c_j=\mu_k$, $p_j$ are respectively the intersections of $\chi_{j}$ with $Z_j$ and with $Y_j$. If $c_j>p_j$ then $\mu(q_j)=p_j<c_j=\mu_{k}$, so $p_j=\mu_{k+1}$ and hence $\mu_{k-1}/\mu_{k}=u_k$, so we are done. We may therefore assume $p_j>c_j$. After $u_k-\mu_{k-1}/\mu_{k}$ further blowups the multiplicity of the cusp drops to $\mu_{k+1}$. Because it drops, the intersection of the proper transforms of $\chi_{j}$ and $Y_{j}$, which is $\mu:=p_{j}-(u_{k}-\mu_{k-1}/\mu_{k})\mu_{k}$, is smaller than $c_j$. Since $Y_j$ is maximally tangent to $\chi_j$, we have $\mu>0$. Indeed, otherwise the image of any germ transversal to the last exceptional component and meeting the proper transform of the cusp would be a smooth germ with a bigger intersection with $\chi_j$ than $Y_j$. We infer that the multiplicity of the cusp is $\min\{c_j,\mu\}=\mu$, so $\mu=\mu_{k+1}$ and we are done. 
\end{proof}

We have the following useful description of the graph of $Q$ in terms of HN-pairs:

\begin{lem}[Dual graphs from HN-pairs, {\cite[3.5]{Palka-Coolidge_Nagata1}}] \label{lem:HN_for_chains}
Let $Q$ be the chain created by the characteristic pair $\binom{\alpha c}{\alpha p}$ with $\gcd(c,p)=1$ and let $L$ be the unique $(-1)$-curve of $Q$. Write $Q-L=A+B$, where $A$ and $B$ are disjoint and connected, with $d(A)\geq d(B)$. Put $B=0$ if $p=1$. Let $A_{L},B_{L}$ be the tips of $Q$ which are components of $A$ and $B$, respectively. Write $c=q\cdot p+r$ for some integers $q>0$ and $0\leq r<p$. Then:
\begin{equation*}
 d(A)=c, \quad d(B)=p, \quad d(A-A_{L})=c-p, \quad d(B-B_{L})=p-r.
\end{equation*}
\end{lem}

\begin{lem}[Equations for HN-pairs, {\cite[3.3]{Palka-Coolidge_Nagata1}}]\label{lem:HN-equations}
Let $\bar{E}\subseteq \P^{2}$ be a rational cuspidal curve with cusps $q_{1},\dots, q_{c}\in \bar{E}$, where $q_{j}$ is described by a sequence of HN-pairs $(\binom{c_{j}^{(k)}}{p_{j}^{(k)}})_{k=1}^{h_{j}}$. Let $(X,D)\to(\P^{2},\bar{E})$ be the minimal log resolution and let $E$ be the proper transform of $\bar{E}$ on $X$. Put $\gamma=-E^{2}$, $M(q_{j})=c_{j}^{(1)}+\sum_{k=1}^{h_{j}}p_{j}^{(k)}-1$ and $I(q_{j})=\sum_{k=1}^{h_{j}} c_{j}^{(k)}p_{j}^{(k)}$. 
Then \begin{enumerate}[(a)]
	\item $ \gamma-2+3\deg\bar{E}=\displaystyle\sum_{j=1}^{c} M(q_{j})$,
	\item $ \gamma+(\deg\bar{E})^{2}=\displaystyle\sum_{j=1}^{c} I(q_{j})$
	\item $ (\deg\bar{E}-1)(\deg\bar{E}-2)=\displaystyle\sum_{i=1}^{c}(I(q_{j})-M(q_{j}))$.
\end{enumerate}
Note that the last formula is the genus formula written in terms of HN-pairs.
\end{lem}

\subsection{Some upper bounds on $E^2$.}

We have the following criterion:
\begin{lem}\label{lem:kappa<=1}For a rational cuspidal curve $\bar{E}\subseteq \P^{2}$ the following conditions are equivalent:\begin{enumerate}[(a)]
\item $\kappa(\P^{2}\setminus \bar{E})\neq 2$,
\item $\P^{2}\setminus \bar{E}$ is $\C^{1}$- or $\C^{*}$-fibered,
\item $\P^{2}\setminus \bar{E}$ contains a curve isomorphic to $\C^{1}$.
\end{enumerate}
\end{lem}

\begin{proof}
The implication (a)$\implies$(b) is obtained in \cite[2.6]{Palka-minimal_models} as a corollary from the existing structure theorems for affine surfaces which are not of log general type (an additional argument is required in case $\kappa=0$). The implication (b)$\implies$(c) follows from Lemma \ref{lem:Suzuki}. The implication (c)$\implies$(a), based on the logarithmic BMY inequality, follows from Lemma \ref{lem:Qhp_has_no_lines}.
\end{proof}
	
In fact there is a detailed classification of rational cuspidal curves $\bar E\subseteq \P^2$ with $\kappa(\P^2\setminus\bar E)\neq 2$. The case $\kappa=0$ was shown to be impossible in \cite{Tsunoda_cusps_complement_kod_0}. The case $\kappa=-\infty$ was treated in \cite{MiySu-curves_with_negative_kod_of_complement} and \cite{Kashiwara}. The case $\kappa=1$ was treated in \cite{Tono_1cusp_with_kod_1} and \cite{Tono_doctoral_thesis}.  We note that one could recover the classification in a more unified manner starting from the equivalence (a)$\iff$(b) above.

\begin{notation}\label{not:cusps}
Let $\bar{E}\subseteq \P^{2}$ be a rational cuspidal curve with $c$ cusps $q_{1},\dots, q_{c}$. By the above discussion we may and shall assume that $\kappa(\P^{2}\setminus\bar{E})=2$. We denote by $\pi\colon (X,D)\to(\P^{2},\bar{E})$ the minimal log resolution and by $E$ the proper transform of $\bar{E}$ on $X$. Then $Q_{j}\de \pi^{-1}(q_{j})\redd$ contains a unique $(-1)$-curve intersecting $E$, which we denote by $C_{j}$.
\end{notation}

\begin{lem}[Upper bounds on $E^2$, Tono]\label{lem:Tono_E2} With the notation as in \ref{not:cusps} we have the following bounds:
\begin{enumerate}[(a)]
\item (cf. \cite{Tono-on_Orevkov_curves}) If $c=1$ then $E^{2}\leq -2$.
\item (\cite[Thm.\ 1]{Tono_nie_bicuspidal}) If $c=2$ then $E^{2}\leq -1$. 
\item (\cite[Thm.\ 1]{Tono-cusps_self-intersections}) In general, $E^{2}\leq 7-3c$.
\end{enumerate}
\end{lem}
\begin{proof}(a),(b) Assume the contrary. Blow up over $C_{1}\cap E$ until the proper transform $E'$ of $E$ has self-intersection $-1$ if $c=1$ and $0$ if $c=2$. Call $C_{1}'$ the last exceptional curve of the blow-up, or put $C_{1}'=C_{1}$ if no blow-ups were needed. If $c=1$ then $E'+C_{1}'$ induces a $\P^{1}$-fibration which restricts to a $\C^{1}$- or $\C^{*}$-fibration on $\P^{2}\setminus \bar{E}$. Similarly, if $c=2$ then so does $E'$. Then $\kappa(\P^{2}\setminus \bar{E})\leq 1$ by Iitaka's Easy Addition Theorem; a contradiction. 

(c) Let $Q$ be a divisor on a smooth projective surface $Y$ and let $\rho\colon Y'\to Y$ be a blowup at a point of multiplicity $\mu$ on $Q$. Put $C=\Exc \rho$ and $Q'=(\rho^{-1}Q)\redd$. Then:
\begin{equation}\label{eq:K(K+Q)_change}
K_{Y'}\cdot (K_{Y'}+Q')=(\rho^{*}K_{Y}+C)\cdot (\rho^{*}K_{Y}+C+\rho^{*}Q-(\mu-1)C)=K_{Y}\cdot (K_{Y}+Q)+\mu-2.
\end{equation}
Recall that the blowups with $\mu=1$ and $\mu=2$ are called respectively outer and inner (for $Q$). Let $(\binom{c_{j}^{(k)}}{p_{j}^{(k)}})_{k=1}^{h_j}$ be the standard sequence of HN-pairs for the cusp $q_{j}\in \bar{E}$ and let $\floor{x}$ denote the biggest integer not greater than $x$. Blowups constituting the  log resolution of $q_j\in\bar E$ other than the first one are either inner or outer and there are exactly $r_j:=\floor{c_j^{(1)}/p_j^{(1)}}+ \sum_{k=2}^{h_j}(1+\floor{p_j^{(k)}/c_j^{(k)}})$ outer ones among them. Since $r_j\geq h_j$, by \eqref{eq:K(K+Q)_change} we get:
\begin{equation*}
\begin{split}
K_{X} \cdot (K_{X}+D)&=K_{X}\cdot( K_{X}+\sum_{j=1}^{c} Q_{j})+K_{X}\cdot E=K_{\P^{2}}^{2}-\sum_{j=1}^{c}(2+r_j)-2-E^{2}\leq \\
&\leq 7-\sum_{j=1}^{c}(2+h_{j})-E^{2} \leq 7-3c -E^{2}.
\end{split}
\end{equation*}
Vanishing theorems for cohomology imply that (cf.\ \cite[2.7]{Tono-cusps_self-intersections}) $$K_{X}\cdot (K_{X}+D)=\chi(\O_{X}(2K_{X}+D))=h^{0}(2K_{X}+D)\geq 0,$$ hence the above inequality gives $E^{2}\leq 7-3c$, as claimed.
\end{proof}

For the boundary values of $-E^{2}$ there is already a classification available.

\begin{lem}[Cases $E^2=-1, -2$, Tono]\label{lem:Tono_E2=-1,-2} With the notation as in \ref{not:cusps} we have the following bounds:
\begin{enumerate}[(a)]
 \item  If $\bar E$ has only one cusp and $E^{2}=-2$ then $\bar{E}\subseteq \P^2$ has HN-type $\ORa$ or $\ORb$.
 \item If $\bar E$ has exactly two cusps and $E^{2}=-1$ then $\bar{E}$ has HN-type $\cA$ - $\cD$ with $\gamma=1$.
\end{enumerate}
\end{lem}

\begin{proof}
(a) By \cite[Thm.\ 1]{Tono-on_Orevkov_curves} $\bar E\subset \P^2$ is projectively equivalent to one of the Orevkov curves with multiplicity sequences $(F_{4k},((F_{4l})_5,F_{4l}-F_{4l-4})_{l=k}^{1})$ or $(2F_k,((2F_{4l})_5,2F_{4l}-2F_{4l-4})_{l=k}^{1})$. Using Lemma \ref{lem:Hn_gives_multiplicities} we compute that their HN-type is $\ORa$ and $\ORb$, respectively.

(b) The possible singularity types of $\bar{E}\subseteq \P^{2}$ are classified in \cite[Thm.\ 2]{Tono_nie_bicuspidal} in terms of multiplicity sequences. Using Lemma \ref{lem:Hn_gives_multiplicities} one can check that the singularity types 1, 2 in loc.\ cit.\ are respectively $\cD(1,b,a)$ and $\cC(1,b,a)$. Similarly, the types 3, 4  in loc.\ cit.\ are respectively
$\cA(1,b,a+1)$, $\cB(1,b,a+1)$.
\end{proof}

\section{Possible types of cusps}\label{sec:classification_of_sing}
In this section we prove the following proposition. We use the notation of Section \ref{ssec:fibrations}.

\begin{prop}\label{prop:possible_cusp_types} Assume that $\bar{E}\subseteq \P^{2}$ is a rational cuspidal curve such that $\P^{2}\setminus \bar{E}$ is $\C^{**}$-fibered and of log general type. Then $\bar{E}$ is of one of the HN-types listed in Theorem \ref{thm:possible_HN-types}.\end{prop}

Note that if the HN-pairs are known then $E^{2}$ and $\deg \bar E$ are determined by Lemma \ref{lem:HN-equations}.

\subsection{Reduction to fibrations with no base points.}

The example below shows that a $\C^{**}$-fibration of $\P^{2}\setminus \bar{E}$ can have base points on the minimal log resolution $(X,D)\to (\P^{2},\bar{E})$. 

\begin{prz}[A $\C^{**}$-fibration with a base point.]\label{ex:Cstst_fibr_with_base_points}
Let $C_{4}\subseteq \P^{2}$ be the Orevkov curve of degree $F_{6}=8$ (type $\ORa(1)$ in Theorem \ref{thm:possible_HN-types}). The graph of $D$ is shown in Figure \ref{fig:exprzedl}. Denote by $A_1$ and $B_1$ the long and short maximal twigs of $D$ contained in $D-E$. Denote by $A_1'$, $A_1''$ respectively the first and the last tip of $A_1$ and by $B_1'$, $B_1''$ the first and the last tip of $B_1$. Assume that there exists a $(-1)$-curve $L_{0}\not \subseteq D$ with $D\cdot L_{0}=2$ meeting $D$ in $A_1'$ and in $A_1''$. Call $T$ the $(-2)$-chain of type $[(2)_{4}]$ in $D$ containing $A_1'$. Let $\psi$ be the blow up over $A_1''\cap C_{1}$. Consider a $\P^{1}$-fibration induced by the proper transform of $T+5L_{0}+A_1''$: it restricts to a $\C^{**}$-fibration of $X\setminus D$ which has base points on $X$.

\begin{figure}[h] \centering
	\includegraphics[scale=0.35]{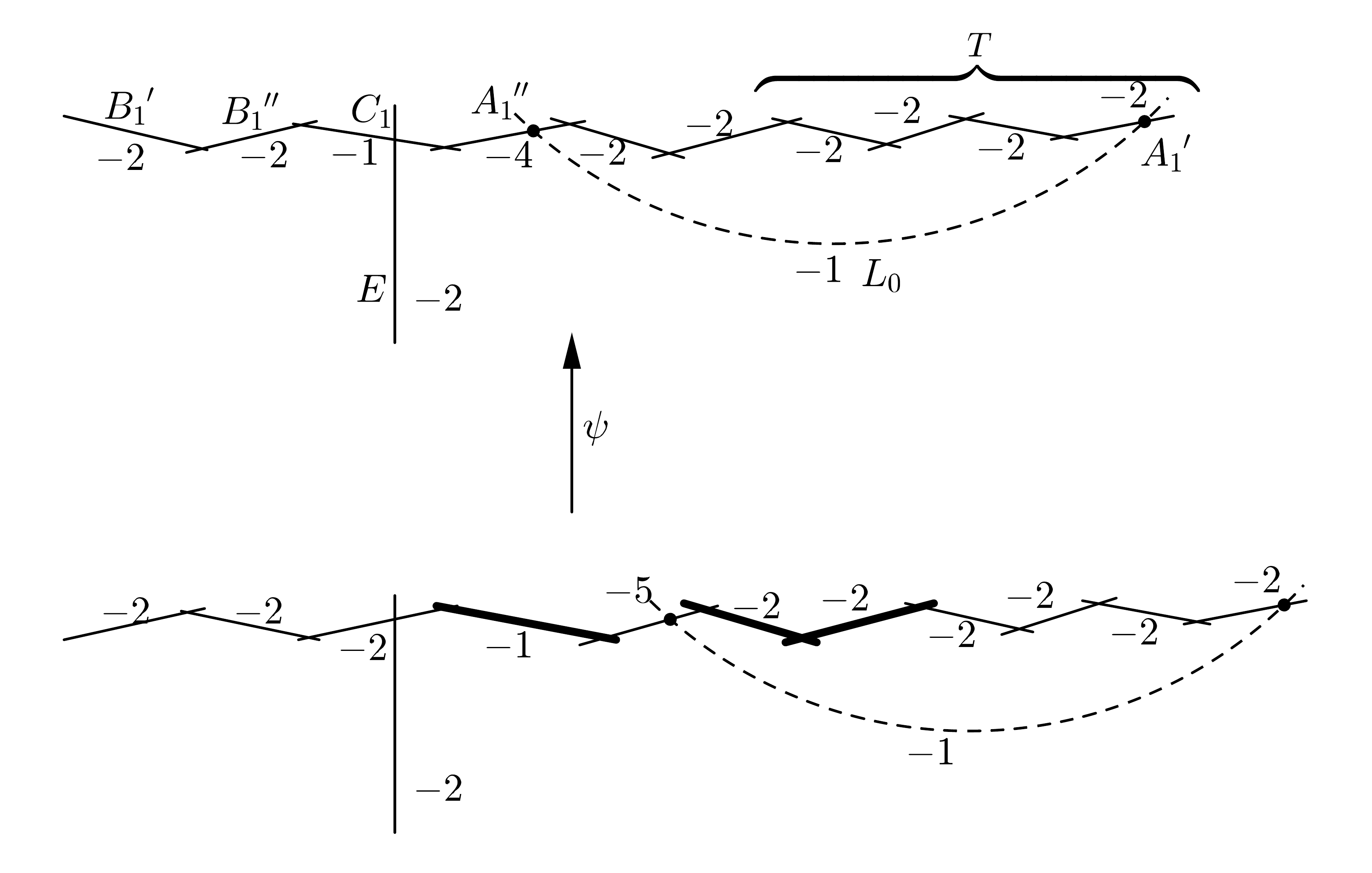}
	\caption{A $\C^{**}$-fibration of $X\setminus D$ with a base point on $D$. Bold curves denote the components which are horizontal for a completion of this fibration.}\label{fig:exprzedl}
\end{figure}

It remains to prove the existence of $L_{0}$. Let $\rho \colon X\to \P^{1}$ be the $\P^{1}$-fibration induced by $|E+2C_{1}+B_{1}''|$. Then $f\cdot A_1''=2$ and $f\cdot B_1''=1$ for any fiber $f$ of $\rho$. By Lemma \ref{lem:fibrations-Sigma-chi} the fiber $F_{A_1'}$ of $\rho$ containing $A_1'$ has at most two components not in $D_{v}$. Let $L$ be such a component. Since $F_{A_1'}\cap D_{v}$ is connected and $F_{A_1'}$ contains no loop, $L\cdot D_{v}\leq 1$. By Lemma \ref{lem:Qhp_has_no_lines} $L\cdot D\geq 2$, so we get $L\cdot D_{h}\geq 1$. If $L$ meets $B_1$ then $\mu(L)=1$ since $F_{A_1'}\cdot B_1=1$; similarly if $L$ meets $A_1''$ then $\mu(L)=1$ since $F_{A_1'}\cdot A_1''=2$ and $F_{A_1'}\cap D_{v}$ meets $A_1''$. By Lemma \ref{lem:singular_P1-fibers}(c) $F_{A_1'}=[1,(2)_{4},1]$. Let $L_{0}$ be the tip of $F_{A_1'}$ meeting $A_1'$. If $L_{0}$ meets $B_1$ then on $\P^{2}$ we have $(\pi_{*}L_{0})^{2}=10$, which is impossible. Hence $L_{0}$ meets $A_1''$ as claimed.
\end{prz}

The following reduction constitutes an important step in the proof of \ref{thm:possible_HN-types}.

\begin{prop}[Reduction to $\C^{**}$-fibrations with no base points]\label{prop:some_Cstst_fibration_extends}
Let $(X,D)$ be a smooth snc-minimal completion of a smooth $\Q$-acyclic surface of log general type. If $X\setminus D$ has a $\C^{**}$-fibration then it has a $\C^{**}$-fibration  without base points on $X$.
\end{prop}
\begin{proof}
Let $p$ be a $\C^{**}$-fibration of $X\setminus D$  and let $\tau\colon (X\s, D\s)\to (X,D)$ be the minimal resolution of base points of $p$ on $X$. Let $\bar{p}$ be a $\P^{1}$-fibration of $X\s$ which restricts to $p$ and let $F$ be a fiber of $\bar{p}$. Recall that by the Poincaré-Lefschetz duality $H_i(X,D;\Q)=H^i(X\setminus D;\Q)=0$, so the long exact sequence of homology for the pair $(X,D)$ shows that $H_1(D,\Q)=H_1(X,\Q)$, hence $D$ is a rational tree.

\begin{claim*}
Let $\phi\colon X\s\to X'$ be a composition of some number of contractions of $(-1)$-curves in $F$ and its images. Put $D'=\phi_{*}D\s$, $F'=\phi_{*}F$. Then for every $(-1)$-curve $C'$ of $F'$ we have $C'\cdot D'_{h} \leq 2$ and $C'\cdot D'_{h}\leq 1$ if $C'\not \subseteq D'$.
\end{claim*}

\begin{proof} Suppose that $C'\cdot D'_{h}=3$. We may assume $F$ is a degenerate fiber. Since $3=F'\cdot D'_{h}\geq \mu(C') C'\cdot D'_{h}$, we have $\mu(C')=1$. We have also $(F'\redd-C')\cdot D_{h}'=0$, so $\phi^{*}(F'\redd-C')\cdot D_{h}\s=0$. The divisor $\phi^{*}C'$ is a rational tree and the divisor $\phi^{*}(F'\redd-C')$ is a sum of rational trees. By Lemma \ref{lem:Qhp_has_no_lines} the tips of the latter divisor which do not meet $\phi^{*}C'$ are contained in $D\s$. Since $X\s \setminus D\s$ is affine, $D\s$ is connected  (see \cite[p.\ 244]{Hartshorne_AG}), which implies that $\phi^{*}(F'\redd-C')\redd \subseteq D\s$. Since $\mu(C')=1$, $F'\redd-C'$, and hence $\phi^{*}(F'\redd -C')$, contains a $(-1)$-curve. The latter is non-branching in $D\s$, so $\tau$ is not a minimal resolution of base points; a contradiction.

Now suppose that $C'\cdot D'_{h}=2$ and $C'\not\subseteq D'$. Then $(F'-\mu(C')C')\cdot D'_h\leq 1$. Again, $\mu(C')=1$, so by Lemma \ref{lem:singular_P1-fibers} $F'$ contains a $(-1)$-curve $L'\neq C'$. We can assume that $\phi$ consists of blow-ups over points of $C'$. Indeed, if $\Exc \phi=R_{1}+R_{2}$, where $R_{1}$ consists of all the connected components of $\Exc \phi$ which meet the proper transform $C$ of $C'$ on $X\s$, then we can decompose $\phi$ as $\phi=\phi_{2}\circ \phi_{1}$ in such a way that $\Exc \phi_{1}=R_{1}$ and $\Exc \phi_{2}=(\phi_{1})_{*}R_{2}$. Then $\phi_{2}$ does not touch $(\phi_{1})_{*}(C)$, hence if the claim holds for $\phi_{1}$, it holds for $\phi$ as well. Suppose $\phi$ touches $L'$. Then $L'\cdot C'\geq 1$, hence $F'=L'+C'=[1,1]$. We have $L'\subseteq D'$, because otherwise $L'\cdot D'=L'\cdot D'_h\leq 1$, contrary to Lemma \ref{lem:Qhp_has_no_lines}. The divisor $\phi^{-1}(L'\cap C')$ contains a $(-1)$-curve $\Gamma$ with $\mu(\Gamma)\geq 2$, hence $\Gamma\cdot D_{h}\s=0$. By the minimality of $\tau$ this $(-1)$-curve is not a component of $D\s$. We get $\beta_{(D\s+\Gamma)\redd}(\Gamma)\leq \beta_{F\redd}(\Gamma)\leq 2$ and $\Gamma \cdot D_{v}\s=\Gamma\cdot D\s\geq 2$ by Lemma \ref{lem:Qhp_has_no_lines}. Since $F$ is a tree, it follows that $F\cap D_{v}\s$ is not connected. But $X\setminus D\s$ is affine, so $D\s$ is connected, which implies that every connected component of $F\cap D_{v}\s$ meets $D_{h}\s$. Since $C\not\subseteq D'$, this means that $L'\cdot D_h'\geq (D_v\s\cap F')\cdot D'_h\geq 2$; a contradiction.

Therefore, $\phi$ does not touch $L'$. Suppose $L'\subset D'$. Then the proper transform of $L'$ on $X\s$ is a $(-1)$-curve in $D\s$, necessarily branching by the minimality of $\tau$, so $\beta_{F'\redd}(L')\geq \beta_{D'}(L')-L'\cdot D'_{h}\geq 3-L'\cdot D'_{h}\geq 2$. We get $\beta_{F'\redd}(L')=2$ and $\mu(L')\geq 2$ by Lemma \ref{lem:singular_P1-fibers}. Because $\mu(L')L\cdot D_h'\leq (F'-C')\cdot D'_h\leq 1 $, we get $L'\cdot D'_{h}=0$, so $L'$ is non-branching in $D'$; a contradiction. Thus $L'\not\subseteq D'$ and then $L'\cdot D'\geq 2$ by  Lemma \ref{lem:Qhp_has_no_lines}.  Since $D'$ is connected, every connected component of $F'\cap D'_{v}$ meets $D'_{h}$, so  $D'_h\cdot (F'\cap D'_v)\geq b_0(F' \cap D_v')$. Because $F'\redd$ is a tree, $L'$ meets each connected component of $F' \cap D_v'$ at most once, which gives $L'\cdot D_h'\geq 2-L'\cdot (F' \cap D_v')\geq 2-b_0(F' \cap D_v')$. We obtain $(F'-\mu(C')C')\cdot D'_h\geq (L'+F' \cap D_v')\cdot D'_h\geq 2$; a contradiction.
\end{proof}

Suppose that $\tau\neq \id$. Note that $\bar{p}$ has at least one degenerate fiber. Indeed, $\rho(X\s)-1\geq \rho(X)= \#D$ and we have $\#D\geq 2$, because otherwise $D$ is a planar curve of degree at most $2$, contrary to the assumption $\kappa(X\setminus D)=2$. Let $H$ be the last exceptional curve of $\tau$. Since $D$ is snc, $H$ is non-branching in $D\s$. By the minimality of $\tau$, $H$ is horizontal. Contract $(-1)$-curves in $F$ and its images until some component meeting $H$, call it $C$, becomes a $(-1)$-curve. Let $\phi\colon X\s\to X'$ be the resulting morphism. Put $D'=\phi_{*}D$, $H'=\phi_{*}H$ and $C'=\phi_{*}C$. 

Suppose $C'\cdot D'\geq 3$. We have $\mu(C')C'\cdot D_h'\leq F'\cdot D_h'=3$. If $C'$ is a component of $D'$ then by the claim above $C'\cdot D_h'\leq \min(2,3/\mu(C'))$, so $C'\cdot D'=C'\cdot D_h'+(C')^2+\beta_{D_v'}(C')\leq \min(2,3/\mu(C'))-1+\mu(C')<3$. Thus $C'\not\subseteq D'$ and by the claim above $C'\cdot D_h'=1$. Since $F'$ is a tree, $C'$ meets each connected component of $D_v'$ at most once, so $2\leq C'\cdot D'-1=C'\cdot D_v'\leq b_{0}(F'\cap D'_{v})$, which means that $C'$ meets two connected components of $F'\cap D_v'$, and hence that $\mu(C')\geq 2$. Since $D'$ is connected, every connected component of $F'\cap D'_{v}$ meets $D'_{h}$, hence $3=F'\cdot D'_{h}\geq b_{0}(F'\cap D'_{v})+\mu(C')C'\cdot D'_{h}\geq 2+\mu(C')\geq 4$; a contradiction.

So we proved that $C'\cdot D'\leq 2$. Then $(C'+H')\cdot D'\leq 2+(-1)+\beta_{D'}(H)\leq 3$, which means that the intersection of $D'$ with a general fiber $f$ of the $\P^{1}$-fibration of $X'$ induced by the linear system $|H'+C'|$ is at most three. Since $X'\setminus D'$ is an open subset of $X\setminus D$, it is of log general type, so by Iitaka's Easy Addition Theorem \cite[10.8]{Iitaka_AG} we have in fact $f\cdot D'=3$. Call $\psi$ the contraction of $H$. Then $(\psi(X\s),\psi_{*}D\s)$ is a smooth pair and the image of $\phi^{*}(H'+C')$ induces a $\P^{1}$-fibration of $\psi(X\s)$ which restricts to a new $\C^{**}$-fibration of $V$. For this $\C^{**}$-fibration the minimal resolution of base points on $X$ is shorter than $\tau$. Thus the lemma follows by induction on the length of $\tau$.
\end{proof}

\subsection{Cases with a fiber at infinity.}

From now on to the end of this section we assume that $\bar{E}\subseteq \P^{2}$ is a rational cuspidal curve, such that $\P^{2}\setminus \bar{E}$ is $\C^{**}$-fibered and of log general type. We keep the notation \ref{not:fibrations_h_and_nu} for $\P^{1}$-fibrations and \ref{not:cusps} for resolutions. By \ref{prop:some_Cstst_fibration_extends}  we can choose a $\C^{**}$-fibration $p$ of $\P^{2}\setminus \bar{E}$ which extends to a $\P^{1}$-fibration $\bar{p}$ of $X$. Moreover, by Lemma \ref{lem:Tono_E2} and Lemma \ref{lem:Tono_E2=-1,-2} we may and shall assume that:

\begin{equation}\label{eq:E2<=_assumption}\tag{$*$}
E^{2}\leq -2  \mbox{ and } E^{2}\leq -3 \mbox{ if } c=1.
\end{equation}

\begin{notation}[Figures of dual graphs]\label{not:graphs}
Throughout the proof we present weighted graphs of $D$ for curves of the HN-types listed in \ref{thm:possible_HN-types} together with $\P^{1}$-fibrations on $X$ restricting to $\C^{**}$-fibrations on $\P^{2}\setminus \bar{E}$. Verification of correctness of the weights is always a straightforward computation based on Lemma \ref{lem:HN_for_chains}. The convention for figures of these graphs is as follows. The horizontal line represents $E$. Bold lines represent horizontal components of $D$. The dashed lines represent components of degenerate fibers which do not lie in $D$. The numbers near the lines are the self-intersection numbers of respective curves. As before, we write $[(n)_{k}]$ for a chain consisting of $k$ $(-n)$-curves. To avoid the necessity of treating some special cases separately, we use the notation $[a,(2)_{-1},3]\de [a+1]$ and $[(2)_{-1},3]\de 0$. 
\end{notation}

We will now prove a few lemmas reducing the proof of \ref{prop:possible_cusp_types} to some special situations.  Let us first make a general remark. Namely, if the $(-1)$-curve $C_j$ is vertical for $\bar p$ then by Lemma \ref{lem:singular_P1-fibers} $\beta_{F\redd}(C_{j})\leq \min(2,\mu(C_j))$, so $C_j\cdot D_h\geq \beta_D(C_j)- \beta_{F\redd}(C_j)\geq 3-\min(2,\mu(C_j))=\max(1,3-\mu(C_j))$. Thus if $C_j$ is vertical for $\bar p$ then 
\begin{equation}\label{eq:C_jD_h}
2\leq \mu(C_j)C_j\cdot D_h\leq 3.\end{equation}
In particular, $C_j$ meets $D_h$.

\begin{lem}[Cases with a fiber at infinity]\label{lem:case_nu=1}
Assume that $D$ contains a fiber of $\bar p$ and that $\bar E\subseteq \P^2$ is not of HN-type $\cD$, $\cE$ or $\cF$. Then there is a $\C^{**}$-fibration of $\P^{2}\setminus \bar{E}$ which has no base points on $X$ and whose completion has no fiber in $D$.
\end{lem}
\begin{proof}
Let $F$ be a fiber of $p$ contained in $D$.

\setcounter{claim}{0}
\begin{claim}\label{clA1}
$F$ contains $E+C_{1}$ and is of type $[2,1,2]$. In particular, $E^{2}=-2$ and $c=2$.
\end{claim}
\begin{proof}
By \eqref{eq:E2<=_assumption} $E^{2}\leq -2$, so $D$ contains no $0$-curves. It follows that $F$ is degenerate, so it has a $(-1)$-curve $L_{F}$. Since $C_{j}$'s are the only $(-1)$-curves in $D$, we have, say, $L_{F}=C_{1}$. Because $Q_{j}$'s are negative definite, $F$ contains $E$.

By \eqref{eq:C_jD_h} $\mu(C_j)C_j\cdot D_h\geq 2$ for every vertical $C_j$, so since $F\cdot D_{h}=3$, we infer that $C_{1}$ is the unique $(-1)$-curve of $F$ and hence that $\mu(C_1)\in\{2,3\}$. The latter implies that if $F$ is not a chain then $C_1$ is a tip of $F$, so $C_1\cdot D_h=\beta_D(C_1)-1=2$ and $\mu(C_1)C_1\cdot D_h\geq 4$, which is false. Thus $F$ is a chain. Since $\mu(C_1)\in\{2,3\}$, we get $F\redd=[2,1,2]$ or $[2,2,1,3]$.

Suppose that $E^{2}<-2$. Then $F=[2,2,1,3]$ by Claim \ref{clA1}, so $E^{2}=-3$. We have $C_1\cdot D_h\leq 3/\mu(C_1)<2$, so $D_{h}$ is a $3$-section, hence $D_{h}\cdot E=0$ and $c=\beta_D(E)=1$. Moreover, $F$ contains a tip of $Q_{1}$. It follows that the last HN-pair of $q_{1}$ is $\binom{3}{3u+1}$ for some $u\geq 0$, so (using the notation of Lemma \ref{lem:HN-equations}) $3|M(q_{1})$ and hence Lemma \ref{lem:HN-equations}(a) fails. Thus $E^{2}\geq -2$. By \eqref{eq:E2<=_assumption} $E^{2}=-2$ and $c\geq 2$. Then $F=[2,1,2]$ and $c=2$.
\end{proof}

The curve $H=D_{h}-C_{2}$ is a $2$-section meeting $C_{1}$. Let $A$, $B$ be the connected components of $Q_{2}-C_{2}$. Since they both meet the $1$-section $C_{2}$, they lie in different fibers $F_{A}$, $F_{B}$ respectively. By \ref{cor:Cstst-fibers}(a) $p$ has two degenerate fibers, so $\bar p$ has at most two degenerate fibers other than $F$. These are necessarily $F_A$ and $F_B$, so by \ref{cor:Cstst-fibers}(b) we have $\sigma(F_{B})=2$ and $\sigma(F_{A})=1$. Denote by $\Gamma_{1}$, $\Gamma_{2}$ and by $L_{F_{A}}$ the components of $F_{B}$ and $F_{A}$ respectively which are not contained in $D_{v}$. 

\begin{claim}\label{clA4}
$F_{B}$ is a chain and we may assume $H$ meets it in tips. \end{claim}
\begin{proof}
Since $F_B$ is a rational tree such that $F_B\cap (X\setminus D)$ is a disjoint union of two $\C^*$'s, the set $D\cap F$ has exactly three connected components (some of them may be points). Because of the connectedness of $D$, each of these connected components meets $D_h$, hence either belongs to (if it is a point) or contains a component of multiplicity $1$. Since $D_v\cap F_B$ contains no $(-1)$-curves, this can happen only if $F_B$ is a chain.

Call $B_{0}$ the component of $B$ meeting $C_{2}$. Assume $B_0$ is a tip of $F_B$. Say that $\Gamma_{1}$ meets $B$. Since $B_0$ is a tip we can perform inner contractions in $F_B$ starting with $\Gamma_1$ until the image of $B_0$ is a $(-1)$-curve. Call $\phi$ the resulting morphism. It contracts only components of multiplicity bigger than $1$, so it does not touch $D_h$. We get $$\beta_{\phi_{*}D}(\phi_{*}(B_{0}+C_{2}))=\beta_{\varphi_*D_v}(\varphi_*B_{0})+\beta_D(C_2)-B_0\cdot C_2=\beta_{\varphi_*D_v}(\varphi_*B_{0})+2\leq 3.$$ Then $|\phi^{*}(\phi_{*}(B_{0}+C_{2}))|$ induces a $\P^{1}$-fibration of $X$ which restricts to a $\C^{**}$-fibration of $X\setminus D$ (it cannot restrict to a $\C^*$-fibration because by assumption $\kappa(X\setminus D)=2$). For this new fibration $E$ is a section, so $\nu=0$ by Claim \ref{clA1}.

Hence we can assume that $B_{0}$ is not a tip of $F_{B}$. Then components of $F_{B}$ of multiplicity $1$ lie in $B$ or are tips of $F_{B}$, so $H$ meets $F_{B}$ in tips; see Fig.\ \ref{fig:fibr_case_nu=1}.
\end{proof}

\begin{figure}[h] \centering \includegraphics[scale=0.5]{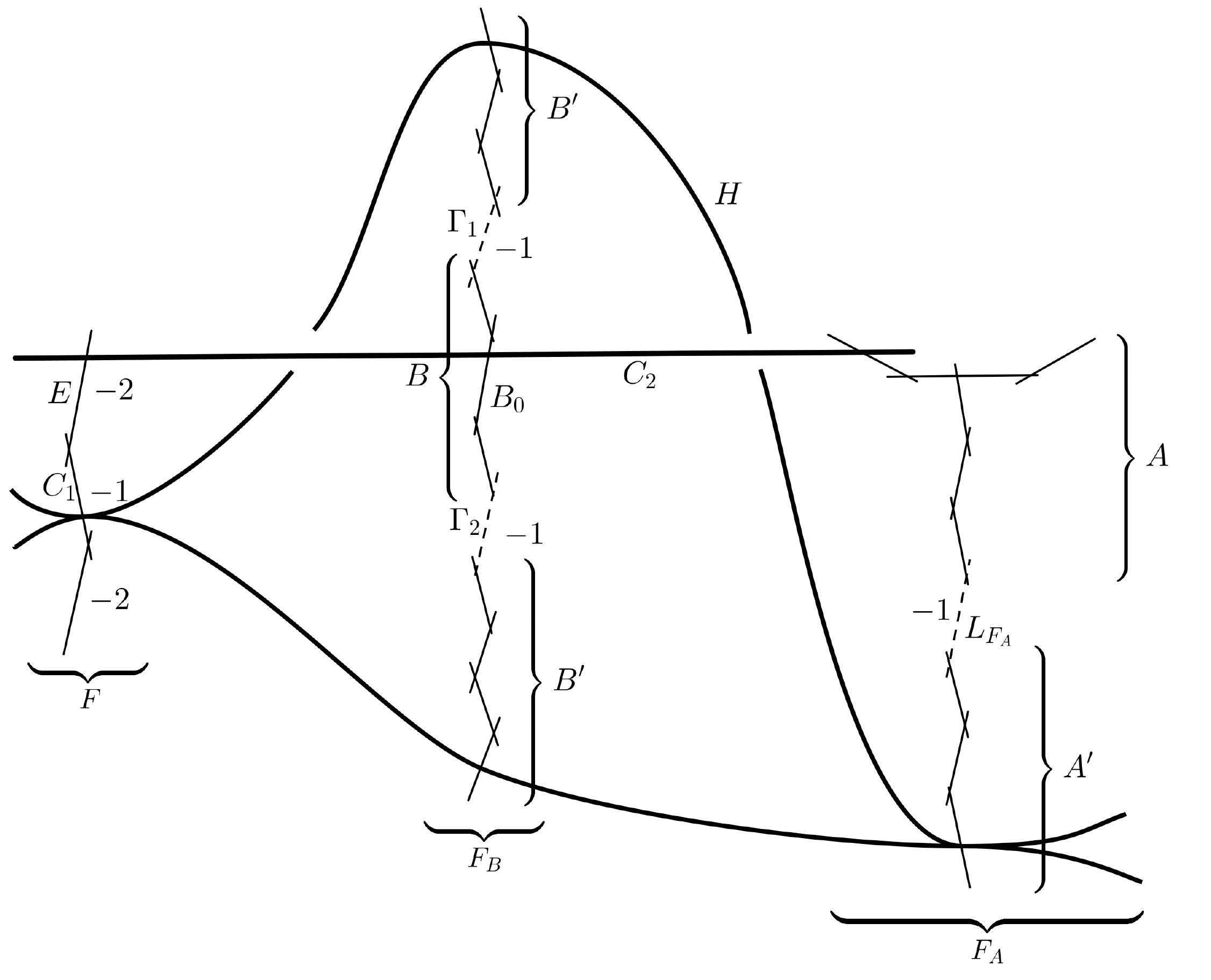} 
\caption{The structure of the $\P^{1}$-fibration in Lemma \ref{lem:case_nu=1}}\label{fig:fibr_case_nu=1}
\end{figure}

\begin{claim}\label{clA5}
We may assume $F_{A}\cap D_{v}$ is connected.
\end{claim}
\begin{proof}
Assume $F_{A}\cap D_{v}$ is not connected. Then it has a connected component $A'$ meeting $H$. By Lemma \ref{lem:singular_P1-fibers} $F_{A}$ is either a chain $[A,1,A']$ or a special fork. Since $\beta_{D}(H)\leq 3$, it follows that $H$ meets at least one of $\Gamma_i$'s, say $\Gamma_2$. Let $B'$ be the connected component of $D_v\cap F_B$ other than $B$ (there is only one, because $\beta_D(H)\leq 3$). If $B_{0}$ meets $\Gamma_{2}$ then $|C_{2}+(F_B-\Gamma_{2})|$ induces a $\P^{1}$-fibration of $X$ as needed. Hence we can assume that $B_{0}\cdot \Gamma_{i}=0$ for every $\Gamma_i$, $i=1,2$ meeting $H$. It follows that $B'\neq 0$. Indeed, if $B'=0$ then by Lemma \ref{lem:singular_P1-fibers}(c) $B$ consists of $(-2)$-curves, so by the negative definiteness of $Q_{2}$, $B_{0}$ is a tip of $B$, so some $\Gamma_i$ meets $H$ and $B_0$, contrary to the assumption.

Consider the case when $F_{A}$ is a chain. Then $A$ meets $C_{2}$ in a tip and $A'$ meets $H$ in a component $A'_{0}$ with $\mu(A'_{0})=2$. Suppose that $\beta_{D}(A'_{0})=3$. Because $Q_1$ contracts to a smooth point, $[1]+B'$ contracts to a smooth point too, so we have  $B'=[(2)_{b}]$ for some $b\geq 1$ and hence $B=[b+2,(2)_{r}]$ for some $r\geq 1$. Note that $B$ meets $C_{2}$ in a tip, for otherwise $B_{0}$ would be a branching $(-2)$-curve in $Q_{2}$ meeting $C_{2}$, which is impossible, as $Q_2$ contracts to a smooth point. Since $B_{0}\cdot \Gamma_{2}=0$, $B_{0}$ meets $\Gamma_{1}$, so $B_0=[b+2]$. Then he contractibility of $Q_{1}$ gives $A=[(2)_{b}]$ or $A=[(2)_s,r+3,(2)_{b}]$ for some $s\geq 0$. Thus $A'=A^{*}=[b+1]$ or $[s+2,(2)_{r},b+2]$. In any case, we check that $A'$ does not contain a component of multiplicity $2$ in $F_{A}$; a contradiction. Therefore, $\beta_{D}(A'_{0})=2$ and since $\mu(A'_{0})=2$, we have $A=[(2)_{u},3]$ and $A'=[2,u+2]$ for some $u\geq 0$. As before, $B_{0}$ is a tip of $B$, for otherwise the contractibility of $Q_{1}$ would imply that $A$ is a $(-2)$-chain, which is false. Hence $B_{0}$ meets $\Gamma_{1}$. Then $Q_{2}$ is a chain, so since it contracts to a smooth point, $B=[2]$ or $[(2)_{v},u+3,2]$ for some $v\geq 0$. From the form of $F_{B}$ we get $B'=[3,(2)_{u}]$. But since $Q_1$ contracts to a smooth point, $A'+[1]+B'$ should contract to a smooth point, and it does not, because $d(A')$ and $d(B')$ are not coprime; a contradiction.

Now consider the case when $F_{A}$ is a special fork and $A$ is a chain. Then $A=[2,b+2,2]$, $A'=[(2)_{b}]$ for some $b\geq 1$ and the component $A_0\subseteq A'$ meeting $H$ is a tip of $F_A$. Since $A$ is not a $(-2)$-chain, the contractibility of $Q_2$ implies that $B_{0}$ is non-branching in $D$, hence it is a tip of $B$ and then again it meets $\Gamma_{1}$. Then $Q_{2}$ is a chain, so we get $B=[(2)_{b-1},3]$ or $[(2)_{s},4,(2)_{b-1},3]$ for some $s\geq 0$. On the other hand, since $A'=[(2)_{b}]$, the contractibility of $Q_1$ implies that $B'=[(2)_{t},b+2]$ for some $t\geq 0$, so because $F_B$ contracts to a $0$-curve, we get $B=[(2)_{r},3,(2)_{b-1},t+2]$ for some $r\geq 0$; a contradiction.

Finally consider the case when $F_{A}$ is a special fork and $A$ is fork. Since $Q_2$ contracts to a smooth point, the long twig of $A$ is necessarily $[(2)_{v},3]$ for some $v\geq 0$. Then $A'=[v+2,3,(2)_{b}]$ and the branching component of $A$ is a $-(b+3)$-curve for some $b\geq 0$. The contractibility of $Q_2$ implies also that $B_{0}$ is a tip of $B$ and $B=[(2)_{b},3]$, so because $F_B$ contracts to a $0$-curve, we get $B'=[2]$. Since $A'+[1]+B'$ contracts to a smooth point, we get $b=v=0$. But then $B=B_{0}$, so $B_0\cdot \Gamma_2=1$; a contradiction.
\end{proof}

\begin{figure}[ht]
	\centering
\vspace{-2em}
\includegraphics[scale=1.2]{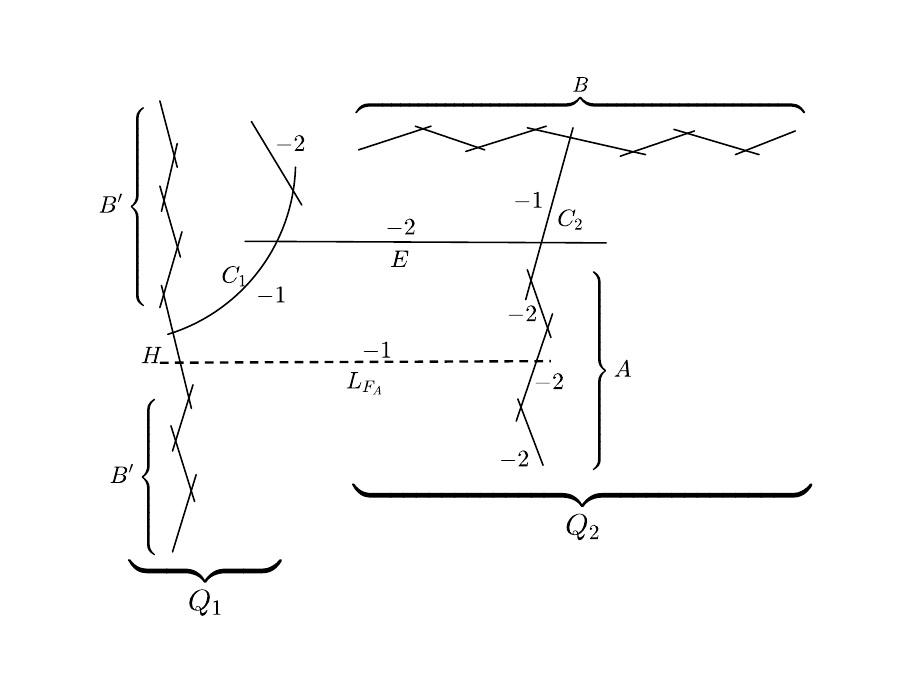}
\caption{The graph of $D$ in Lemma \ref{lem:case_nu=1}.}
\label{fig:graph_case_nu=1}
\end{figure}

By Claim \ref{clA5} and Lemma \ref{lem:singular_P1-fibers}(b), $F_{A}$ is a special fork and $A=F_{A}\cap D_{v}$ consists of $(-2)$-curves. Since $Q_2$ contracts to a smooth point, $A$ is necessarily a chain, so $A=[2,2,2]$; see Fig.\ \ref{fig:graph_case_nu=1}. Assume that $B':=F_B\cap D_v-B$ is connected. Then by Claim \ref{clA4} $Q_{1}$ is a chain, so the contractibility of $Q_{1}$ to a smooth point implies that $H^2=-3$ and $B'=[(2)_{b}]$ for some $b\geq 0$. Hence $F_{B}=[(2)_{b},1,b+2,(2)_{u},1]$ for some $u\geq 0$. Because $Q_{2}$ contracts to a smooth point, we have $B_{0}^{2}=-5$, so $b=3$. It follows that $\bar{E}\subset \P^2$ is of HN-type $\binom{9}{2}$, $\binom{4u+1}{4}$. Since $\gamma=2$, from Lemma \ref{lem:HN-equations} we get $u=1$. Thus $\bar{E}$ is of HN-type $\cD(2,2,1)$.

We can therefore assume that $B'$ is not connected. Write $\pi=\pi_2\circ\pi_1$, where $\pi_1\:(X,D)\to (X',D')$ is the contraction of $(F\redd-E)+C_2+A_0$ and $A_0$ is the tip of $F_A$ meeting $C_2$. Put $L'=\pi_1(L_{F_A})$. Because $\beta_{D}(H)=3$, the morphism $\pi_2\:(X',D')\to (\P^2,\pi(L_{F_A}))$ is a minimal log resolution of the bicuspidal rational curve $\pi(L_{F_A})=\pi_2(L')\subset \P^2$. Since $\pi_1$ does not touch $L_{F_A}$, we have $(L')^2=-1$. By Lemma \ref{lem:Tono_E2=-1,-2}, $\pi_2(L')\subseteq \P^2$ is of one of the HN-types $\cA$ - $\cD$ with $\gamma=1$. But one of the cusps of $\pi_2(L')$ has one HN-pair, because the respective exceptional divisor is a chain, and the second one has two HN-pairs, the second being $\binom{2}{1}$, because $\pi_1(A+C_2)=[2,1]$. It follows that $\pi_2(L')\subseteq \P^2$ is of HN-type $\cC(1,2,k)$ or $\cD(1,2,k)$ for some $k\geq 1$. The resulting HN-pairs are $\binom{4k+4}{2k+1}$ and $\binom{4k+2}{2k+2}\binom{2}{1}$ for HN-type $\cC$ and $\binom{4k}{2k+1}$ and $\binom{4k+2}{2k}\binom{2}{1}$ for HN-type $\cD$, $k\geq 1$ (for $k=1$ the latter sequence degenerates to $\binom{7}{2}$). Then the HN-pairs for $\bar E\subseteq \P^2$ are $\binom{8k+8}{4k+2}\binom{2}{1}$ and $\binom{8k+4}{4k+4}\binom{4}{1}$ for HN-type $\cC$ and $\binom{8k}{4k+2}\binom{2}{1}$ and $\binom{8k+4}{4k}\binom{4}{1}$ for HN-type $\cD$ (degenerating to $\binom{13}{4}$ for $k=1$), which are exactly the HN-pairs of $\cE(k)$ (Fig.\ \ref{fig:E}) and $\cF(k)$ (Fig.\ \ref{fig:F}).
\end{proof}

\begin{figure}[htbp]
	\centering
	\begin{minipage}{0.45\textwidth}
		\centering
		\includegraphics[scale=0.4]{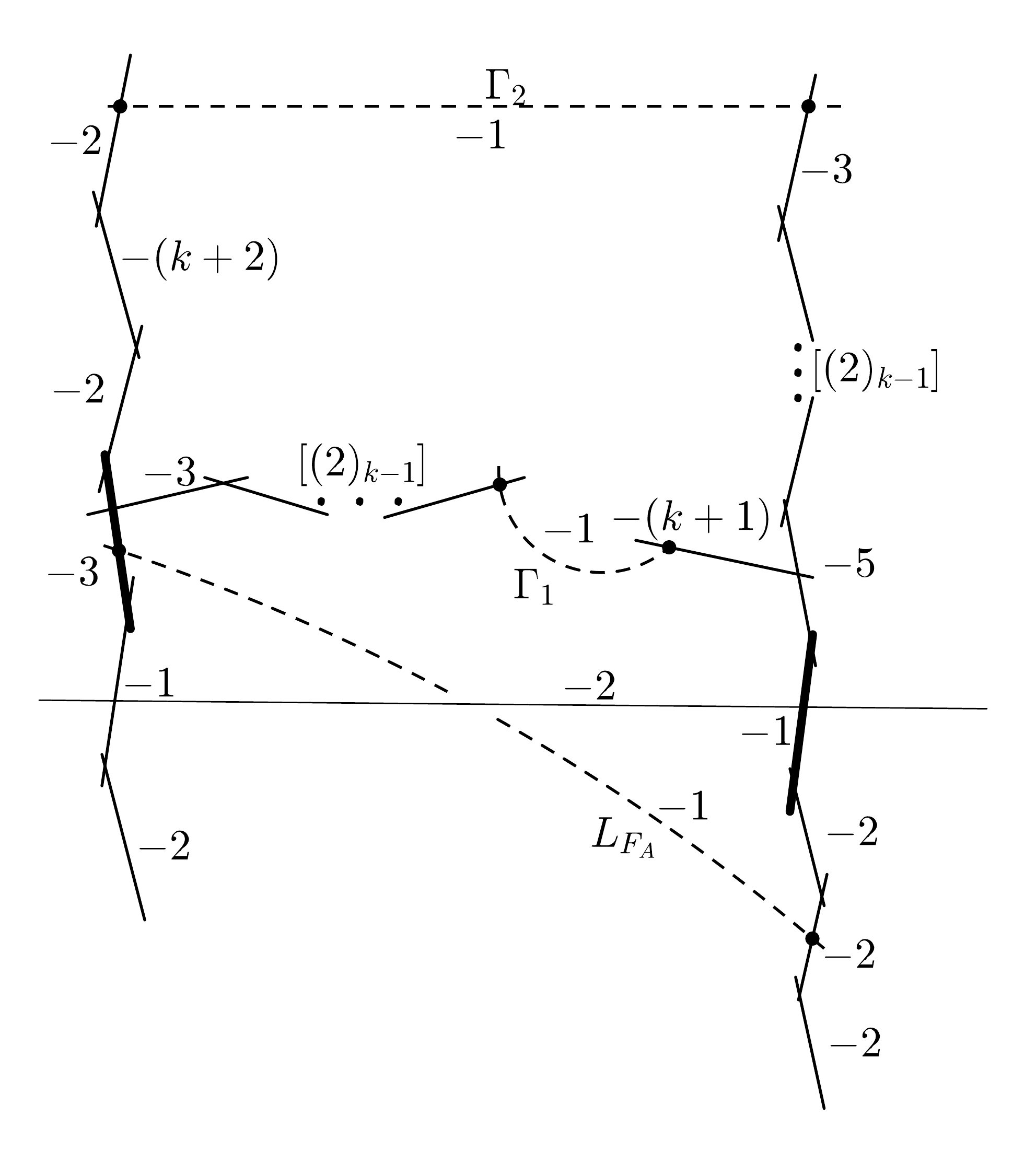}
		\caption{Type $\cE(k)$.}
		\label{fig:E}
	\end{minipage}\hfill
	\begin{minipage}{0.45\textwidth}
		\centering
		\includegraphics[scale=0.4]{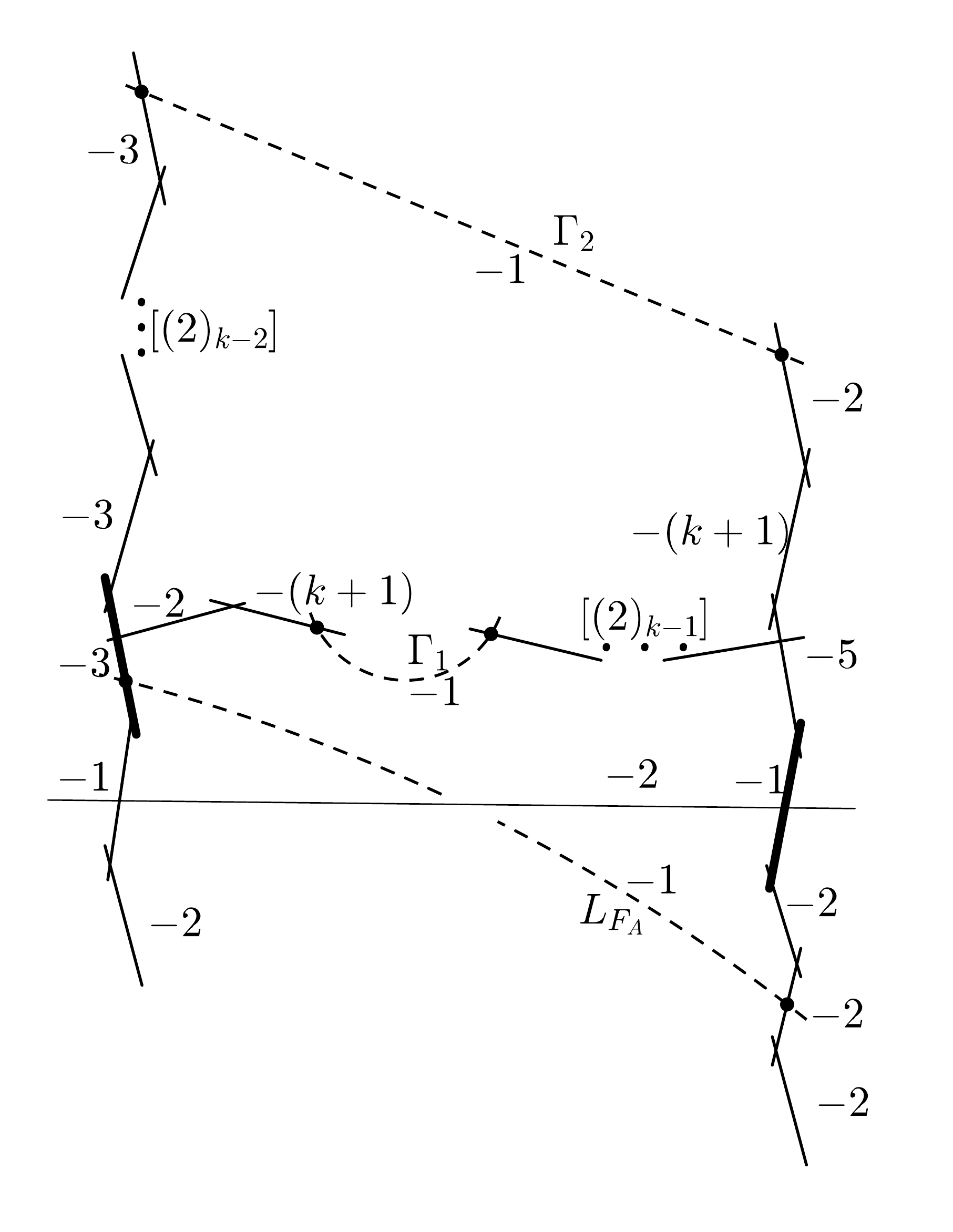}
		\caption{Type $\cF(k)$.}
		\label{fig:F}
	\end{minipage}
\end{figure}

\subsection{Cases with two horizontal components.}

By Lemma \ref{lem:case_nu=1} from now on we may and shall assume that $D$ contains no fiber of $\bar p$, that is, $\nu=0$. From Lemma \ref{lem:fibrations-Sigma-chi} we get $h\in\{2,3\}$.

\setcounter{claim}{0}
\begin{lem}[Cases with two horizontal components in $D$]\label{lem:case_h=2}If $h=2$ then we may choose $p$ so that it has no base points on $X$, $D$ contains no fiber of $\bar p$ and $E$ is a $2$-section of $\bar p$.\end{lem}
\begin{proof}
Since $\nu=0$ and $h=2$, by Lemma \ref{lem:fibrations-Sigma-chi} every degenerate fiber $F$ has a unique component $L_{F}$ not in $D_{v}$. We can write $D_{h}=H_{1}+H_{2}$, where $H_{1}$ is a $1$-section and $H_{2}$ is a $2$ section of $\bar p$. Moreover, by Corollary \ref{cor:Cstst-fibers} $F\cap D$ has exactly two connected components for every degenerate fiber $F$ (one of them could be a point) and each of them has a unique common point with $D_h$.

Suppose $\bar E$ is not a $2$-section of $\bar p$.

\begin{claim}$E$ is vertical.\end{claim}
\begin{proof}
Suppose that $E=H_{1}$. Then $Q_{j}$ is not vertical. Indeed, otherwise $\mu(C_{j})\geq 2$, since $\beta_{Q_{j}}(C_{j})=2$, and then $C_{j}$ could not meet $E$, which is a $1$-section. Thus every $Q_{j}$ contains a component of $D_{h}$ and hence $c=1$. If there exists a degenerate fiber $F$ which does not contain $C_{1}$ then $L_F$ is its unique $(-1)$-curve and $F\cap D_{v}$ does not meet $E$, so $L_{F}$ meets $E$, which gives $\mu(L_{F})=1$, contrary to Lemma \ref{lem:singular_P1-fibers}(b). Hence $\bar p$ has a unique degenerate fiber $F$ and this fiber contains $C_1$. Because $C_{1}$ meets $E$, we have $\mu(C_{1})=1$, so $C_{1}$ is a tip of $F$  and by Lemma \ref{lem:singular_P1-fibers}(a) $L_{F}$ is a $(-1)$-curve. 

Contract $L_{F}$ and the new $(-1)$-curves appearing in the subsequent images of $F-C_1$. Call $\phi\colon X\to X'$ the resulting morphism. Since $\phi$ contracts the unique degenerate fiber $F$ of $\bar{p}$ to a $0$-curve, $X'$ is a Hirzebruch surface $\F_{n}=\P (\O_{\P^{1}}\oplus \O_{\P^{1}}(n))$ for some $n\geq 0$. Since $\beta_{D}(C_{1})=3$, we see that $C_{1}$ meets $H_{2}$, so $H_2$ meets $F$ in components on multiplicity $1$, it is therefore not touched by inner blow-downs in $F$ and its images. It follows that $\phi(H_{2})$ is a smooth rational $2$-section on $\F_{n}$ and, since $C_{1}$ is the only component of $F$ meeting $E$, $\phi(E)$ is a $1$-section disjoint from $\phi(H_{2})$. This is possible only if $n=1$ and $\varphi(E)^2=-1$. But $\varphi$ does not touch $E$, so $E^2=-1$; a contradiction with Lemma \ref{lem:Tono_E2}(a).
\end{proof}

Call $F_{E}$ the fiber containing $E$. If $(F_{E})\redd-L_{F_{E}}$ is not connected, denote by $R$ the connected component of $(F_{E})\redd-L_{F_{E}}$ not containing $E$, otherwise put $R=0$. Put $\gamma=-E^{2}$. 

\begin{claim}\label{cl:C1vert}We may assume that $C_{1}$ is vertical and $c=1$.\end{claim}
\begin{proof}
Suppose that all $C_{j}$'s are horizontal. Then $D_{v}$ contains no $(-1)$-curves, so for every degenerate fiber $F$, $L_{F}$ is the unique $(-1)$-curve in $F$. In particular, $\mu(L_F)\geq 2$. Since $E\cdot (D-\sum_jC_j)=0$, $E$ is a tip of $F_{E}$ and meets $L_{F_E}$. Let $R_0$ be the component of $R$ meeting $D_h$. We have $\mu(E)\leq 2$ and $\mu(R_0)\leq F_{E}\cdot D_{h}-\mu(E)=3-\mu(E)\leq 2$. We have also $c\leq h=2$.

Suppose that $\mu(E)=2$. By Lemma \ref{lem:singular_P1-fibers}(b) $F_{E}$ is a special fork and $E$ meets $D_{h}$ only in $H_{2}$, hence $c=1$. Therefore, by \eqref{eq:E2<=_assumption} $E^{2}<-2$. Then $R$ is a fork with two twigs of type $[2]$, one of which is $R_0$, and the third twig meeting $L_{F_E}$ in a tip. It follows that $L_{F_E}$ meets the last exceptional component of $Q_1$, so $\pi(L_{F_E})$ is a $0$-curve on $\P^{2}$; a contradiction.

Thus $\mu(E)=1$. Now $F_E=[\gamma,1,(2)_{\gamma-1}]$, so $R=[(2)_{\gamma-1}]$ and $R_0$ is a tip of $F_E$. If $c=2$ then some $C_i$, say $C_2$, meets $R_0$, so $|L_{F_E}+R+C_2|$ induces a $\P^{1}$-fibration of $X$ such that $E$ is a $2$-section. But then $D_{v}$, being contained in $D-E$, is negative definite, so $\nu=0$, hence this $\P^{1}$-fibration restricts to a $\C^{**}$-fibration of $X\setminus D$ as needed. 

We may therefore assume that $c=1$. If  $C_1$ is a $2$-section then $R_0$ meets both $H_{1}$ and $C_{1}$, so $R_{0}$ is a branching $(-2)$-curve in $Q_1$ which meets a $(-1)$-curve $C_1$. Since $Q_1$ contracts to a smooth point, this is possible only if $C_1$ is a tip of $Q_1$, and this is not the case. So $C_1$ is a $1$-section. Then $R_0$ meets the $2$-section $H_{2}$, so $\mu(R_0)=2$. If $\#R>2$ then $R_0$ meets two $(-2)$-twigs of $D$, which is impossible, as $Q_1$ is negative definite. Hence $F_{E}=[3,1,2,2]$. In particular, $\gamma=3$.

\begin{figure}[h] \centering 
	\includegraphics[scale=0.35]{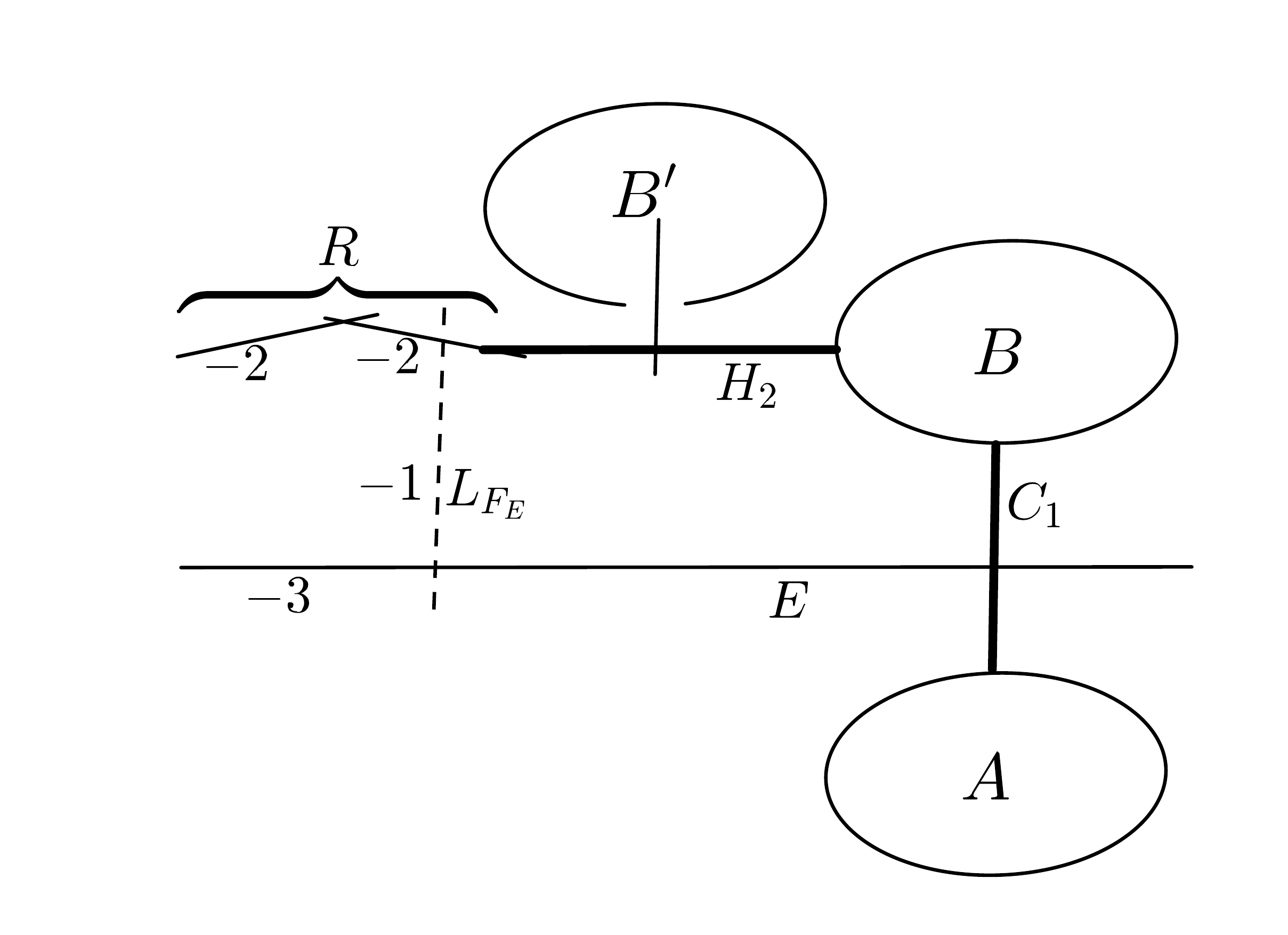}
\caption{The graph of $D$ in Lemma \ref{lem:case_h=2}, Claim \ref{cl:C1vert}.} \label{fig:35fig1}
\end{figure}

Let $A$ be the connected component of $Q_1-C_1$ not containing $H_2$. We can write $Q_1-C_1-A$ as $B+H_2+B'+R$ where $B$ and $B'$ disjoint and connected (possibly zero divisors), $B$ meets $C_1$ and $H_2$ and $B'$ meets $H_2$. Then $A, B, B'$ are connected components of $D_v$; see Fig.\ \ref{fig:35fig1}. Suppose $B\neq 0$. Then $H_{2}\cdot C_{1}=0$. Call $F_{B}$ the fiber containing $B$. Since $\mu(L_{F_{B}})\geq 2$, we have $2= F_{B}\cdot H_{2}\geq 1+2L_{F_{B}}\cdot H_{2}$, and hence $L_{F_{B}}\cdot H_{2}=0$, so $F_{B}$ contains $B'$. Then both $B$ and $B'$ intersect $H_{2}$ in components of multiplicity $1$, so by Lemma \ref{lem:singular_P1-fibers} $F_{B}$ is a chain of the form $B+[1]+B'$ and $H_2$ meets it in tips. Since $R+[1]+B'$ contracts to a smooth point, we have $B'=[(2)_{b},4]$ for some $b\geq 0$, hence $B=(B')^{*}=[b+2,2,2]$. The last tip of $B$ is the unique component of $B$ of multiplicity $1$, so it meets $C_1$. Then this tip of $F_B$ is a branching $(-2)$-curve of $Q_1$, so the contraction of $A+C_1$ makes it into a curve of non-negative self-intersection, which contradicts the negative definiteness of $Q_1$.

Hence $B=0$ and $H_{2}$ meets $C_{1}$. Suppose that $B'\neq 0$. Let $F_{B'}$ be the fiber containing $B'$. Since $\mu(L_{F_{B'}})\geq 2$, $L_{F_{B'}}\cdot C_{1}=0$ and hence $F_{B'}$ contains $A$. We have  $\beta_{D}(H_{2})=3$, so since $Q_1$ contracts to a smooth point, $A=[(2)_{a}]$ for some $a>0$. By Lemma \ref{lem:singular_P1-fibers}(b) $F_{B'}$ is a chain, so $B'=A^{*}=[a+1]$ and $\mu(B')=1$. But then $\mu(L_{F_{B'}})L_{F_{B'}}\cdot H_2=F_{B'}\cdot H_2-B'\cdot H_2=1$; a contradiction with $\mu(L_{F_{B'}})\geq 2$.

Hence $B'=0$. Then $\beta_{D}(H_{2})=2$. Call $F_{A}$ the fiber containing $A$. Since $F_{A}\cap D_{v}$ is connected, $H_2$ meets $L_{F_A}$. By Lemma \ref{lem:singular_P1-fibers} $F_{A}$ is a special fork and $A$ consists of $(-2)$-curves. Since $Q_1$ contracts to a point, $A=[2,2,2]$ and hence $Q_{1}=[2,2,2,1,5,2,2]$. Then $q_{1}\in \bar{E}$ is of type $\binom{13}{4}$ and Lemma \ref{lem:HN-equations}(c) fails; a contradiction.

Thus we may indeed assume that $C_1$ is vertical. It remains to show that $c=1$. If $C_j$ is vertical then it is contained in $F_E$ and by \eqref{eq:C_jD_h} $2\leq \mu(C_{j})C_{j} \cdot D_{h}\leq 3$. Since $F_E\cdot D_h=3$, we see that $C_1$ is the only vertical $C_j$. Since $C_1\cdot D_h\geq 1$, we get $c=E\cdot D_h+1\leq h$. But if $c=2$ then $(E+\mu(C_1)C_1)\cdot D_h\geq E\cdot C_2+2\geq 3$, so the connected component of $D\cap F_E$ not containing $E+C_1$ has no common points with $D_h$; a contradiction.
\end{proof}

By \eqref{eq:E2<=_assumption} $\gamma\geq 3$. Note also that $D_v$ contains no $(-1)$-curves other than $C_1$, so for every degenerate fiber $F\neq F_E$ the curve $L_F$ is the unique $(-1)$-curve in $F$. In particular, $\mu(L_F)\geq 2$ and hence $L_F\cdot H_1=0$. Recall that $R$ denotes the connected component of $(F_{E})\redd-L_{F_{E}}$ not containing $E$ and in case $R\neq 0$ we denote by $R_0$ the component of $R$ meeting $D_h$.

\begin{claim}\label{C1tip} $F_{E}=[2,1,\gamma,1,(2)_{\gamma-3}]$ with $C_1$ as the second curve.
\end{claim}

\begin{proof} Suppose $C_1$ is a tip of $F_E$. Because $\beta_D(C_1)=3$, $C_{1}$ meets both $H_{1}$ and $H_{2}$. Since $\gamma\geq 3$, after the contraction of $C_{1}$ the curve $L_{F_{E}}$ is the unique $(-1)$-curve in the image of $F_{E}$. It follows from Lemma \ref{lem:singular_P1-fibers} that $F_{E}=[1,\gamma,1,(2)_{\gamma-2}]$, so $R=[(2)_{\gamma-2}]$ and $R_{0}$ is a tip of $F_{E}$. 

Suppose that $\beta_{D}(H_{2})=3$. Let $B'\neq R$ be the connected component of $D_{v}$ meeting $D_{h}$ only in $H_{2}$ and let $F_{B'}$ be the fiber containing it. We have $\mu(L_{F_{B'}})L_{F_{B'}}\cdot H_i\leq (F_{B'}-B')\cdot H_i\leq 1$ for $i=1,2$. Since $L_{F_{B'}}$ is the unique $(-1)$-curve in $F_{B'}$, the multiplicity of $L_{F_{B'}}$ is at least $2$, so we get $L_{F_{B'}}\cdot D_h=0$. Then $B'$ intersects $H_{2}$ in a component $B_0$ of multiplicity $2$ and, since $H_1\cdot L_{F_{B'}}=0$, $F_{B'}$ contains a connected component $A$ of $D_v$, which meets $D_h$ only in $H_1$. Because of the contractibility of $Q_1$ to a smooth point, $H_1+A$ is a chain of $(-2)$-curves. It follows that $F_{B'}$ is not a chain, because otherwise $B'=A^*$ is irreducible, and hence $B_0=B'$ would not have multiplicity $2$, as required. By Lemma \ref{lem:singular_P1-fibers} $F_{B'}$ is a special fork. But then, since $A$ is a $(-2)$-chain, we have $(F_{B'})\redd=A+L_{B'}$, so $B'=0$; a contradiction.

Thus $\beta_{D}(H_{2})=2$. If $H_1$ does not meet $D_v-C_1$ then because of the contractibility of $Q_1$ we have either $H_2^2=-2$ and $Q_1=[\gamma,1,(2)_{\gamma-1}]$ or $H_2^2=-3$ and $Q_1=[2,1,3,(2)_{\gamma-2}]$. In the first case Lemma \ref{lem:HN-equations} fails and in the second $\pi_{*}L_{F_{E}}$ is a $0$-curve on $\P^2$, which is impossible. Hence there is a connected component $A$ of $D_v$ meeting $D_h$ only in the $1$-section $H_1$. If $\beta_D(H_1)=3$ there are in fact two such connected components but they lie in different fibers of $\bar p$. This means that $F_{A}\cap D_{v}$ is connected, so by Lemma \ref{lem:singular_P1-fibers} $F_{A}$ is a special fork and $A$ consists of $(-2)$-curves. Since $Q_1$ is contractible to a smooth point, it follows that $\beta_D(H_1)=2$ and that $A$ is a chain $[2,2,2]$ meeting $H_1$ in a tip. Moreover, one of $H_1$ or $H_2$ is a $(-2)$-curve. If $H_1^2=-2$ then $L_{F_E}$ is a $0$-curve on $\P^2$ and if $H_2^2=-2$ then $Q_{1}=[2,2,2,\gamma+1,1,(2)_{\gamma-1}]$ and $q_{1}\in\bar{E}$ is of HN-type $\binom{4\gamma+1}{\gamma}$, so Lemma \ref{lem:HN-equations} gives $\gamma=1$; a contradiction.

Thus we proved that $C_1$ is not a tip of $F_E$. We infer that there is a component $C$ of $(F_E)\redd-E$ meeting $C_1$ and hence that $\mu(C_{1})\geq 2$. By \eqref{eq:C_jD_h} $2\leq \mu(C_1)C_1\cdot D_h\leq 3$, which gives $C_{1}\cdot D_h=C_1\cdot H_2=1$ and $\mu(C_{1})=2$. Let $\phi$ be the contraction of $C_{1}$. Then the connected component of $(\phi_{*}F_{E})\redd - \phi_{*}L_{F_{E}}$ meeting $\phi_{*}H_{2}$ contains two components of multiplicity $1$ in $\phi_{*}F_{E}$, namely $\phi_{*}C$ and $\phi_{*}E$. we have $(\varphi_*E)^2=-\gamma+1\leq -2$, so if $\varphi_*C$ is not a $(-1)$-curve then $\phi_{*}L_{F_{E}}$ is a unique $(-1)$-curve in $\phi_{*}F_{E}$ and hence $\mu(L_{F_{E}})\geq 2$. But in the latter case $R$ meets $H_1$ and hence $R$ also contains a component of multiplicity $1$, which is impossible by Lemma \ref{lem:singular_P1-fibers}. Therefore, $\phi_{*}C$ is a $(-1)$-curve, necessarily unique in $\phi_{*}(F_{E}\cap D_{v})$. Since $\mu(\phi_{*}C)=1$, $\phi_{*}C$ is a tip of $\phi_{*}F_{E}$. The $1$-section $H_1$ meets some component of $R+L_{F_E}$ of multiplicity $1$. It follows now from  Lemma \ref{lem:singular_P1-fibers} that $\phi_{*}F_{E}=[1,\gamma-1,1,(2)_{\gamma-3}]$, and hence that $F_{E}=[2,1,\gamma,1,(2)_{\gamma-3}]$.
\end{proof}

As above, denote by $C$ the tip of $Q_{1}$ contained in $F_{E}$. Clearly, $H_2$ meets $F_E$ only in $C_1$ and $H_1$ meets $F_E$ in the tip contained in $L_{F_E}+R$, because it is the only component of $L_{F_E}+R$ of multiplicity $1$. Since $\beta_D(H_1)\leq 3$, there are at most three connected components of $D_v$ meeting $H_1$, call them $B$, $B'$ and $B''$ with $B$ meeting $H_2$ and $B'$, $B''$ meeting $D_h$ only in $H_1$. If $R\neq 0$ then $R$ is one of these connected components, that is $B''=R$. Now $D_v\setminus F_E$ may have one more connected component $A$, which meets $D_v$ only in $H_2$; see Fig.\ \ref{fig:35fig2}. As usually, we call $F_U$ the fiber of $\bar p$ containing $U$, where $U$ is one of  $A$, $B$, $B'$ or $B''$. 

\begin{figure}[h]\centering 
	\includegraphics[scale=0.3]{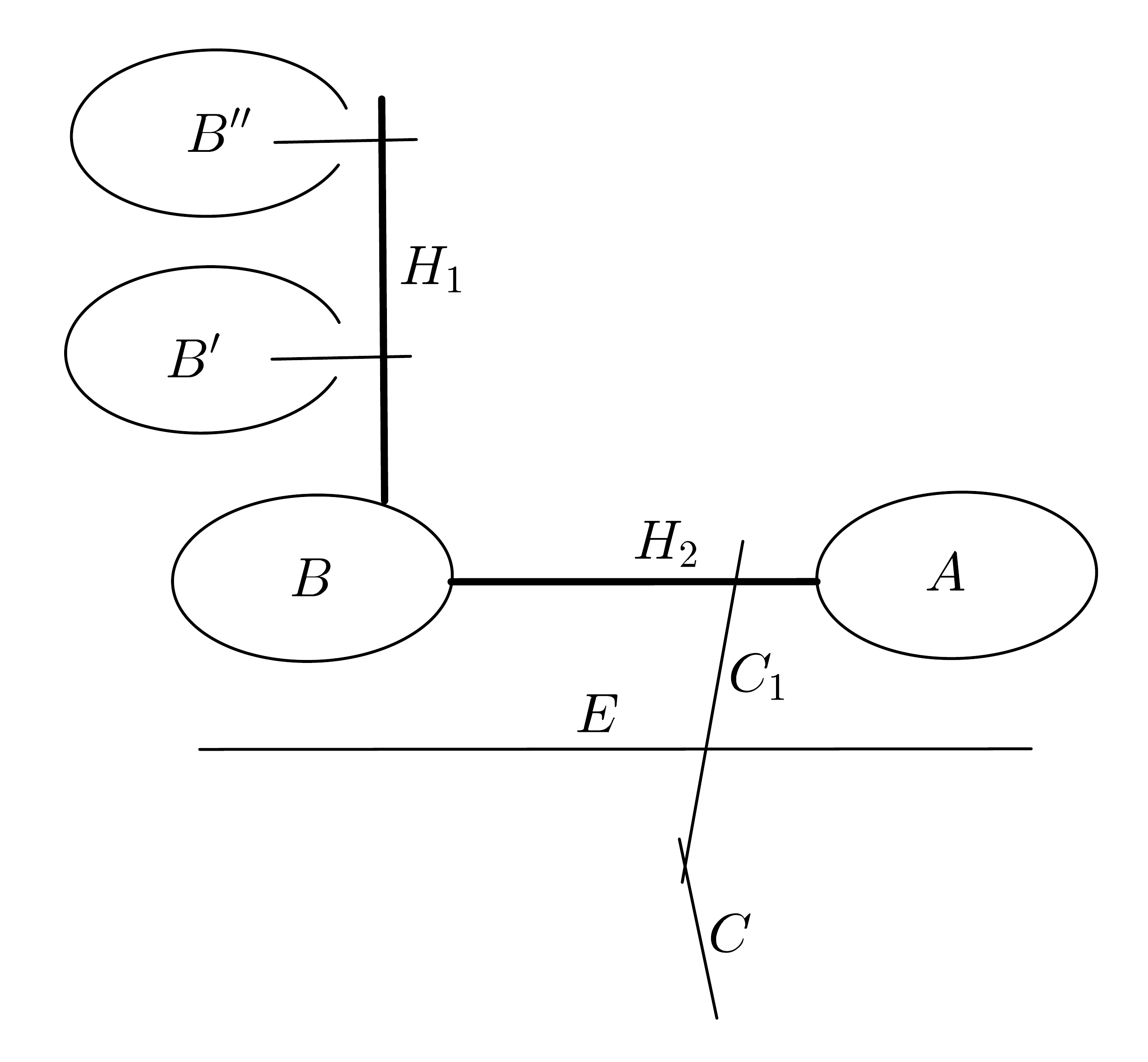}
\caption{The graph of $D$ in Lemma \ref{lem:case_h=2}, Claim \ref{B,B',B''}.} \label{fig:35fig2}
\end{figure}

\begin{claim}\label{B,B',B''} $B=B''=0$ and $B'\neq 0$. Moreover, $F_{B'}$ contains $A$. \end{claim}
\begin{proof}
Suppose $H_1\cdot H_2=0$, that is, $B\neq0$. Both $H_1$ and $H_2$ meet $B$, so $H_2\cdot (F_B-B)\redd\leq 1$. It follows that $H_2$ does not meet $L_B$ and hence $F_B$ contains $A$ and $H_2$ meets it in a component of multiplicity $1$. By Lemma \ref{lem:singular_P1-fibers} $F_{B}$ is a chain $B+[1]+A$ meeting $D_{h}$ in tips (which are unique components of multiplicity $1$). It follows that $H_1$ and $H_2$ meet the same component of $B$. If $B$ is irreducible then $A=B^{*}=[(2)_{-B^2-1}]$. On the other hand, if $B$ is reducible then the contractibility of $Q_1$ to a smooth point implies that $A=[(2)_a]$ for some $a\geq 1$ and hence that $B=(A)^{*}=[a+1]$. In both cases the contraction of $C_1+C_2$ maps $B+H_2+A+C_1+C_2\subseteq Q_1$ onto the chain $A+[1]+B$ of discriminant $0$. This is a contradiction, because $Q_1$ is negative definite.

Now suppose $B'=0$. In case $R=0$ we can assume, renaming $B'$ and $B''$ if necessary, that also $B''=0$. Now $H_1$ meets $D_h$ only in $H_2$ and $R$ and either $H_{1}$ is a tip of $D$ or $R\neq 0$ and $\beta_{D}(H_{1})=2$. We have $D_v=D_v\cap F_E+A$, so $F_A\cap D_v$ is connected. If $A\neq 0$ then by Lemma \ref{lem:Qhp_has_no_lines} $L_{F_{A}}\cdot D_h\neq 0$, so for some $i\in \{1,2\}$ we have $\mu(L_{F_A})\leq \mu(L_{F_A})L_{F_A}\cdot H_{i}\leq i-A\cdot H_{i}= 1$, which is in contradiction to Lemma \ref{lem:singular_P1-fibers}. So $A=0$ and hence $Q_1$ is a chain. Since $Q_1$ contracts to a smooth point we have $Q_1=[2,1,3,(2)_{\gamma-2}]$. Then $\pi_{*}L_{F_{E}}$ is a $0$-curve on $\P^2$; a contradiction.

The fibers $F_{B'}$ and $F_A$ have unique $(-1)$-curves and these $(-1)$-curves have multiplicity at least $2$. If $A\neq 0$ then, since $H_2$ meets $A$, we have $L_{F_A}\cdot D_h=0$, so $D_v\cap F_A$ is not connected, so in any case $(F_A)\redd=(F_{B'})\redd=B'+L_{F_{B'}}+A$. Note also that if $A\neq 0$ then $H_2\cdot A=1$, so since the affine part of the fiber $F_{B'}$ is isomorphic to $\C^*$, $H_2$ meets $F_{B'}$ in a unique point belonging to a component of $L_{B'}+A$ of multiplicity $2$.  Call this component $A_0$.

Suppose that $B''\neq0$, equivalently, that $\beta_D(H_1)=3$. The $1$-section $H_1$ meets $B'$ in some component of multiplicity $1$, so by Lemma \ref{lem:singular_P1-fibers} $F_{B'}$ is a chain or a special fork. Since $Q_1$ contracts to a smooth point, $A=[(2)_a]$ for some $a\geq 0$ and if $a>0$ then the component $A_0$ meeting $H_2$ is a tip of $A$. We observe that the $1$-section $H_1$ meets a $(-2)$-curve contained in $B''$. This is clear if $R\neq 0$, and if $R=0$ then the divisor $F_{B''}\cap D_{v}=B''$ is connected and $F_{B''}$ has only one $(-1)$-curve $L_{F_{B''}}$, so $F_{B''}$ is a special fork and $B''$ consists of $(-2)$-curves. It follows that $F_{B'}$ is a chain, because otherwise $F_{B'}$ is a special fork and then $B'$ meets $H_1$ in a $(-2)$-tip, which contradicts the contractibility of $Q_1$. Since $A_0$ is a tip of $A$, we get $a=2$, hence $B'=A^*=[3]$. If $R\neq 0$ then the contractibility of $Q_1$ gives $R=[2]$, so $\gamma=4$ and hence $q_{1}\in \bar{E}$ is of type $\binom{18}{12}\binom{6}{2}\binom{2}{1}$ and Lemma \ref{lem:HN-equations}(c) fails. Hence $R=0$ and, as we already argued, $F_{B''}$ is a special fork and $B''$ consists of (at least three) $(-2)$-curves. Then $Q_1$ does not contract to a smooth point; a contradiction.
\end{proof}

By Claim \ref{B,B',B''} $R=0$, so $\gamma=3$. As observed in the proof of the above claim, $F_{B'}$ is either a chain or a special fork. Suppose $F_{B'}$ is a special fork and $B'$ is a chain. Since $F_{B'}$ contracts to a $0$-curve, $B'=[2,a+2,2]$ and $A=[(2)_{a}]$ for some $a\geq 0$. In fact $a=0$, because otherwise $Q_1$ does not contract to a smooth point. Then $q_{1}\in \bar{E}$ is of type $\binom{11}{2}$ and Lemma \ref{lem:HN-equations} fails; a contradiction. Suppose $F_{B'}$ is a special fork and $B'$ is a fork. From the contractibility of $Q_{1}$ to a smooth point it follows that the long twig of $B'$ is of type $[(2)_{b},3]$ for some $b\geq 0$. Hence $A=[b+2]$ or $[(2)_{a},3,b+2]$ for some $a\geq 0$. In the latter case, the branching component of $B'$ is a $-(a+3)$-curve and so the contractibility of $Q_1$ implies that the chain $[2,-H_{1}^{2},1,b+2,3,(2)_{a}]$ contracts to a smooth point, which is impossible. Hence $A=[b+2]$ and now the branching component of $B'$ is a $(-2)$-curve. The contractibility of $Q_1$ gives $b=0$, so $q_{1}\in \bar{E}$ is of HN-type $\binom{12}{8}\binom{4}{10}\binom{2}{1}$. Then again Lemma \ref{lem:HN-equations} fails; a contradiction. We are left with the case when $F_{B'}$ is a chain. If the component $A_0\subseteq A$ meeting $H_2$ is branching in $D$ then $B'=A^*$ is not a $(-2)$-chain and we see easily that $Q_1$ does not contract to a smooth point. So $A_0$ is a tip of $A$ of multiplicity $2$, hence $F_{B'}=[2,2+b,1,(2)_b,3]$ for some $b\geq 0$. Then $B'=[3,(2)_b]$ and the contraction of $C+C_1$ maps $Q_1$ onto $[2,b+2,1,-H_1^2,3,(2)_b]$, which does not contract to a smooth point; a contradiction.
\end{proof}

\begin{wn}[Cases with two horizontal components in $D$]\label{cor:case_h=2}
If $h=2$ then $\bar{E}$ is of one of the HN-types $\FZa$, $\cD$ or $\cG$.
\end{wn}

\begin{proof}Choose $p$ as in Lemma \ref{lem:case_h=2}. By Lemma \ref{lem:fibrations-Sigma-chi} every degenerate fiber $F$ has a unique component $L_{F}$ not in $D_{v}$. Let $H_1$ be the $1$-section of $\bar p$ contained in $D$, say, $H_{1}\subseteq Q_{1}$. We now describe the sum of connected components of $Q_{1}-H_{1}$ not containing $C_{1}$, call it $A$.

Suppose $A'\neq 0$ is a connected component of $Q_{1}-H_{1}$ not containing $C_{1}$. Let $F_{A'}$ be the fiber containing $A'$. Since $H_{1}$ is a $1$-section, the connected components of $Q_{1}-H_{1}$ lie in different fibers, in particular $F_{A'}$ does not contain $C_{1}$. Suppose $L_{F_A'}$ meets $Q_j$ for some $j>1$. Then $F_{A'}$ contains $Q_j$, so by \eqref{eq:C_jD_h} $L_{F_{A'}}$ meets exactly one $Q_j$ for $j>1$. After the contraction of $Q_j$ the image of $F_{A'}$ only one $(-1)$-curve, namely the image of $L_{F_{A'}}$, which has therefore  multiplicity at least two and its intersection number with the image of $E$ is at least $2$. The latter is impossible, as $F_{A'}\cdot E=2$. Thus $L_{F_{A'}}$ does not meet any $Q_j$ for $j>1$. It follows that $L_{F_{A'}}$ is the unique $(-1)$-curve in $F_{A'}$ and it meets $E$. By Lemma \ref{lem:singular_P1-fibers} $F_{A'}$ is a special fork and ${A'}$ consists of $(-2)$-curves. Since $Q_1$ contracts to a smooth point, we have in fact $A'=[2,2,2]$ and $H_1$ meets $A'$ in a tip. It follows that every connected component of $Q_{1}-H_{1}$ not containing $C_{1}$ meets $H_{1}$ in a $(-2)$-curve. The contractibility of $Q_1$ to a smooth point implies in this case that $H_{1}\neq C_{1}$ and $\beta_{D}(H_{1})\leq 2$, hence $A$ is connected. Thus, if $A\neq 0$ then $A=[2,2,2]$ and $H_1$ meets $A$ in a tip. Because $H_1\subseteq Q_1$, we see that all $C_j$'s for $j>1$ are vertical.

Denote by $F_j$ the fiber containing $C_j$. Recall that $\mu(q_j)$ denotes the multiplicity of $q_j\in \bar E$. By \eqref{eq:C_jD_h} the fibers $F_j$ are different from each other. Since $D_h\subseteq Q_1$, we have $Q_j\subseteq F_j$ for $j\geq 2$.  It follows that $B_1=F_1\cap D_v$ is connected. By Corollary \ref{cor:Cstst-fibers}(a) $F_{1}\cap D$ has two connected components, so $L_{F_1}$ meets $D_h$. Then $(B_1-C_1)\cdot D_h\leq 2-\mu(C_1)C_1\cdot D_h$ and the inequality \eqref{eq:C_jD_h} gives $(B_1-C_1)\cdot D_h=0$ and $\mu(C_1)C_1\cdot D_h=2$. Since $H_1$ meets $B_1$, we get $H_1\cdot C_1=E\cdot C_1=\mu(C_1)=1$. In particular, $C_{1}$ is a tip of $F_{1}$. Moreover $\mu(L_{F_1})L_{F_1}\cdot E=1$, so $\mu(L_{F_1})=1$. By Lemma \ref{lem:singular_P1-fibers} $F_{1}=[1,(2)_{u-1},1]$ for some $u\geq 2$. It follows that $q_{1}\in \bar{E}$ is of HN-type $\binom{4u+1}{u}$ if $A\neq 0$ and $\binom{u+1}{u}$ otherwise.

Fix $j\neq 1$ and denote by $\varphi$ the contraction of $Q_j$. We have 
\begin{equation*}
2=\varphi_*F_1\cdot \varphi_*E\geq \mu(\varphi_*L_{F_j})\varphi_*L_{F_j}\cdot \varphi_*E\geq \mu(\varphi_*L_{F_j})\mu(q_j) \geq 2\mu(L_{F_j}),
\end{equation*}
where $\mu(q_j)$ is the multiplicity of $q_j\in \bar E$, so $\varphi_*L_{F_j}$ has multiplicity $1$ and hence is not a unique $(-1)$-curve in $\varphi_*F_j$. It follows that $\varphi_*L_{F_j}=[0]$ and $\mu(q_j)=2$, hence $q_{j}\in \bar{E}$ is of HN-type $\binom{2a_{j}+1}{2}$ for some $a_j\geq 1$ and $F_{j}=[1,(2)_{a_j-1},3,1,2]$. 

Consider the case $A\neq 0$. We compute $(\pi_{*}L_{F_1})^2=4$ and $\bar E\cdot \pi_{*}L_{F_1}=(4u+1)+1$, hence by the Bezout theorem $\deg \bar E=2u+1$. Since $F_A\cap (X\setminus D)$ and $F_j\cap (X\setminus D)$ are degenerate fibers of $p$, by Corollary \ref{cor:Cstst-fibers} $c\leq 2$. For $c=1$ Lemma \ref{lem:HN-equations}(c) gives $k=0$ and then $\kappa(X\setminus D)=-\8$, contrary to our assumptions. Thus $c=2$ and then Lemma \ref{lem:HN-equations} gives $\gamma=u+1=a_2+1$, so $\bar E\subseteq \P^2$ is of HN-type $\binom{4\gamma-3}{\gamma-1}$, $\binom{2\gamma-1}{2}$. This is the HN-type $\cG(\gamma)$ (Fig.\ \ref{fig:G}).
\begin{figure}[ht]
	\centering
	\includegraphics[scale=0.45]{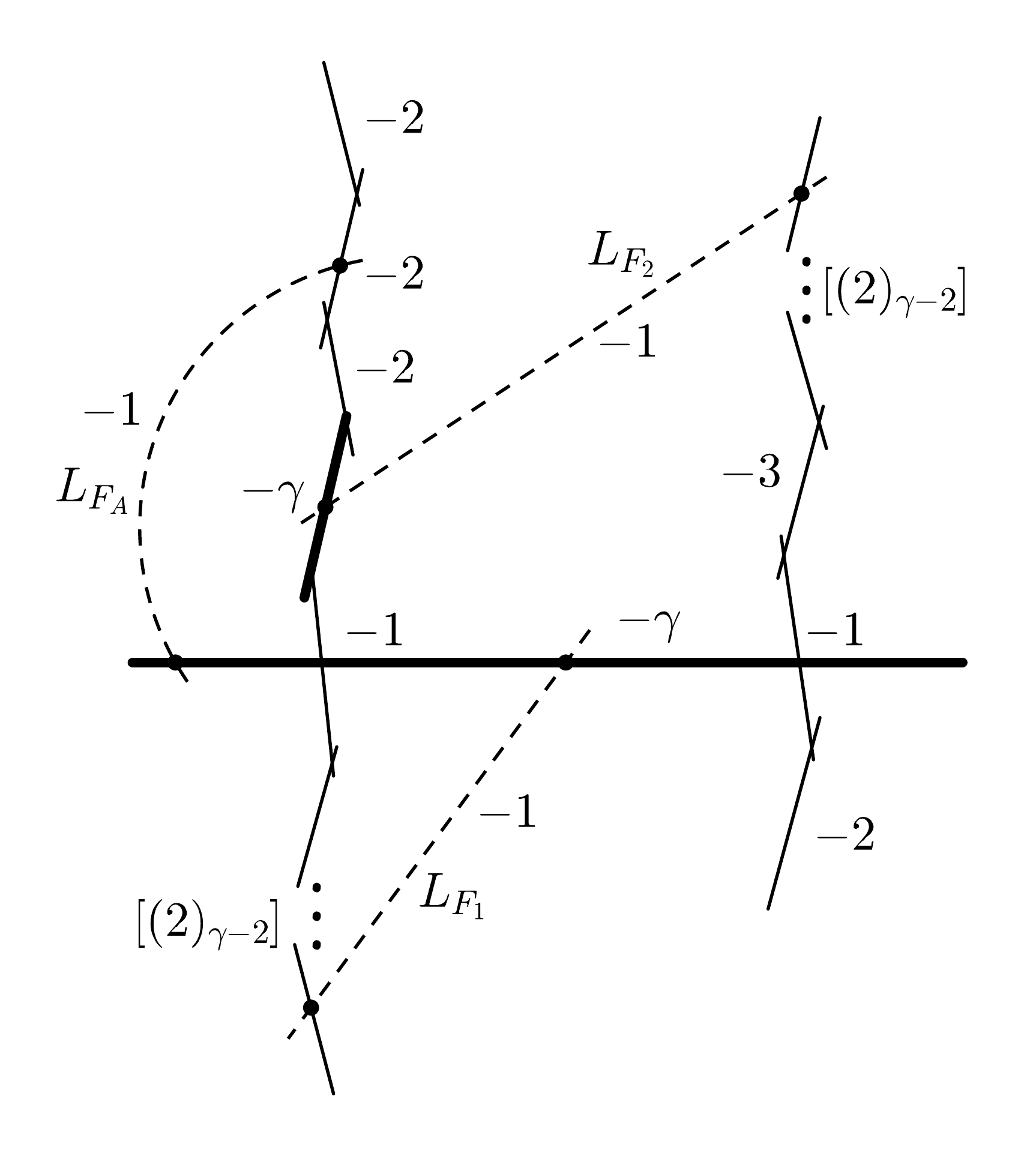}
	\caption{Type $\cG(\gamma)$}
	\label{fig:G}
\end{figure}

Consider the case $A=0$. We compute $(\pi_{*}L_{F_1})^2=1$ and $\bar E\cdot \pi_{*}L_{F_1}=(u+1)+1$, hence $\deg \bar E=u+2$. By Corollary \ref{cor:Cstst-fibers} $c\leq 3$. The equations Lemma \ref{lem:HN-equations} give $\deg \bar{E}=2+a_2+\ldots+a_c$ and $\gamma=\deg \bar{E}+2(c-4)$. For $c=1$ we have $\deg \bar{E}=2$ and then $\kappa(X\setminus D)=-\8$, contrary to our assumptions. For $c=3$ we obtain HN-type $\FZa(\deg\bar{E},a_2)$ if $a_{2}\geq \frac{1}{2}\deg \bar{E}-1$ and $\FZa(\deg\bar{E},\deg \bar{E}-2-a_2)$ otherwise. For $c=2$ the HN-types of the cusps of $\bar E\subseteq \P^2$ are $\binom{\gamma+3}{\gamma+2}$ and $\binom{2\gamma+5}{2}$, which is HN-type $\cD(\gamma,2,1)$ (Fig.\ \ref{fig:D1}): indeed, the second cusp can be described by a non-standard HN-type $\binom{2\gamma+4}{2}\binom{2}{1}$ and thus agrees with the formula in Theorem \ref{thm:possible_HN-types}. Note that $\gamma\geq 2$ by \eqref{eq:E2<=_assumption}. 
\begin{figure}[htbp]
	\centering
	\begin{minipage}{0.5\textwidth}
		\centering
		\includegraphics[scale=0.43]{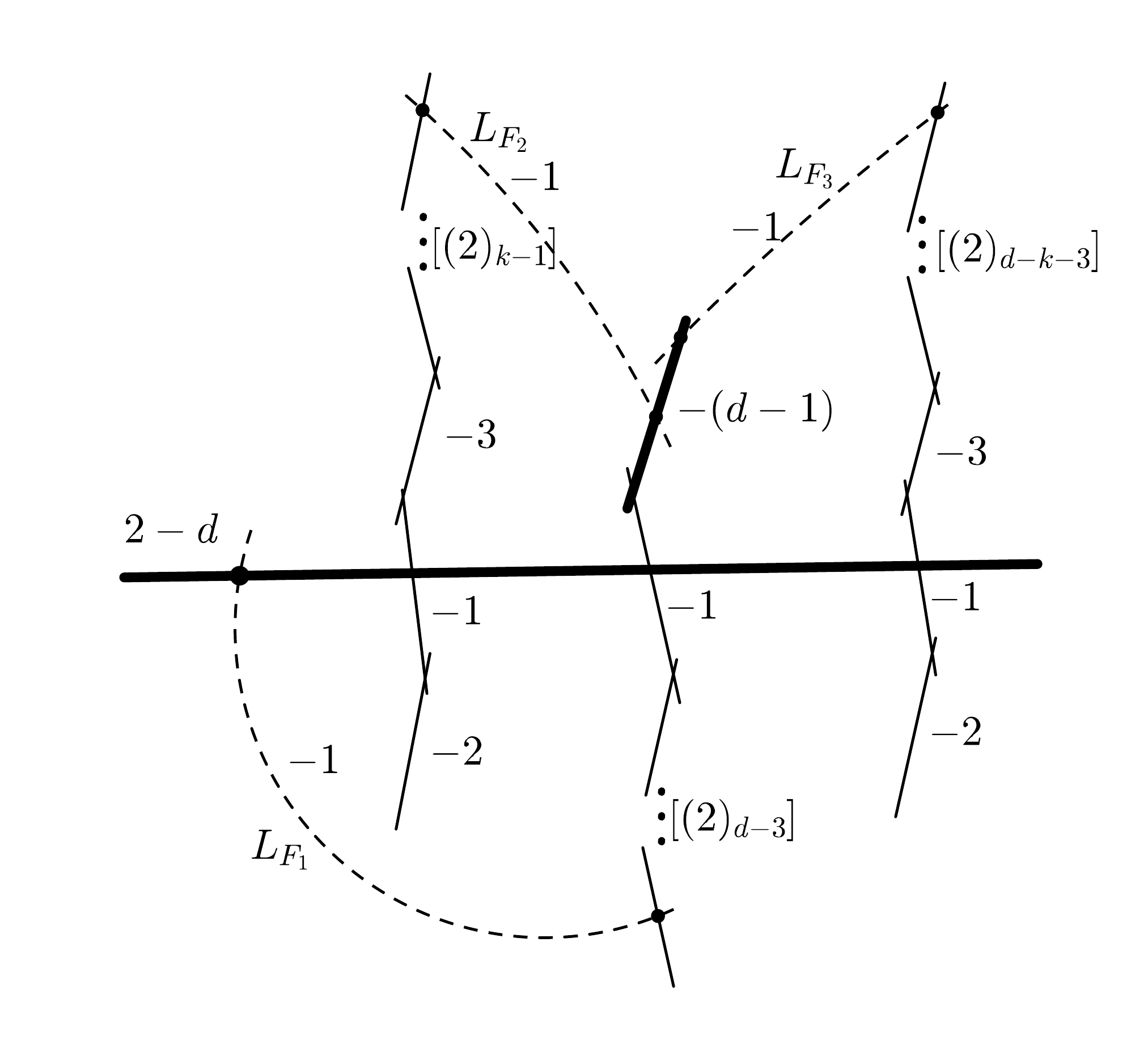}
		\caption{Type $\FZa(d,k)$}
		\label{fig:FZ1}
	\end{minipage}\hfill
	\begin{minipage}{0.5\textwidth}
		\centering
		\includegraphics[scale=0.45]{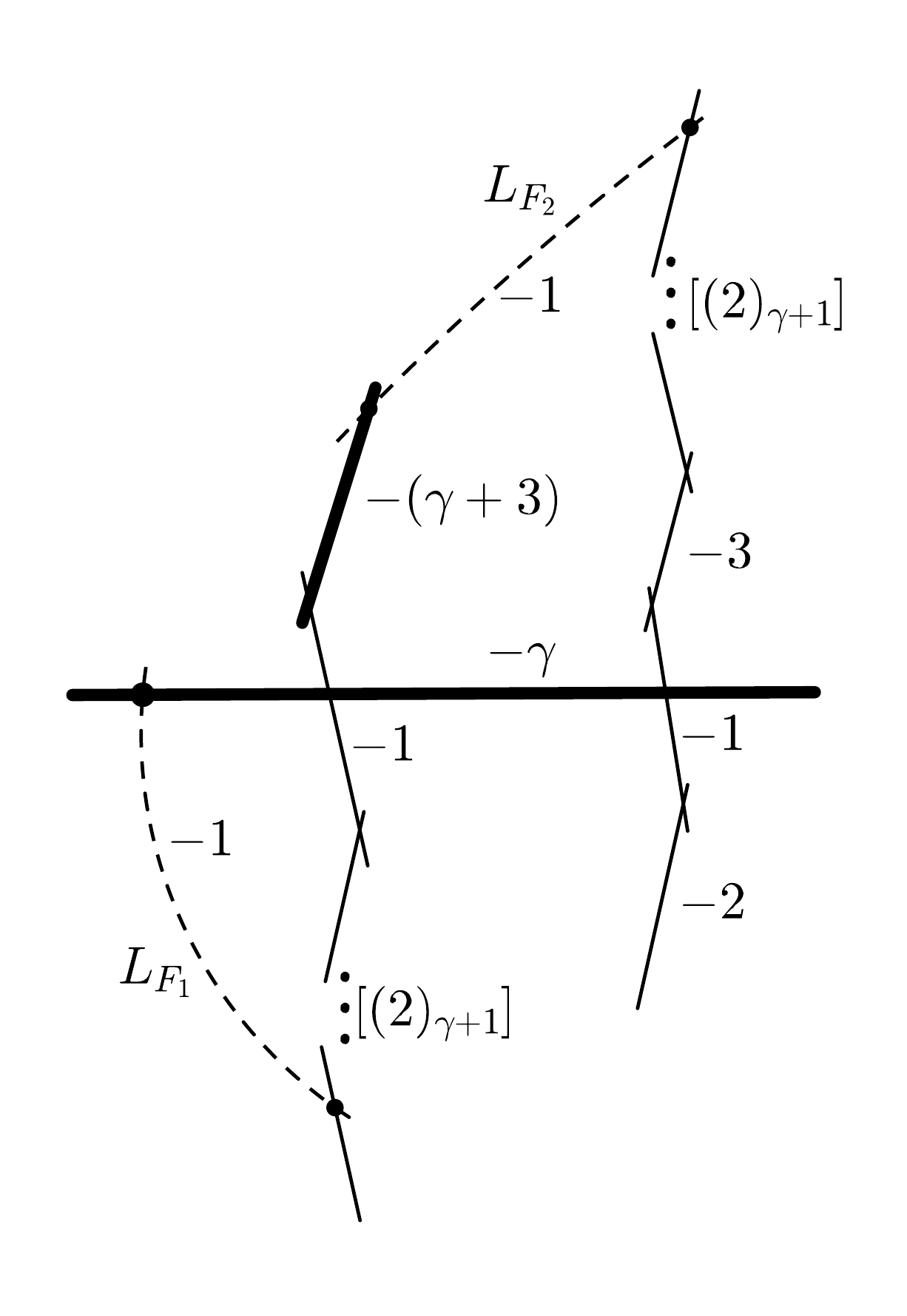}
		\caption{Type $\cD(\gamma,2,1)$.}
		\label{fig:D1}
	\end{minipage}
\end{figure}
\end{proof}

\subsection{Remaining cases with one cusp.}

\begin{notation}[Cases with three sections in $D$]\label{not:h=3}
By the results of previous sections (Proposition \ref{prop:some_Cstst_fibration_extends}, Lemma \ref{lem:case_nu=1}, Corollary \ref{cor:case_h=2} and \eqref{eq:E2<=_assumption}) from now on we may and shall assume that we have a $\C^{**}$-fibration $p$ of $\P^{2}\setminus \bar{E}$, such that $p$ extends to a $\P^{1}$-fibration $\bar{p}$ of $X$ with no fibers in $D$, the horizontal part $D_{h}$ consists of three $1$-sections $H_{1},H_{2},H_{3}$ and we have $E^{2}\leq -2$ and $E^{2}\leq -3$ if $c=1$. As before, we put $\gamma=-E^2$.

Now Lemma \ref{lem:fibrations-Sigma-chi} implies that there exists exactly one fiber $F_{0}$ of $\bar{p}$ containing exactly two components not in $D$, which we denote by $\Gamma_{1}$, $\Gamma_{2}$. Every other degenerate fiber $F$ of $\bar p$ has a unique component not in $D$, which we denote by $L_{F}$.  
\end{notation}

We have $F_0\cap D_v\neq 0$. Indeed, otherwise  by Corollary \ref{cor:Cstst-fibers}(a) $\Gamma_{i}\cdot D_h=\Gamma_i\cdot D\geq 2$, which is impossible, as $F_0\cdot D_h=3$.

\begin{lem}[Degenerate fibers, $h=3$]\label{lem:fibers_are_chains} Let $\bar p$ be as in Notation \ref{not:h=3} and let $F$ be a degenerate fiber of $\bar p$.
\begin{enumerate}[(a)]
	\item If $F$ contains some $C_j$ then $\sigma(F)=\mu(C_1)=1$ and $C_j$ meets two sections from $D$.
	\item Components of $F$ not contained in $D$ are $(-1)$-curves.
	\item Assume $\sigma(F)=1$. Then $D_h$ meets $F$ only in tips and if $F$ contains no $C_j$ then these tips are contained in $D$.
	\item $F$ is a chain.
	\item $F$ contains no subchain of the form $[2,2,2]$ meeting $D_{h}$ in the second component.
\end{enumerate}
\end{lem}

\begin{proof}We have $h=3$ and $\nu=0$, so by Corollary \ref{cor:Cstst-fibers} $\sigma(F)\leq 2$ and $F\cap D$ has $\sigma(F)+1$ connected components. Since $D$ is connected, each of them meets some $1$-section contained in $D$. 

(a),(b) Assume $F$ contains some $C_j$. By \eqref{eq:C_jD_h} $F$ contains exactly one $C_{j}$ and, since $C_j\cdot D_h>0$, $\mu(C_j)=1$, so by \eqref{eq:C_jD_h} $C_j$ meets two sections and hence $\sigma(F)=1$. By Lemma \ref{lem:singular_P1-fibers} $L_F$ is a $(-1)$-curve. Assume now that $F$ contains no $C_j$ and suppose some component not in $D$ is not a $(-1)$-curve. This can happen only if $F$ has a unique $(-1)$-curve and $\sigma(F)=2$. Then $F\cap D$ has three connected components and each of them contains or belongs to a component of $F$ of multiplicity $1$. By Lemma \ref{lem:singular_P1-fibers}($\mathrm{b}_{1}$) $F$ is a chain and $D_h$ meets it only in tips, so one of its tips meets two sections from $D$ and hence is not contained in $D$, say it is $\Gamma_1$. Then $F\cap D_v$ is connected, so $\Gamma_2\cap X\setminus D$ is an affine line; a contradiction with Lemma \ref{lem:Qhp_has_no_lines}.

(c) Since $\sigma(F)=1$, $D\cap F$ has two connected components. If $F$ contains a unique $(-1)$-curve then the assertion follows from Lemma \ref{lem:singular_P1-fibers}($\mathrm{b}_{1}$). Otherwise $F$ contains some $C_j$ which meets two components of $D_h$. If $A$ denotes the connected component of $F\cap D$ containing $C_j$ then $F\redd-A$ has only one component of multiplicity $1$ and this component is a tip of $F$.

(d) Suppose $R$ is a branching component of $F$. Since $F$ contracts to a $0$-curve, there is a connected component $T$ of $F\redd-R$ which contracts to a point. Denote this contraction by $\phi$. Clearly, $T$ contains a $(-1)$-curve which is non-branching in $F$. If $\mu(R)>1$ then all components of $T$ have multiplicity bigger than $1$, so $T_0\cdot D_h=0$ and hence the $(-1)$-curve is not a component of $D$. But in the latter case, Corollary \ref{cor:Cstst-fibers} implies that there is a connected component of $F\cap D$ contained in $T$, which is impossible, as $D$ is connected. Thus $\mu(R)=1$ and hence $\varphi_*R$, having $\beta_{(\varphi_*F)\redd}(\varphi_*R)>1$, is not a $(-1)$-curve. Since $F$ has at most two $(-1)$-curves, $(\varphi_*F)\redd$ has at most one. But then Lemma \ref{lem:singular_P1-fibers}($\mathrm{b}_{1}$) implies that $\phi_{*}R$ is a tip of $\phi_{*}F$; a contradiction.

(e) Assume $T=[2,2,2]$ is a subchain of $F$ and suppose a section $H\subseteq D_h$ meets the middle component of $T$. By (d) and (b) $F$ is a chain and its components not contained in $D$ are $(-1)$-curves. It follows that $T$ is contained in $D$ and its tips are contained in different maximal twigs of $D$. Because $E$ does not meet any $(-2)$-curve from $D$, we may assume $T+H\subseteq Q_1$. Then successive contractions of $(-1)$-curves in the contractible divisor $Q_1$ map $T$ onto $[2,1,2]$, which is not negative definite; a contradiction.
\end{proof}

\begin{notation}[$F$-distance of sections]\label{not:F-distance}
For a given $\P^1$-fibration of smooth projective surface $Y$ and a $1$-section $H$ of $\bar p$ we denote by $\phi_{H}$ the composition of contractions of vertical $(-1)$-curves in $Y$ and its images which do not meet the image of $H$. Such a contraction exists by Lemma \ref{lem:singular_P1-fibers}(a) and $\phi_H(Y)$ is a $\P^1$-bundle. For two sections $H$, $H'$ of $\bar p$ and a fiber $F$ we define their \emph{$F$-distance} $d_F(H,H')$ as the number of components of multiplicity $1$ in the minimal subchain of $F$ meeting both $H$ and $H'$ minus $1$. In particular, the $F$-distance is zero if and only if $H$ and $H'$ meet a common component of $F$.
\end{notation}

\begin{lem}\label{lem:E_is_vert_if_c=1}
If $\bar p$ is as in Notation \ref{not:h=3} and $c=1$ then $E$ is vertical for $\bar p$.
\end{lem}
\begin{proof} Suppose that $E=H_{3}$. Suppose first that $F_{0}$ is a unique degenerate fiber of $\bar{p}$. By Lemma \ref{lem:fibers_are_chains} $F_0$ does not contain $C_1$. It has exactly two $(-1)$-curves and they are not components of $D$. Moreover, $D\cap F$ has three connected components (which may be points). We infer that there exists a contraction $\varphi\:X\to X'$ not touching $D_h$, such that $F'_0=\varphi_*F_0$ is of type $[1,2,\ldots,2,1]$ and both tips meet $\varphi_*D_h$. If $d_{F_0}(H_i,H_j)\geq 3$ for every $i\neq j$ then $F_0'$ contains a chain $[2,2,2,2,2]$ meeting $\varphi_*D_h$ in the middle, so because the base points of $\varphi^{-1}$ belong to the $(-1)$-curves of $F_0'$, $F_0$ contains a chain of type $[2,2,2]$ meeting $D_h$ in the middle component, contrary to Lemma \ref{lem:fibers_are_chains}(e). Thus, $d_{F_0}(H_i,H_i)\leq 2$ for some $i\neq j$. Then $\phi_{H_i}:X\to \F_n$ does not touch $H_i$ and touches $H_j$ at most twice. Because in $D_h$ there may be at most one section with self-intersection bigger than $(-2)$ (namely, $C_1$), we have $\varphi_{H_i}(H_i)^2+\varphi_{H_i}(H_j)^2\leq H_i^2+H_j^2+2<0$, hence $(\varphi_{H_i}(H_i)-\varphi_{H_i}(H_j))^2<0$. But on a Hirzebruch surface $\varphi_{H_i}(H_i)-\varphi_{H_i}(H_j)$ is numerically equivalent to a vertical divisor, so the latter inequality fails; a contradiction.

Thus there is a degenerate fiber $F_1$ of $\bar p$ other than $F_0$. If $F_1$ does not contain $C_1$ then by Lemma \ref{lem:fibers_are_chains} it has a unique $(-1)$-curve, so Lemma \ref{lem:singular_P1-fibers} implies that $D_h$ meets $F_1$ in tips and these tips are components of $D$. But $E$ does not meet $D-C_1$, so we infer that $F_1$ contains $C_1$, and hence $\bar p$ has exactly two degenerate fibers, $F_0$ and $F_1$. By Lemma \ref{lem:fibers_are_chains} $H_1$ meets $C_1$. By Corollary \ref{cor:Cstst-fibers} $p$ has three degenerate fibers, so one of them is contained in a non-degenerate fiber of $\bar{p}$. Hence the latter contains a point where two components of $D_{h}$ meet. This point does not lie on $E$, so $H_{1}$ meets $H_{2}$.

Now Lemma \ref{lem:fibers_are_chains} and Lemma \ref{lem:singular_P1-fibers}(c) imply that for $i=0,1$ $(F_{i})\redd=[1,U_{i},1,V_{i}]$, where the first tip meets $E$ and $U_{i}$, $V_{i}$ are either zero or connected components of $D_{v}$. Consider the contraction $\varphi_E\:X\to X'$ onto some Hirzebruch surface defined above. For a component $S$ of $D_{h}$ denote by $S'$ its image on $X'$. It follows from the definition of $\varphi_E$ that $E'\cdot H_i=0$ for $i=1,2$ and $E'^2=E^2=-\gamma$. Since $E'-H_{i}'$, $i=1,2$ are numerically equivalent to vertical divisors, they intersect trivially, so 
\begin{equation*}-E'^2=E'\cdot (H_2'-E')=H_1'\cdot (H_2'-E')=H_1'\cdot H_2'=\min\{d_{F_0}(H_{1},E),d_{F_0}(H_{2},E)\},
\end{equation*}
hence $d_{F_0}(H_{i},E)\geq \gamma\geq 3$ for $i=1,2$.

Let $R_1\subseteq F_1$ be the component of $F_1$ meeting the sum of $C_1$ and a maximal chain of $(-2)$-curves in $F_1$ meeting it and let $R_0\subseteq F_0$ be the component meeting $H_1$. The contraction of $C_1$ and a maximal chain of $(-2)$-curves in $F_1$ meeting it makes $H_1$ into a $(-1)$-curve. Since both $R_i$'s have multiplicity $1$, there is a contraction $\varphi\:X\to X''$ onto a Hirzebruch surface which contracts neither $R_0$ nor $R_1$. For a component $S$ of $D_{h}$ denote by $S''$ its image on $X''$. The contractions in $F_1$ touch $H_2$ at most once, so $(H_{2}'')^{2}\leq H_{2}^{2}+d_{F_0}(H_{1},H_{2})+1$. Since $H_1''-H_2''$ is a numerically vertical divisor, we have $(H_{1}''-H_{2}'')^{2}=0$, so
\begin{equation*}
2H_{1}''\cdot H_{2}''=(H_{1}'')^2+(H_{2}'')^2=-1+(H_{2}'')^2\leq H_{2}^{2}+d_{F_0}(H_{1},H_{2})\leq d_{F_0}(H_{1},H_{2})-2.
\end{equation*}
But $H_{1}''\cdot H_{2}''\geq H_1\cdot H_2\geq 1$, so $d_{F_0}(H_{1},H_{2})\geq 4$. Since we already proved that $d_{F_0}(H_{i},E)\geq 3$, we get a contradiction with Lemma \ref{lem:fibers_are_chains}(e) as above.
\end{proof}

\setcounter{claim}{0}
\begin{lem}[Reduction to the case $c\geq 2$]\label{lem:E2=2_if_c=1}
If $\bar p$ is as in Notation \ref{not:h=3} then $c\geq 2$.
\end{lem}
\begin{proof}
Suppose $c=1$. By Lemma \ref{lem:E_is_vert_if_c=1} $E$ is vertical. Denote by $F_{E}$ the fiber containing $E$.

\begin{claim}\label{cl:c=1_C1vertical}
We may assume $C_{1}$ is vertical.
\end{claim}
\begin{proof}
Suppose $C_1$ is horizontal, say $C_{1}=H_{3}$. Then $D_v$ contains no $(-1)$-curves and $E$ is one of the connected components of $F\cap D$. By Lemma \ref{lem:fibers_are_chains}(c) every degenerate fiber $F$ is a chain and degenerate fibers with $\sigma(F)=1$ meet $H_i$'s in tips contained in $D$.

Consider the case $\sigma(F_E)=1$. Since $E$ meets $C_1$, $E$ is a tip of $F_E$, so $F_E=[\gamma,1,(2)_{\gamma-1}]$ and the sections $H_1$, $H_2$ meet the tip $R_0$ of $F_E$ other than $E$ (recall that $\gamma\geq 3$). Then $(F_{E})\redd-R_0+C_1$ supports a fiber of a $\P^{1}$-fibration of $X$, for which $E+C_{1}$ vertical, so we replace $\bar p$ with this new fibration. Since $\beta_D(C)=3$, we still have $h=3$. Two components meeting $C_1$ which are now horizontal belong to different connected components of $D-C_1$, so we have also $\nu=0$, as required.

Now consider the case $\sigma(F_E)=2$. Then $\Gamma_1$, $\Gamma_2$ are the only $(-1)$-curves in $F_E$. Since $E^2\leq -3$, we see that $F_{E}\cap D_{v}$ contains more components than just $E$. So there is a $\Gamma_i$, say $\Gamma_{1}$, which meets $E$ and a connected component $A_{1}\neq 0$ of $D_{v}$. Let $A_2$ be the divisorial part of the third connected component of $F_E\cap D$. By Corollary \ref{cor:Cstst-fibers}(a) we have $(\Gamma_i+A_i)\cdot H_i=1$ for $i=1,2$.

Assume that $A_{1}$ is not a $(-2)$-chain. Let $\varphi$ be the contraction of $\Gamma_{1}$ and subsequent new $(-1)$-curves in the image of $A_1$ meeting the image of $E$. By assumption the image of $\Gamma_{1}+A_{1}$ is nonzero and contains no $(-1)$-curves. It follows that all contracted curves had multiplicity bigger than $1$, so $\varphi$ does not touch $D_{h}$. Since $\phi_{*}F_{E}$ meets $D_{h}$ in three different components, Lemma \ref{lem:singular_P1-fibers}($\mathrm{b}_{1}$) implies that $\phi_{*}F_E$ contains at least two $(-1)$-curves, so $\phi_{*}E$ is one of them, hence $\phi_*(E+C_1)=[1,1]$. Then $|\phi^{*}\phi_{*}(E+C_{1})|$ induces a $\P^{1}$-fibration of $X$ such that $E+C_{1}$ is vertical and the new $D_{h}$ also consists of three $1$-sections. Again, two components meeting $C_1$ which are now horizontal belong to different connected components of $D-C_1$, so we have also $\nu=0$, as required. This ends the proof of the claim in this case.

Assume $A_{1}$ is a $(-2)$-chain. $\Gamma_2$ does not meet $A_{1}$, because otherwise $F_E$ contains a proper subchain $\Gamma_{1}+A_{1}+\Gamma_{2}=[1,(2)_{r},1]$ for some $r>0$, which is impossible, as the chain is not negative definite. Hence $\Gamma_2$ meets $E$, so by the above argument we may in fact assume that both $A_i$'s are $(-2)$-chains (possibly empty). Call $\phi$ the contraction of $F_E-E$. Then $\phi$ contracts $F_E$ to a $0$-curve, it does not touch $C_{1}$ and touches $H_{1}$ and $H_{2}$ once each. In particular, self-intersections of all components of $\phi_*D_h$ are negative. By  Lemma \ref{lem:fibers_are_chains} every other degenerate fiber of $\bar p$ is a chain with a unique $(-1)$-curve and it meets $D_h$ in tips. In particular, such an $F$ can be contracted to a $0$-curve so that $D_{h}$ is touched exactly once. But by Corollary \ref{cor:Cstst-fibers}(b) $\bar{p}$ has at most three degenerate fibers, so we obtain a contraction of $X$ onto some Hirzebruch surface such that the image of $D_h$ contains one section with negative and one with non-positive self-intersection number. This is a contradiction.
\end{proof}

By Claim \ref{cl:c=1_C1vertical} we have $C_{1}\subseteq F_{E}$. From \eqref{eq:C_jD_h} we see that $F_E$ is a chain with $\sigma(F_E)=1$, $C_1$ is a tip of $F_E$ meeting $H_1$ and $H_2$, and $H_3$ meets the second tip of $F_E$. We infer that $F_{E}=[1,\gamma,1,(2)_{\gamma-2}]$. Since $\gamma\geq 3$, we have $R=[(2)_{\gamma-2}]\neq 0$. As before, we denote the unique degenerate fiber with $\sigma=2$ by $F_0$.

Denote by $A$ and $B$ the connected components of $Q_{1}-C_{1}$, with $A$ containing $H_1$ and $H_{3}$ and $B$ containing $H_2$. If $H_{1}$ meets $H_{3}$, put $A_1=0$. Otherwise, call $A_1$ the connected component of $D_{v}$ intersecting $H_{1}$ and $H_{3}$ and call $F_{A_1}$ the fiber containing $A_1$. Note that $((F_{A_1})\redd-A_1)\cdot D_{h}=1$, so $F_{A_1}\cap D$ has two connected components, hence Lemma \ref{lem:Qhp_has_no_lines} gives $\sigma(F_{A_1})=1$. By Lemma \ref{lem:fibers_are_chains} if $A_1\neq 0$ then there exists a connected component $A_1'\neq 0$ meeting $D_{h}$ only in $H_{2}$ such that $(F_{A_1})\redd=[A_1,1,A_1']$ meets $D_{h}$ in tips. In particular, $H_1$ and $H_3$ meet the same tip of $F_{A_1}$ which is a branching component of $D$, unless $A_1$ is irreducible.

\begin{claim}\label{cl:B_not-2chain}
$B$ is not a $(-2)$-chain.
\end{claim}
\begin{proof}
Suppose the contrary. Since $Q_1$ contracts to a smooth point, $B$ meets $C_1$ in a tip. 

We claim there is a vertical $(-1)$-curve $\Gamma\not\subseteq D$ such that $\Gamma\cdot H_{1}=0$ and $\Gamma$ meets the tip of $D$ contained in $B$. We have $B_v=B-H_2$. If $B_v=0$ then $B=H_{2}$ is a tip of $D$, so it does not meet $D-H_2-C_1$, hence by Lemma \ref{lem:fibers_are_chains}(b) for some $i\in \{1,2\}$ the curve $\Gamma=\Gamma_{i}\subseteq F_0$ meets $H_{2}$. By Corollary \ref{cor:Cstst-fibers} $\Gamma\cdot (D_h-H_2)=\Gamma\cdot (D-D_v)-1=\Gamma\cdot D-2=0$, so $\Gamma\cdot H_{1}=0$. We may therefore assume $B_v\neq 0$. Call $F_{B}$ the fiber containing $B_v$. By Lemma \ref{lem:fibers_are_chains} $\Gamma$ exists in case $\sigma(F_{B})=1$, so we may assume $\sigma(F_{B})=2$. The section $H_2$ meets $B_v$ in a component of multiplicity $1$. Suppose $B_v$ contains more than one component of multiplicity $1$. Then we have $F_{0}=\Gamma_{1}+B_{v}+\Gamma_{2}$, because $B_{v}$ is a $(-2)$-chain. The fibration $\bar p$ has only two degenerate fibers, $F_0$ and $F_E$, because for a third degenerate fiber $F_1$ we would have that $F_1\cap D_v$ is contained in $A$ and meets $H_2$, which is impossible. It follows that $Q_{1}=B+C_{1}+H_{1}+H_{3}+R$ is a chain and $L_{F_{E}}$ meets its last tip. Then $(\pi_{*}L_{F_{E}})^{2}=0$, which is impossible on $\P^{2}$. Thus there is only one component of multiplicity $1$ in $B_v$. Then by Lemma \ref{lem:singular_P1-fibers}(c) this component is a tip of $B_v$ and some $\Gamma=\Gamma_{i}$ meets $B_v$ in the other tip. Now $F_{0}\cap D_{v}$ is not connected and $\Gamma$ meets two components of $F\cap D_{v}$, so $\Gamma\cdot H_1=0$.

So $\Gamma$ does exist. By Corollary \ref{cor:Cstst-fibers}(a) $\Gamma\cdot D=2$, so $|\Gamma+B+C_{1}|$ induces a $\P^{1}$-fibration of $X$ such that $D_{h}$ consists of three $1$-sections, one of which is $E$. Then the vertical part of $D$ with respect to this new fibration is contained in $Q_1$, hence is negative definite, which gives $\nu=0$. This is a contradiction with Lemma \ref{lem:E_is_vert_if_c=1}.
\end{proof}

\begin{claim}\label{cl:B+C1+H1_twig}$B+C_1+H_1$ is a twig of $Q_{1}$ and $A_1=0$.
\end{claim}
\begin{proof}
We can write $B=H_2+B_1+B_2$, where $B_{1}, B_{2}\subseteq B$ are connected components of $D_{v}$ meeting $D_h$ only in $H_{2}$ (we allow $B_i=0$ if there is less then one such), see Fig.\ \ref{fig:3.11fig1}. Note that $B_{1}$, $B_{2}$ lie in different fibers $F_{B_{1}}$, $F_{B_{2}}$ of $\bar{p}$. 

\begin{figure}[h]\centering \vspace{-3em} \includegraphics[scale=1.8]{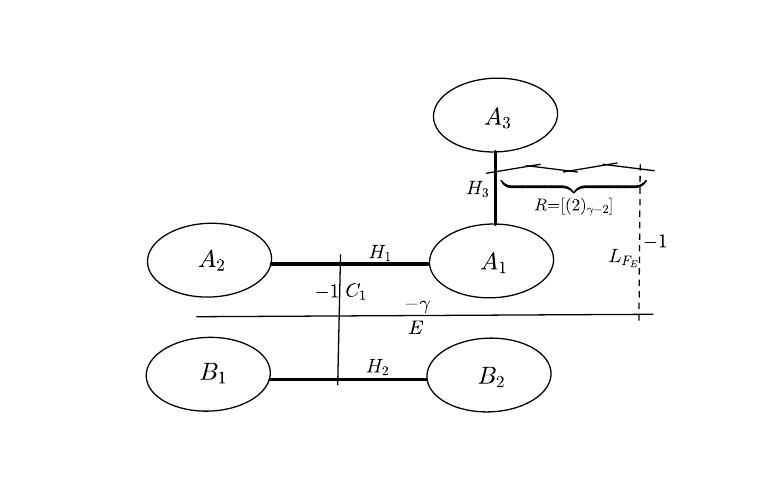} \vspace{-3em}
\caption{The graph of $D$ in Lemma \ref{lem:E2=2_if_c=1}, Claim \ref{cl:B+C1+H1_twig}.} \label{fig:3.11fig1}
\end{figure}

Suppose firstly that $\beta_{D}(H_{2})=3$. Then $B_1, B_2\neq 0$. The contractibility of $Q_{1}$ implies that $H_{2}+C_{1}+A$ contracts to a smooth point, so $A$ is a $(-2)$-chain and $A=H_1+A_1+H_{3}+R$. In particular $D_v-F_E\cap D_v=A_1+B_1+B_2$ and either $A_1=0$ or $A_1=[2]$. If $F_{B_i}$ does not contain $A_1$ then, $F_{B_{i}}\cap D_{v}$ is connected, so by Lemma \ref{lem:fibers_are_chains} and Lemma \ref{lem:singular_P1-fibers} $\sigma(F_{B_i})=2$ and $B_{i}$ is a $(-2)$-chain. This is also the case if $A_1\neq 0$ is contained in some $F_{B_{i}}$, because then $B_{i}=A_1^{*}=[2]$. But the contractibility of $Q_{1}$ implies that $B_{1}+[1]+B_{2}$ is negative definite, so $B_{i}$'s cannot both be $(-2)$-chains; a contradiction.

Thus $\beta_{D}(H_{2})\leq 2$ and we may assume $B_2=0$. By Lemma \ref{lem:fibers_are_chains} $B_1=B-H_2$, which is the vertical part of $B$, is a chain. Suppose it is nonzero and does not meet $H_2$ in a tip. By Lemma \ref{lem:fibers_are_chains}(c) $B_1\subseteq F_{0}$ and, since now $\beta_{D}(H_{2})=2$, $A_1'=0$ and hence $A_1=0$. The contractibility of $Q_{1}$ to a smooth point implies that $A$ is a chain, so $A=H_{1}+H_3+R$. Then $F_{0}\cap D_{v}=B_1$, which is connected. As before, it follows that $B_1$ is a $(-2)$-chain. Since by assumption $H_{2}$ does not meet it in a tip, $Q_{1}$ is not contractible to a smooth point; a contradiction.

Thus $B+C_1$ is a twig of $Q_{1}$. Then in fact $B+C_1+H_1$ is a twig of $Q_{1}$, because otherwise $H_{1}+C_{1}+B$ contracts to a smooth point, so $B$ is a $(-2)$-chain, contrary to Claim \ref{cl:B_not-2chain}. We have now $B=B_1$ and $A-H_1-H_3=A_1+R+A_3$, where $A_3$ is a connected component of $D_v$ meeting $D_h$ only in $H_3$. It remains to show that $A_1=0$. Suppose that $A_1\neq 0$. From the discussion preceding the claim we know that $F_{A_1}$ is a chain containing $A_1'=B_1$. Then $F_{0}\cap D_{v}$ equals $A_3$, so it is connected. Since $F_0$ is a chain with two $(-1)$-curves, this means that $A_3$ is a $(-2)$-chain. But then $H_{3}$ meets two $(-2)$-curves, one in $R$ and one in $A_3$. This is impossible, as $Q_{1}$ contracts to a smooth point; a contradiction.
\end{proof}

We have now $H_{1}\cdot H_{3}=1$ and $D_v-F_E\cap D_v=B_1+A_3$ with $A_3$ as above. Since $H_1\cdot(B_1+A_3)=0$, it follows from Lemma \ref{lem:fibers_are_chains}(c) that $\bar{p}$ has only two degenerate fibers, $F_{0}$ and $F_{E}$. From the above claim and Lemma \ref{lem:fibers_are_chains} in follows that $(F_{0})\redd$ is a chain $[1]+A_3+[1]+B_1$ or $[1]+B_{1}+[1]+A_{3}$. The contraction $\varphi_{H_2}\:X\to X'$ onto a Hirzebruch surface (see Notation \ref{not:F-distance}) does not touch $H_2$ and touches $H_1$ exactly $d_{F_0}(H_1,H_2)$ times. Put $H_i'=\varphi_{H_2}(H_i)$. The divisor $H_1'-H_2'$ is vertical, so we obtain 
\begin{equation*}
0=2H_1'\cdot H_2'=H_1'^2+H_2'^2=H_1^2+H_2^2+d_{F_0}(H_1,H_2),
\end{equation*} hence $d_{F_0}(H_{1},H_{2})=-H_{1}^{2}-H_{2}^{2}\geq 4$. Similarly for $j=1,2$
\begin{equation*}
2H_j\cdot H_3 \leq 2H_{j}'\cdot H_3'=H_{j}'^2+H_3'^2=H_{j}^2+H_3^2+2+d_{F_0}(H_{j},H_3)\leq d_{F_0}(H_{j},H_3)-2,
\end{equation*}
so $d_{F_0}(H_{1},H_{3})\geq 4$ and  $d_{F_0}(H_{2},H_{3})\geq 2$. If $d_{F_0}(H_{2},H_{3})\geq 3$ then $F_{0}$ contains a subchain of type $[2,2,2]$ meeting $H_{2}$ in the middle, which contradicts Lemma \ref{lem:fibers_are_chains}(e). Thus $d_{F_0}(H_{2},H_{3})=2$. Then the above equality gives $H_{2}^2=H_3^2=-2$. If $A_3=0$ then $F_{0}\cap D_{v}=B_{1}$ is connected, so since $F_0$ is a chain with two $(-1)$-curves, $B_{1}$, and hence $B$, is a $(-2)$-chain, contrary to Claim \ref{cl:B_not-2chain}.  Thus $A_3\neq 0$ and $H_{3}$ is a branching $(-2)$-curve in $D$. Because of the contractibility of $Q_{1}$ to a smooth point we see that $B_1+H_2+C_1+H_1+H_3$ contracts to a smooth point and $H_{3}$ contracts last. It follows that $B$ is a $(-2)$-chain; a contradiction with Claim \ref{cl:B_not-2chain}.
\end{proof}

\begin{wn}\label{cor:c=1}
Let $\bar E\subseteq \P^2$ be a rational unicuspidal curve whose complement $\P^2\setminus \bar E$ is of log general type. If the complement is $\C^{**}$-fibered then $\bar E$ is of HN-type $\ORa$ or $\ORb$.
\end{wn}

\subsection{Remaining cases with at least two cusps.}

We keep the assumptions from Notation \ref{not:h=3} and the notation from Section \ref{ssec:fibrations}. In particular, $h=3$ and $\nu=0$ (that is, $D$ contains no fiber and has exactly three horizontal components, the components are sections of $\bar p$). We have also $-E^2=\gamma\geq 2$.

\setcounter{claim}{0}
\begin{lem}\label{lem:E_is_vert_for_c>=2}
If $\bar p$ is as in Notation \ref{not:h=3} then $c\leq 2$. In case $c=2$ we may choose $\bar p$ so that $E$ is vertical and $C_1+C_2$ is horizontal.
\end{lem}
\begin{proof}

For $c=1$ there is nothing to prove, so we assume $c\geq 2$. Suppose that $E$ is horizontal, say $E=H_{3}$. Let $j\in\{1,\ldots,c\}$. If $C_{j}$ is horizontal put $\phi^{j}=\id$, otherwise define $\phi^j$ as the contraction of the fiber containing $C_j$ to a $0$-curve which does not contract $C_j$. Note that in the second case by Lemma \ref{lem:fibers_are_chains}(a) $C_j$ meets some component of $D_h-E$, so $Q_j$ has a horizontal component. We obtain $c=2$. 

The fiber $F_{0}$ meets $E$ in a component which is not contained in $D_{v}$. Therefore, Lemma \ref{lem:fibers_are_chains} and Lemma \ref{lem:singular_P1-fibers} imply that $(F_{0})\redd$ is a chain of type $[1,U_{1},1,U_{2}]$ meeting $E$ in the first tip, say, $\Gamma_{1}$. Note that $U_{1}\cdot D_{h}=1$. Say that $F_{0}$ meets $H_{1}$ in a component $U_{1}'$ of $U_{1}$. Then $F_{0}$ meets $H_{2}$ and $E$ in different tips. Let $\psi$ denote the contraction of $F_{0}$ to a $0$-curve which does not touch $H_{1}$. Denote by $H_{1}'$, $H_{2}'$ and $E'$ the images of $H_{1}$, $H_{2}$ and $E$ on the Hirzebruch surface $\psi\circ \phi^{1}\circ\phi^{2}(X)$. The divisors $H_j'-E'$, $j=1,2$ are vertical, so they intersect trivially. We get
\begin{equation*}
(E')^{2}=E'\cdot (E'-H_{1}')+E'\cdot H_{1}'=H_{2}'\cdot(E'-H_{1}')+E'\cdot H_{1}'=E'\cdot (H_1'+H_2')\geq 0.
\end{equation*}
Since $\phi^{1}\circ \phi^{2}$ does not touch $E$, we actually have $(\psi_{*}E)^{2}\geq 0$. Because $E^{2}<0$, there is a decomposition $\psi=\psi^{1}\circ \psi^{0}$ such that $\psi^{0}_{*}E$ is a $0$-curve. Then the total transform $F$ of $\psi^{0}_{*}E$ on $X$ induces a $\P^{1}$-fibration $\bar{q}$ of $X$ for which $E$ is vertical. The divisor $F\redd$ consists of $E$ and a subchain of $(F_{0})\redd-U_{1}'$ containing $\Gamma_{1}$. By Lemma \ref{lem:singular_P1-fibers} the components of $F$ of multiplicity $1$ for $\bar q$ are tips of $F$. It follows that the horizontal part of $D$ for $\bar{q}$ consists of three $1$-sections, namely $C_{1}$, $C_{2}$ and a component of $F_{0}$. Moreover, $F\redd \not\subseteq D$ and $D-(F\cap D)$ is negative definite, so $D$ contains no fiber of $\bar{q}$. Replacing $\bar{p}$ with $\bar{q}$ we can assume that $E$ is vertical.

Denote by $F_{E}$ the fiber containing $E$. Since $\nu=0$, by Lemma \ref{lem:Qhp_has_no_lines} $F_{E}\cap D$ has at least two connected components and each of them meets $D_{h}$. Suppose some $C_j$ is vertical. Then $C_j\subseteq F_E$ and by Lemma \ref{lem:fibers_are_chains}(a) $C_1$ meets two sections contained in $D$. We infer that all other $C_j$'s are horizontal, so they meet $E$. This means that $E+C_j$ meets all sections of $D_h$; a contradiction. Thus $C_1+\ldots+C_c$ is horizontal. Since $F_{E}\cap D$ has at least two connected components, $E$ meets at most two horizontal components of $D$, so $c=2$.
\end{proof}

\vspace{0.5em}
\setcounter{claim}{0}

\begin{proof}[Proof of the Proposition \ref{prop:possible_cusp_types}]
By Corollary \ref{cor:c=1} if $c=1$ then $\bar{E}$ is of HN-type $\ORa$ or $\ORb$.  Hence we can assume that $c\geq 2$. By the results of the previous subsections we may assume that $p$ is as in Notation \ref{not:h=3}, that is, it extends to a $\P^1$-fibration $\bar p$ of $X$, $D$ has three horizontal components and contains no fiber, and $\gamma= -E^{2}\geq 2$. Then by Lemma \ref{lem:E_is_vert_for_c>=2} $c=2$ and we can assume that $D_{h}=C_{1}+C_{2}+H$ for some $H\subseteq Q_{1}-C_1$. 

We first give a rough description of degenerate fibers of $\bar p$ (see Fig.\ \ref{fig:proof3.1fig1}). 
\begin{figure}[h]\centering 
	\includegraphics[scale=0.45]{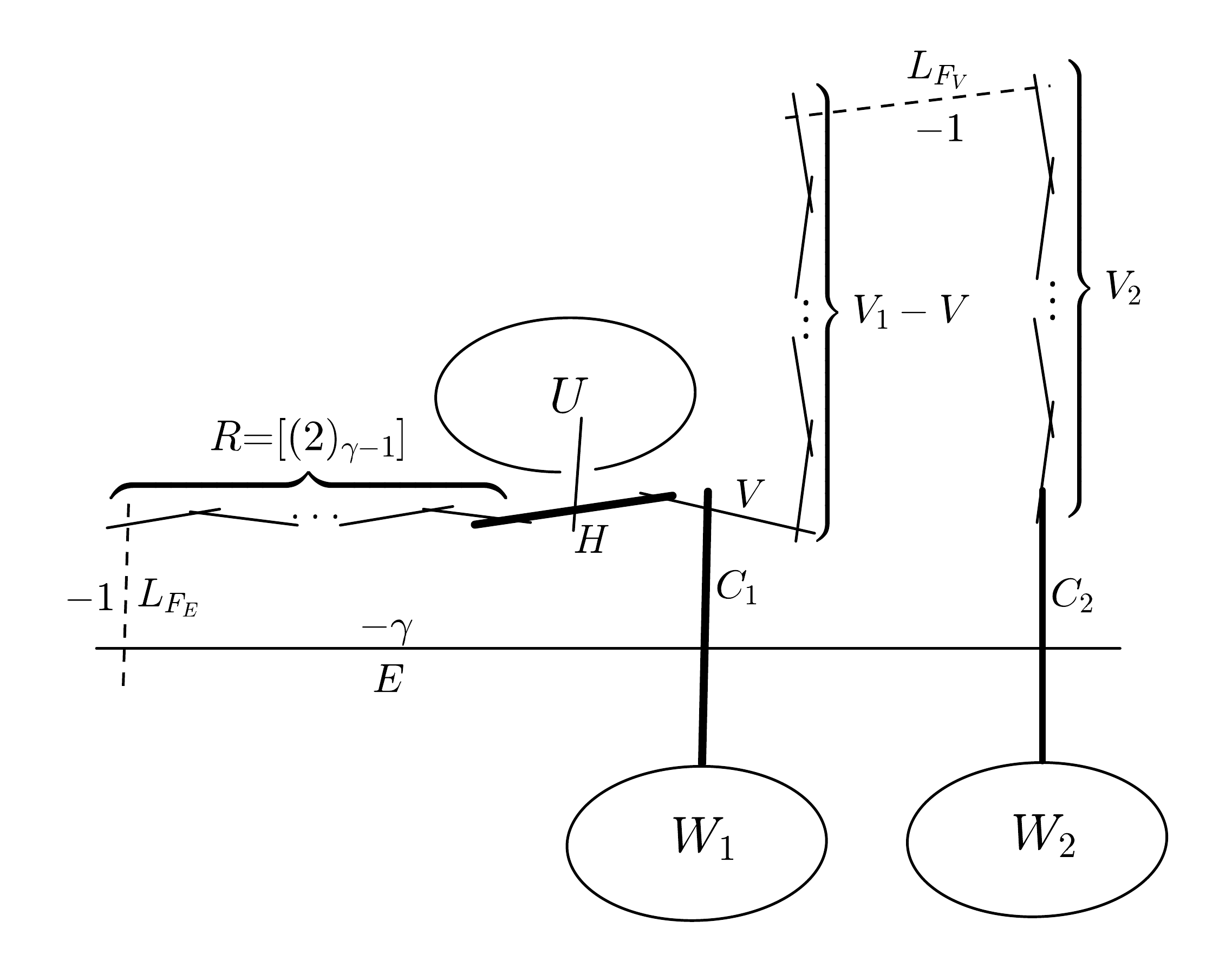}
\caption{The graph of $D$ in the proof of Proposition \ref{prop:possible_cusp_types}.} \label{fig:proof3.1fig1}
\end{figure}
By Lemma \ref{lem:fibers_are_chains} they are chains, there are at most three of them and their $(-1)$-curves are exactly the components not contained in $D$. Denote by $F_{E}$ the fiber containing $E$ and by $F_0$ the unique fiber with $\sigma=2$. Because $E\cdot D_{h}=2$, we see that $F_{E}$ is a chain $[\gamma, 1,(2)_{\gamma-1}]$ with a unique $(-1)$-curve $L_{F_{E}}\subseteq D$. Then $R=(F_{E})\redd-L_{F_{E}}=[(2)_{\gamma-1}]$ is a connected component of $D_{v}$ meeting $H$ in a tip of $F_E$. Denote by $V_{2}$, $W_{2}$ the connected components of $(Q_{2})_{v}=Q_{2}-C_{2}$. They both meet the $1$-section $C_{2}$, so they lie in different fibers $F_{V}$ and $F_{W}$, respectively. We may assume that $\sigma(F_V)\leq \sigma(F_W)$. By Lemma \ref{lem:fibers_are_chains} $F_{V}$ is a chain of the form $V_{1}+L_{F_V}+V_{2}$, where $V_{1}\neq 0$ is the connected component of $(Q_{1})_{v}=Q_1-C_1-H$ meeting both $H$ and $C_{1}$. Moreover, $H$ and $C_1$ meet the same component $V$ of $V_1$, which is a tip of $F_V$.  Call $W_{1}$ and $U$ the connected components of $(Q_{1})_{v}-R$ meeting $D_{h}$ respectively only in $C_1$ and only in $H$ (we allow $U=0$). We have $F_E\cap D_v=E+R$, $F_V\cap D_v=V_{1}+V_{2}$ and hence $F_0\cap D_v=W_{1}+W_{2}+U$. Since $F_0$ is a chain, $W_1$, $W_2$ and $U$ are chains. In $F_0$ we have two $(-1)$-curves $\Gamma_1$ and $\Gamma_2$ and we can choose their names so that $\Gamma_i\cdot W_i=1$ for $i=1,2$. Exactly one of $W_1$, $W_2$, $U$ meets both $\Gamma_i$'s, but for now we do not know which one.

Put $s=-H^{2}-1\geq 1$. Put $p=-V^{2}-1\geq 1$ if $V\subsetneq V_1$ and $p=-V^{2}$ if $V_1=V$. This notation is introduced in a way which makes the final result consistent with notation in Theorem \ref{thm:possible_HN-types} and in the Figures \ref{fig:A}-\ref{fig:D}
The contractibility of $Q_{1}$ to a smooth point implies that $p\geq 2$.

\begin{claim}\label{cl:order-F_0}
$\Gamma_{2}\cdot W_1=0$.
\end{claim}

\begin{proof} Suppose $W_1$ meets $\Gamma_2$. Then the order of components in $F_0$ is $U$, $\Gamma_1$, $W_1$, $\Gamma_2$, $W_2$ and $C_2$ meets $W_2$ in a tip, hence $Q_2$ is a chain. Note that $W_{1}$ is not a $(-2)$-chain, for otherwise a proper subchain $\Gamma_{1}+W_{1}+\Gamma_{2}$ of $F_{0}$ would not be negative definite, which is impossible. The contractibility of $Q_{1}$ to a smooth point implies that $V_{1}=V$ and $s\geq 2$ if $U\neq 0$. We obtain $V_{2}=V_{1}^{*}=[(2)_{p-1}]$ and, because $Q_{2}$ contracts to a smooth point, $W_{2}=[(2)_{r},p+1]$ for some $r\geq 0$. Denote by $W_{1}'$ the component of $W_{1}$ meeting $C_{1}$.

Suppose $U\neq 0$. Then $H+V+C_{1}+W_{1}$ contracts to a smooth point and $H$ contracts last, so $W_{1}=[(2)_{s-2},3,(2)_{p-2}]$ and $W_{1}'$ is a tip of $W_1$, hence it meets $\Gamma_{i}$ for some $i\in \{1,2\}$. Because $\mu(W_{1}')=1$, Lemma \ref{lem:singular_P1-fibers}(c) implies that we can contract this $\Gamma_{i}$ and new $(-1)$-curves in the images of $F_{0}$ in such a way that the image of $W_{1}'$ is a tip of the image of $F_{0}$. Denote this morphism by $\phi$. Then the exceptional divisor of $\phi$ is $[1,2,\ldots,2]$ and, since $U,W_2\neq 0$, it is reducible. But $(\phi_{*}W_{1}')^{2}<0$ and $W_1'$ is a $(-2)$- or a $(-3)$-curve, so $(W_{1}')^{2}=-3$ and then $\phi_{*}(\Gamma_{1}+W_{1}+\Gamma_{2})=[1,2,\ldots,2,1]$ is a proper subchain of $\phi_{*}F_{0}$ which is not negative definite; a contradiction.

Thus $U=0$. Suppose $Q_1$ is not a chain. Then $\beta_{Q_{1}}(W_{1}')=3$, so the contractibility of $Q_{1}$ implies that $R+H+V+C_{1}+W_{1}'$ contracts to a smooth point and $W_{1}'$ contracts last. Thus $R+H+V$ is a $(-2)$-chain and $(W_{1}')^{2}=-\#(R+H+V)-2\leq -4$. In particular, $p=2$. Because $F_0$ contracts to a $0$-curve, we get $W_{1}=[(2)_{u},3,r+2]$ for some $u\geq 0$. But then $W_1'$ is a tip of $W_1$; a contradiction.

Hence $Q_{1}$ is a chain. Note that the component of $R$ meeting $L_{F_E}$ is not the last component of $Q_{1}$ contracted by $\pi$, because otherwise $\pi_{*}L_{F_{E}}$ would be a $0$-curve on $\P^{2}$, which is impossible. The contractibility of $Q_{1}$ gives $W_{1}=[(2)_{u},\gamma+2,(2)_{s-2},3,(2)_{p-2}]$ if $s\geq 2$ and $W_{1}=[(2)_{u},\gamma+1,(2)_{p-2}]$ otherwise. If the tip of $W_1$ met by $\Gamma_1$ is a tip of $D$ then $(F_{0})\redd=[1,(2)_{u},\gamma+2,(2)_{s-2},3,(2)_{p-2},1,(2)_{r},p+1]$ if $s\geq 2$ and $[1,(2)_{u},\gamma+1, (2)_{p-1},1,(2)_{r},p+1]$ if $s=1$. But we check easily that for $p,\gamma\geq 2$ these chains do not contract to $0$-curves, as they should.  Since $F_0$ is a chain, we infer that $\Gamma_{2}$ meets the chains $W_1$ and $W_2$ in tips of $D$ and hence $\pi_{*}\Gamma_{2}$ is a line. We have $\pi_{*}L_{F_{E}}\cdot \pi_{*}\Gamma_{2}=1$, so $\pi_{*}L_{F_{E}}$ is a line, and thus $u+1=(\pi_{*}L_{F_{E}})^{2}=1$, so $u=0$. In case $s=1$ we get $(F_{0})\redd=[1,(2)_{p-2},\gamma+1,1,(2)_{r},p+1]$ and, because $p\geq 2$, this chain does not contract to a $0$-curve. Hence $s\geq 2$ and $(F_{0})\redd=[1,(2)_{p-2},3, (2)_{s-2},\gamma+2,1, (2)_{r},p+1]$. This chain contracts to a $0$-curve if and only if $s=p$ and $r=\gamma$. Then $\pi_{*}L_{F_{V}}$ is a rational unicuspidal curve of HN-type $\binom{p\gamma+p+1}{p\gamma+1}$ and $\deg\pi_{*}L_{F_{V}} =\pi_{*}L_{F_{V}}\cdot \pi_{*} \Gamma_{2}=p\gamma+2$. Moreover, the self-intersection number of the proper transform of $\pi_{*}L_{F_{V}}$ on the minimal log resolution of $(\P^{2},\pi_{*}L_{F_{V}})$ equals $\gamma+1$. This is a contradiction with Lemma \ref{lem:HN-equations}(a).
\end{proof}

\begin{claim}\label{cl:U} $U\cdot\Gamma_1=0$. Either $U=0$ or  $U=[\gamma+1]$.
\end{claim}
\begin{proof}Assume $U\neq 0$ and denote by $U_{0}$ the component of $U$ meeting $H$. Consider a contraction of $Q_2$ to a point followed by contractions in $Q_1$ until $U_0$ becomes a $(-1)$-curve. The image of  $L_{F_E}$ is a $0$-curve and the image of $\Gamma_1$ is disjoint from it and has non-negative self-intersection, so they are in fact fibers of the same $\P^1$-fibration of the resulting surface. It follows that $\Gamma_2$ meets $U$ in $U_0$. 

Suppose that $U$ meets $\Gamma_{1}$ too. Since $\mu(U_{0})=1$, it is possible to contract $F_0$ to a $0$-curve without contracting $U_0$, which implies that $W_{1}=[(2)_{w_1},u+2]$ and $W_{2}=[(2)_{w_2}]$ for some $w_{1}\geq 0$, $w_{2}>0$ such that $w_{1}+w_{2}=\gamma+2$. Since $C_2+V_2$ is a chain, the contractibility of $Q_2$ implies now that $V_{2}=[(2)_v,w_2+2]$ for some $v\geq 0$, hence $V_1=V_2^{*}=[v+2,(2)_{w_2}]$. In particular, $V$ is a branching $(-2)$-curve in $D$ meeting $C_1$. Since $W_1\neq 0$, $Q_1$ does not contract to a smooth point in this case; a contradiction.

Thus $U\cdot \Gamma_{1}=0$, which means that the order of components in $F_0$ is $W_1$, $\Gamma_1$, $W_2$, $\Gamma_2$, $U$. Then $U_{0}$, having multiplicity $1$, is a tip of $F_{0}$ and meets $\Gamma_2$, so $U=U_0$ and $R+[1]+U_0$ contracts to a smooth point, hence $U_0=[\gamma+1]$.
\end{proof}

We infer that the order of components in $F_0$ is $W_1$, $\Gamma_1$, $W_2$, $\Gamma_2$, $U$. Then the tip of $W_1$ meeting $\Gamma_1$ is a tip of $D$. We have the following four cases.

\vspace{0.5em}
\textbf{Case $U=[\gamma+1]$, $V_{1}-V\neq 0$.}
\vspace{0.5em}

Since $V_{1}-V\neq 0$, $V$ is a branching component of $D$, so the contractibility of $Q_{1}$ implies that $V+C_{1}+W_{1}$ contracts to a smooth point and $V$ contracts last. Then $W_{1}=[(2)_{p-1}]$. Similarly, $H+[1]+(V_{1}-V)$ contracts to a smooth point and $H$ contracts last, so $V_{1}=[(2)_{s-1},p+1]$ and hence $s\geq 2$. We obtain $V_{2}=V_{1}^{*}=[s+1,(2)_{p-1}]$. Because of the contractibility of $Q_2$, we infer that $W_2$ meets $C_2$ in a $(-p-1)$-curve which is non-branching in $Q_2$. In particular, $Q_2$ is a chain.

Let $\phi$ be the contraction of $W_1+\Gamma_{1}$. If $\varphi_*\Gamma_2$ is a unique $(-1)$-curve in $\phi_{*}F_{0}$ then by Lemma \ref{lem:singular_P1-fibers} $F_{0}=[(2)_{p-1},1,p+2,(2)_{\gamma-1},1,\gamma+1]$, so $W_2$ does not contain a $(-p-1)$-curve, as it should. So $\phi_{*}F_{0}$ contains a $(-1)$-tip and then Lemma \ref{lem:singular_P1-fibers} gives  $F_{0}=[(2)_{p-1},1,p+1,(2)_{r},3,(2)_{\gamma-1},1,\gamma+1]$ for some $r\geq 0$. We have $\gamma,p\geq 2$, so since $C_2$ meets $W_2$ in a $(-p-1)$-curve which is non-branching in $Q_2$, $C_{2}$ meets the same component of $W_2$ as $\Gamma_1$. Then $Q_{2}=V_{2}+C_{2}+W_{2}=[s+1,(2)_{p-1},1,p+1,(2)_{r},3,(2)_{\gamma-1}]$ contracts to a smooth point, so $0\leq r=s-2$. Using Lemma \ref{lem:HN_for_chains} we check that this is the HN-type $\cA(\gamma,p,s)$ with $\gamma\geq 2$, see Fig.\ \ref{fig:A}.

\begin{figure}[h]
	\centering \hspace{-1.5cm}
	\begin{minipage}{0.43\textwidth}
		\centering
		\includegraphics[scale=0.4]{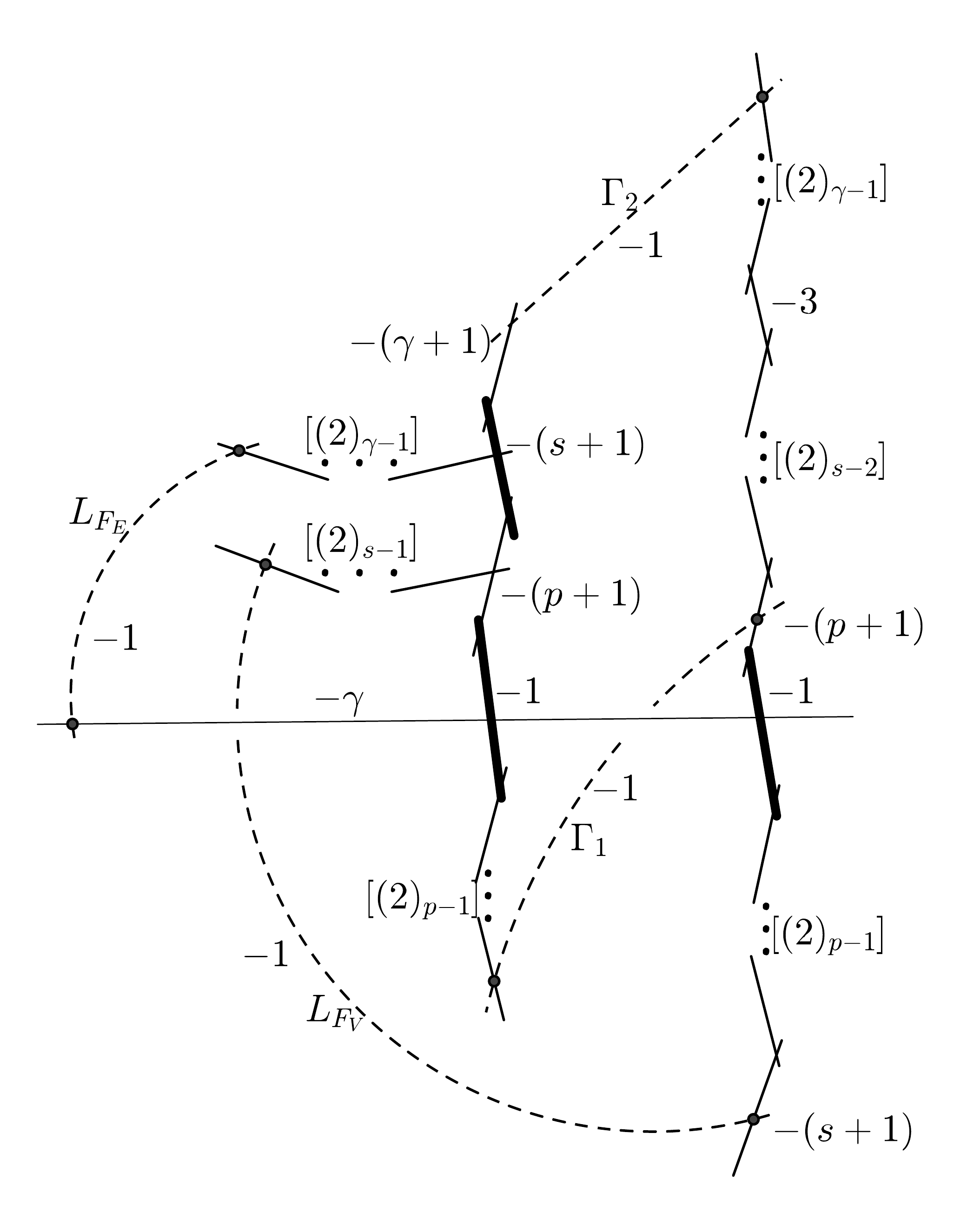}
		\caption{Type $\cA(\gamma,p,s)$.}
		\label{fig:A}
	\end{minipage}\hspace{1.5cm}
	\begin{minipage}{0.43\textwidth}
		\centering
		\includegraphics[scale=0.4]{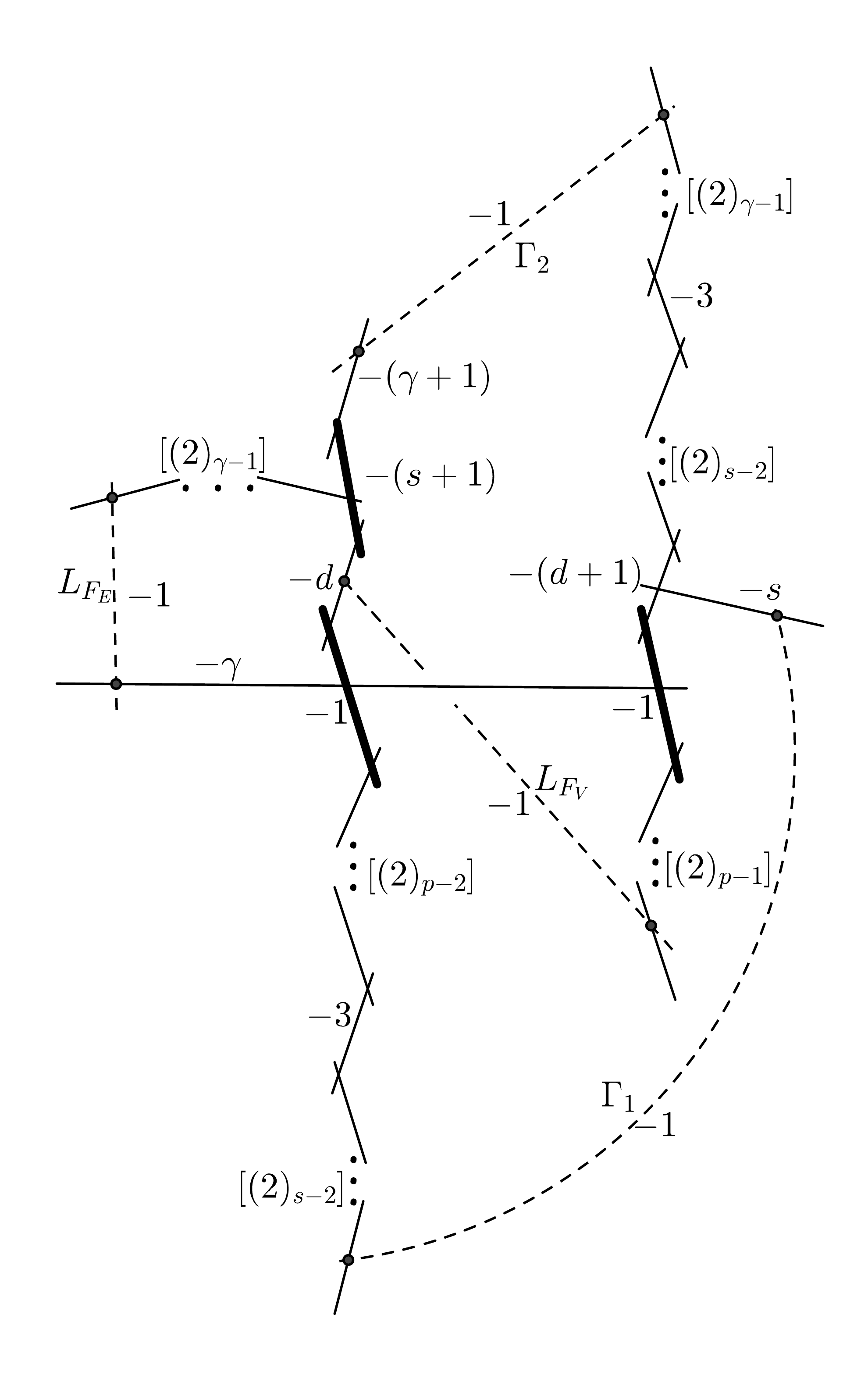}
		\caption{Type $\cB(\gamma,p,s)$.}
		\label{fig:B}
	\end{minipage}
\end{figure}

\vspace{0.5em}
\textbf{Case $U=[\gamma+1]$, $V_{1}=V$.}
\vspace{0.5em}

We have $V_{2}=V_{1}^{*}=[(2)_{p-1}]$, so $C_{2}+V_{2}$ contracts to a smooth point. Because $Q_{2}$ contracts to a smooth point too, $W_{2}$ meets $C_{2}$ in a $-(p+1)$-curve. The divisor $H+V+C_{1}+W_{1}$ contracts to a smooth point in such a way that $H$ contracts last. Thus either \emph{(i)} $W_{1}=[(2)_{s-2},3,(2)_{p-2}]$ if $s\geq 2$ or \emph{(ii)} $W_{1}=[(2)_{p-2}]$ if $s=1$. Since $W_1\neq 0$, we have $p\geq 3$ in case \emph{(ii)}.

Contract $\Gamma_{1}$ and the new $(-1)$-curves in the subsequent images of $F_{0}$ until the image of $\Gamma_2$ is the only $(-1)$-curve in the image of $F_{0}$ which is not a tip. The contracted curves are of multiplicity at least $2$ in $F_{0}$, so the resulting morphism $\phi$ does not touch $D_{h}$. Then $\phi_{*}F_{0}$ has three components of multiplicity one, so Lemma \ref{lem:singular_P1-fibers}($\mathrm{b}_{1}$) implies that $\phi_{*}F_{0}$ does contain a $(-1)$-tip. It follows from Lemma \ref{lem:singular_P1-fibers}(c) that $\phi_{*}F_{0}$ is of the form $[1,(2)_{u},3,(2)_{\gamma-1},1,\gamma+1]$. Depending on whether $u>0$ or not, we infer that there are the following possibilities for $F_{0}$:
\begin{equation*}\begin{split}
\mbox{case (i): either } & F_{0}=[(2)_{p-2},3,2_{s-2}, 1 , s,p+1, (2)_{u},3,(2)_{\gamma-1},1,\gamma+1] \\
\mbox{or } &                   F_{0}=[(2)_{p-2},3,2_{s-2}, 1 , s,p+2, (2)_{\gamma-1},1,\gamma+1],
\end{split}\end{equation*}
\begin{equation*}
\begin{split}
\mbox{case (ii): either } & F_{0}=[(2)_{p-2}, 1 , p,(2)_{u},3,(2)_{\gamma-1},1,\gamma+1] \\
\mbox{or } & F_{0}=[(2)_{p-2},1,p+1,(2)_{\gamma-1},1,\gamma+1].
\end{split}\end{equation*}
As noted before, $W_{2}$ meets $C_{2}$ in a $-(p+1)$-curve. In the second and third case this implies that $s=p+1$ and we check easily that $Q_{2}=V_2+C_2+W_2$ does not contract to a smooth point. Hence the first or the last case holds. The contractibility of $Q_{2}$ implies that in the first case the chain $[s,1,(2)_u,3,(2)_{\gamma-1}]$ contracts to a smooth point, so $u=s-2$. This is the HN-type $\cB(\gamma,p,s)$ for $\gamma, s\geq 2$, see Fig.\ \ref{fig:B}. The last case corresponds to HN-type $\cA(\gamma,p-1,1)$ with $\gamma\geq 2$.

\vspace{0.5em}
\textbf{Case $U=0$, $V_{1}-V\neq 0$.}
\vspace{0.5em}

We have now $\beta_{Q_{1}}(H)=2$ and $(F_0)\redd\cap D_v=W_1+W_2$. Since $V$ is a branching component of $Q_1$, $V+C_{1}+W_{1}$ contracts to a smooth point and $V$ contracts last. Then $W_{1}=[(2)_{p-1}]$, so $F_{0}$ is of the form $[(2)_{p-1},1,p+1,(2)_{r},1]$ for some $r\geq 0$ and $W_{2}=[p+1,(2)_{r}]$. If $W_{2}$ does not meet $C_{2}$ in a tip then $C_{2}+V_{2}$ contracts to a smooth point, so $V_{2}$ is a $(-2)$-chain and hence $V_{1}=V_{2}^{*}$ is irreducible, which is not the case. Hence $W_{2}$ intersects $C_{2}$ in a tip.

\begin{figure}[htbp]
	\centering \hspace{-1.5cm}
	\begin{minipage}{0.43\textwidth}
		\centering
		\includegraphics[scale=0.4]{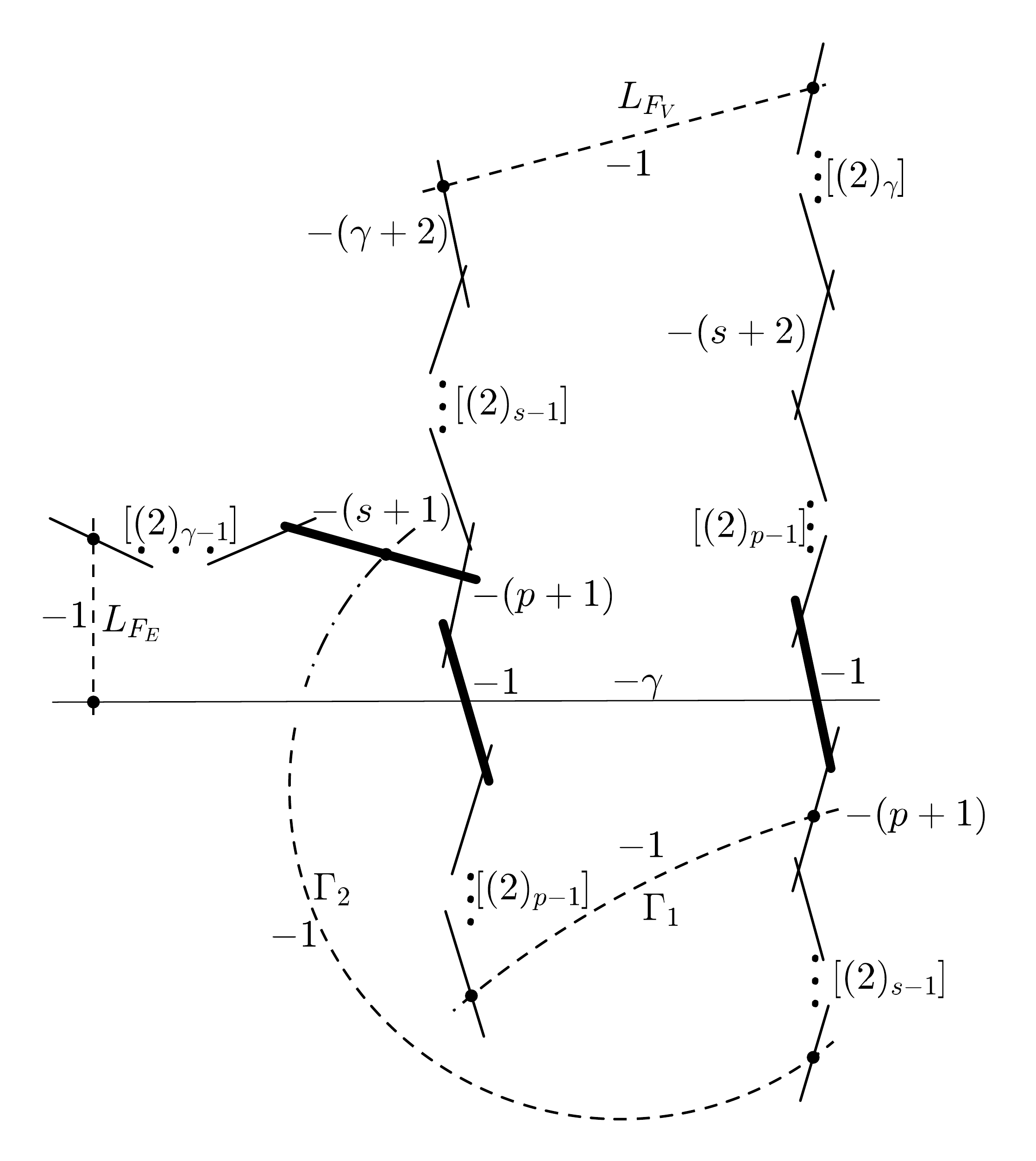}
		\caption{Type $\cC(\gamma,p,s)$.}
		\label{fig:C}
	\end{minipage}\hspace{1.5cm}
	\begin{minipage}{0.43\textwidth}
		\centering
		\includegraphics[scale=0.4]{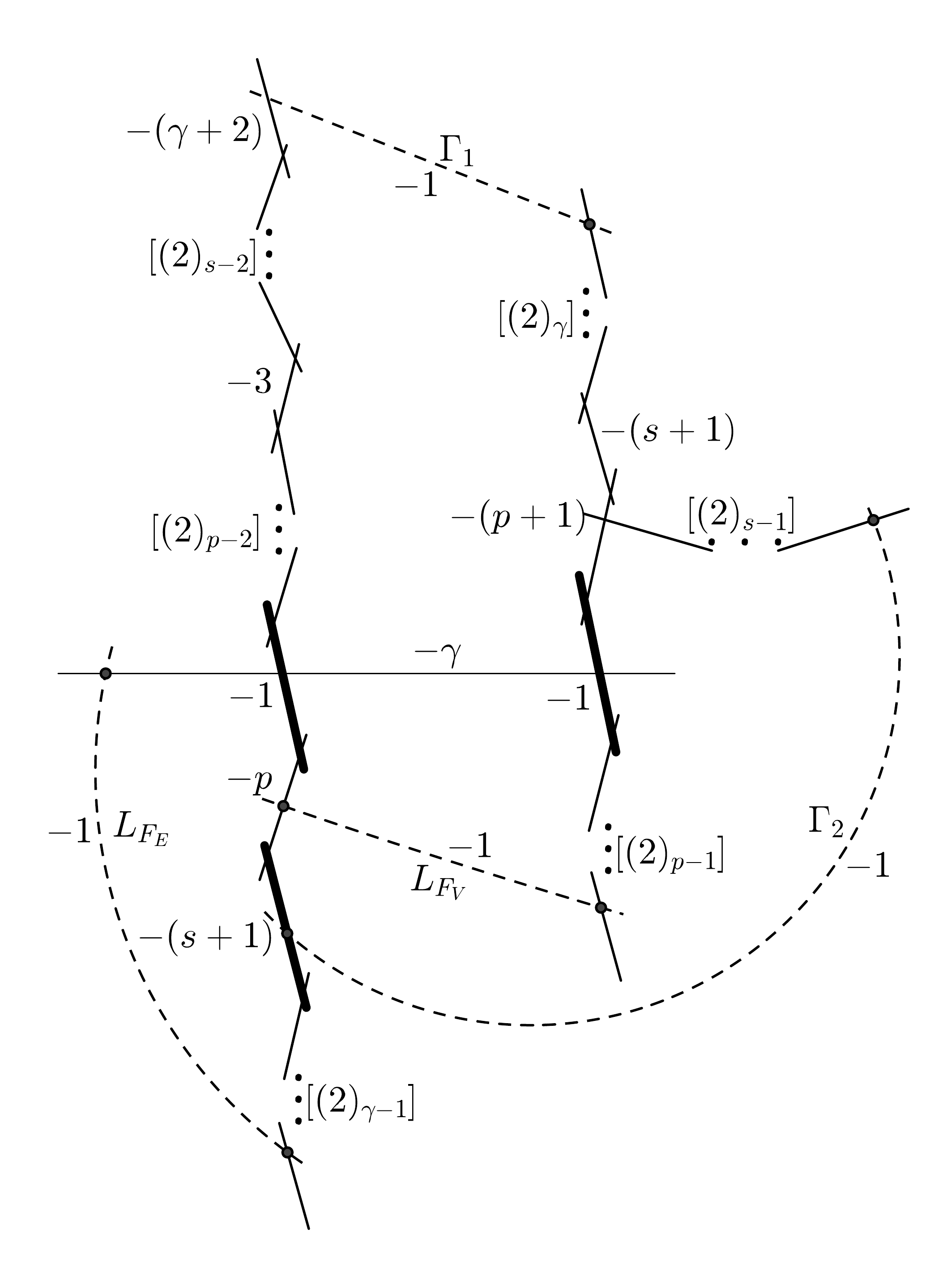}
		\caption{Type $\cD(\gamma,p,s)$.}
		\label{fig:D}
	\end{minipage}
\end{figure}

Recall that the component of $R$ meeting $L_{F_E}$ is not the last component of $Q_{1}$. Then the contractibility of $Q_1$ implies that $V_{1}=[(2)_{u},\gamma+2,(2)_{s-1},p+1]$ for some $u\geq 0$, hence $V_{2}=V_1^*=[u+2,(2)_{\gamma-1},s+2,(2)_{p-1}]$. In particular, the component of $V_2$ meeting $C_2$ is a $(-2)$-curve, because $p\geq 2$. The contractibility of $Q_2$ implies that the component of $W_2$ meeting $C_1$ is not a $(-2)$-curve, hence it is a $-(p+1)$-curve and then $r=s-1$ and $u=0$. This is the HN-type $\cC(\gamma,p,s)$ with $\gamma\geq 2$, see Fig.\ \ref{fig:C}.
\end{proof}

\vspace{0.5em}
\textbf{Case $U=0$, $V_{1}=V$.}
\vspace{0.5em}

In this case $Q_{1}$ is a chain. As in the proof of the Claim \ref{cl:order-F_0}, the fact that $Q_{1}$ contracts to a smooth point and $R$ is not contracted last implies that $W_{1}=[(2)_{u},\gamma+2,(2)_{s-2},3,(2)_{p-2}]$ if $s\geq 2$ and $W_{1}=[(2)_{u},\gamma+1,(2)_{p-2}]$ if $s=1$. Furthermore, $V_{2}=V_{1}^{*}=[(2)_{p-1}]$. Since $F_0$ contracts to a $0$-curve, we obtain $W_2=[(2)_{r},p+1,s+1,(2)_{\gamma-1},u+2]$ for some $r\geq 0$. Clearly, $C_2$ meets $W_2$ in a component of multiplicity one and these are contained in the initial subchain $[(2)_r,p+1]$. The contractibility of $Q_{2}$ to a smooth point implies that this component is a $-(p+1)$-curve, so the image of $W_{2}$ after the contraction of $C_{2}+V_{2}$ equals $[u+2,(2)_{\gamma-1},s+1,1,(2)_{r}]$ and this chain contracts to a smooth point. The latter is possible only if $r=s-1$ and $u=0$. This is the HN-type $\cD(\gamma,p,s)$ for $\gamma\geq 2$, see Fig.\ \ref{fig:D}.

\section{Realization of HN-types.} \label{sec:existence_of_fibrations}

In this section we show that each of the HN-types listed in Theorem \ref{thm:possible_HN-types} is realized by a rational, cuspidal planar curve which is unique up to a projective equivalence. We also prove that the complements of those curves are $\C^{**}$-fibered surfaces of log general type. 

The fact that curves of different HN-types in Theorem \ref{thm:possible_HN-types} are not projectively equivalent can be seen by a direct comparison either of standard HN-types (which are unique) or of multiplicity sequences or of weighted graphs of minimal log resolutions.

\subsection{Existence of $\C^{**}$-fibrations.}\label{ssec:existence_of_fibrations}
  
\begin{prop}\label{prop:Cstst_fibr_exists_and_k=2} Let $\bar{E}\subseteq \P^{2}$ be a rational cuspidal curve of one of the HN-types listed in Theorem \ref{thm:possible_HN-types}. Then $\P^{2}\setminus \bar{E}$ admits a $\C^{**}$-fibration and is of log general type. \end{prop}

We first prove (Lemma \ref{lem:log_general_type}) that $\P^2\setminus \bar E$ is of log general type. As it was noticed in Lemma \ref{lem:kappa<=1}, this could be deduced from the available classification of cuspidal curves with $\kappa\neq 2$. However, we provide a direct argument. We have the following analogue of Proposition \ref{prop:some_Cstst_fibration_extends}:

\begin{prop}[Reduction to $\C^{*}$-fibrations with no base points]\label{prop:some_Cst_fibration_extends}
Let $(X,D)$ be a smooth snc-minimal completion of a smooth $\Q$-acyclic surface of non-negative Kodaira dimension. If $X\setminus D$ has a $\C^{*}$-fibration then it has a $\C^{*}$-fibration  without base points on $X$.
\end{prop}
\begin{proof}
Let $\rho$ be a $\C^{*}$-fibration of $X\setminus D$ and let $\tau\colon (X\s,D\s)\to (X,D)$ be the minimal resolution of base points of $\rho$ on $X$. Denote by $\bar \rho$ the completion of $\rho$ and suppose $\tau \neq \id$. By the minimality of $\tau$, its last exceptional curve $H$ is horizontal for $\bar \rho$. Since $D$ is snc, $\beta_{D\s}(H)\leq 2$.

Because $\kappa(K_{X\s}+D\s)=\kappa(X\setminus D)\geq 0$, $K_{X\s}+D\s$ is linearly equivalent, as a $\Q$-divisor, to some effective $\Q$-divisor. For a general fiber $f$ of $\bar{\rho}$ we have $(K_{X\s}+D\s)\cdot f=0$ by the adjunction formula, so this divisor is vertical, hence it can be written as $\alpha f+N$, where $\alpha$ is rational and non-negative and $N$ is an effective vertical $\Q$-divisor with negative-definite intersection matrix. It follows that $\kappa(X\setminus D)=\kappa(\alpha f)$. Since $H$ is non-branching, we have $\alpha\leq H\cdot (K_{X\s}+D\s)=-2+\beta_{D\s}(H)\leq 0$, hence $\alpha=0$ and $\beta_{D\s}(H)=2$. In particular, $\kappa(X\setminus D)=0$. (Note that for now we did not use the fact that $X\setminus D$ is $\Q$-acyclic).

The fibration $\rho$ is necessarily twisted, that is $\#D_{h}\s=1$. Indeed, otherwise by  \cite[4.10]{Palka-classification2_Qhp} at least three fibers of $\bar \rho$ meet $H$ in a component of $D_{v}\s$, contrary to the fact that $\beta_{D\s}(H)\leq 2$. Thus we are left with the case when all $\C^{*}$-fibrations of $X\setminus D$ are twisted. By \cite[6.1(4)]{Palka-classification2_Qhp} this is possible only if there is an snc-minimal smooth completion $(X',D')$ of $X\setminus D$ such that $D'=T_{1}+T_{2}$, where $T_{1}=[2,1,2]$, $T_{2}=[2,k,2]$ for some $k\geq 1$ and these chains meet once in their middle components. Because all components of $D'$ have negative self-intersection numbers, the snc-minimal smooth completion is unique, so $(X,D)\cong (X',D')$. Now $T_{1}$ supports a fiber of a $\P^{1}$-fibration of $X$ which restricts to a (twisted) $\C^{*}$-fibration of $X\setminus D$ with no base points on $D$, which finishes the proof.  
\end{proof}

\begin{lem}[Complements are of log general type]\label{lem:log_general_type}
Let $\bar{E}\subseteq \P^{2}$ be a rational cuspidal curve of one of the HN-types listed in Theorem \ref{thm:possible_HN-types}. Then $\kappa(\P^{2}\setminus \bar{E})=2$.
\end{lem}

\begin{proof}
Note first that since all components of $D$ have negative self-intersection numbers, up to an isomorphism $(X,D)$ is the unique minimal snc-completion of $X\setminus D$ and hence all boundary components in all snc completions of $X\setminus D$ have negative self-intersection numbers.

If $\kappa(\P^{2}\setminus \bar{E})=-\infty$ then by \cite[3.4.3.1]{Miyan-OpenSurf} $\P^{2}\setminus \bar{E}$ admits a $\C^{1}$-fibration $\rho$ over $\C^1$. But the fiber at infinity of a minimal completion of such a fibration is a smooth $(0)$-curve, which contradicts the claim above. So we have $\kappa(X\setminus D)\in \{0,1\}$. 

By Lemma \ref{lem:kappa<=1} there is a $\C^*$-fibration $\rho$ of $\P^{2}\setminus \bar{E}$, and by Proposition \ref{prop:some_Cst_fibration_extends} we may assume that it has no base points on $X$. Denote by $\bar \rho\:X\to \P^1$ the completion of $\rho$. Suppose $D$ contains a fiber $F_\8$ of $\bar \rho$. Since $D$ contains no $0$-curve, \cite[7.5]{Fujita-noncomplete_surfaces} implies that $\rho$ is twisted, $F_\8=[2,1,2]$ and the horizontal component of $D$ meets the fiber once in the middle curve $D_0$, which is a branching $(-1)$-curve in $D$. Because $[2]$ is one of the connected components of $D-D_0$, we have $D_0\neq E$, hence $D_0=C_1$. Since $F_\8$ is not negative definite, it is not contained in $D-E$, so $E\subset F_\8$. It follows that $c=1$, so $\bar E\subset \P^2$ is of HN-type $\ORa$ or $\ORb$ (see Fig.\ \ref{fig:ORa}-\ref{fig:ORb}). But $C_1=D_0$ meets two $(-2)$-tips of $D$ contained in $F_\8$, which is false for the latter HN-types; a contradiction.

\begin{figure}[!htbp]
	\centering
	\includegraphics[scale=0.45]{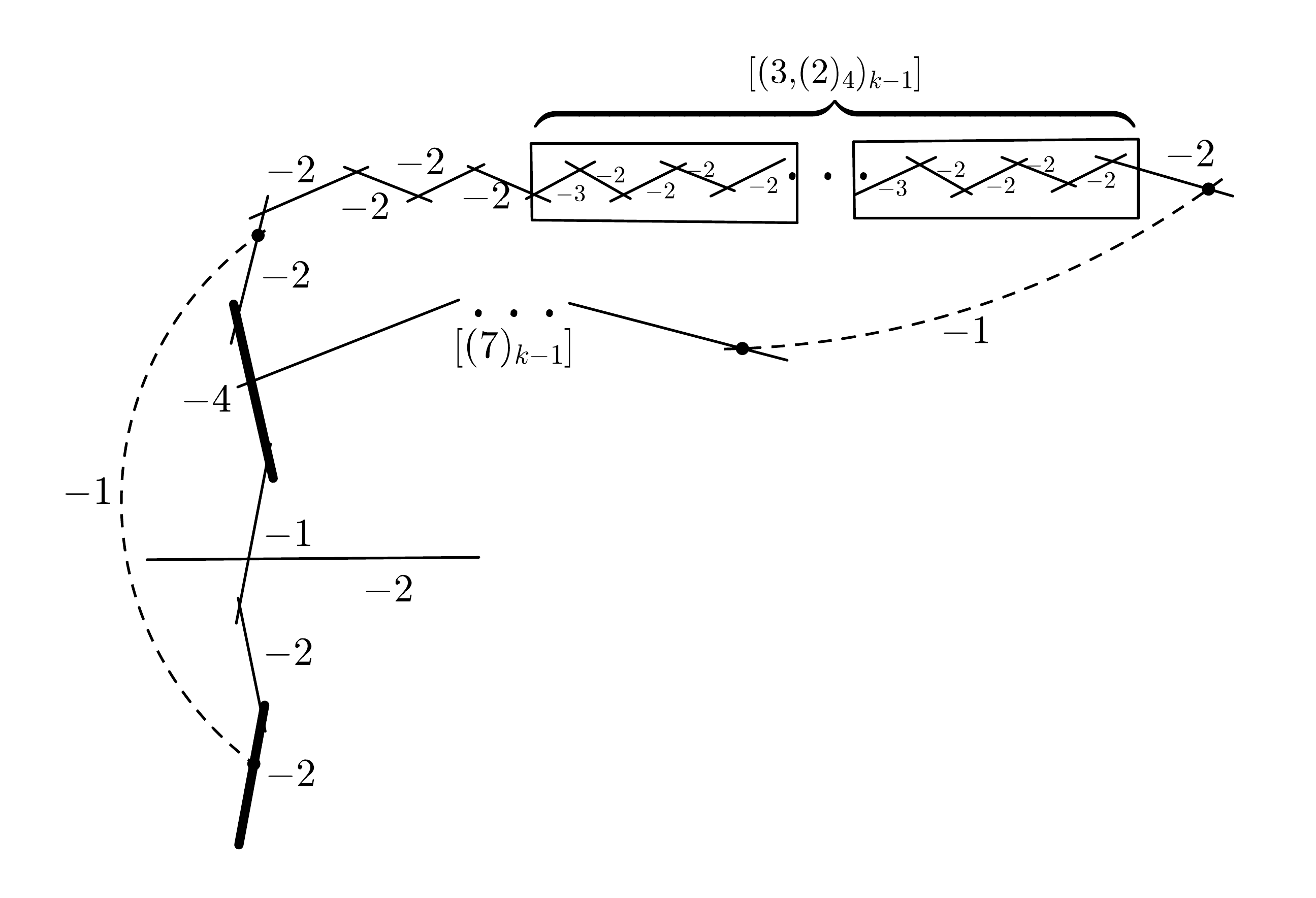}
	\caption{Type $\ORa(k)$.}
	\label{fig:ORa}
\end{figure}
\begin{figure}[!htbp]
	\centering
	\includegraphics[scale=0.45]{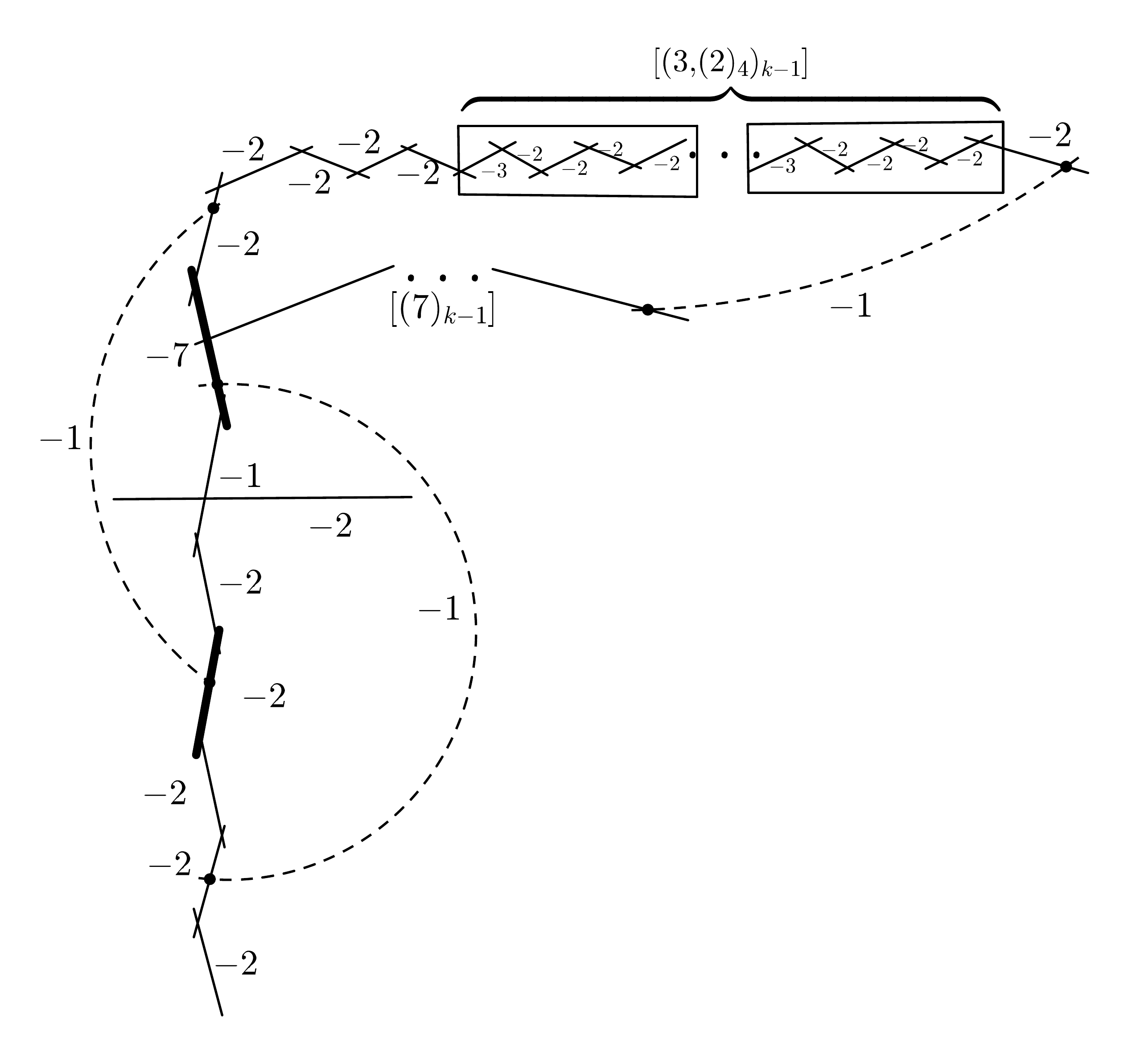}
	\caption{Type $\ORb(k)$.}
	\label{fig:ORb}
\end{figure}

Thus $D$ contains no fiber of $\bar \rho$. By Lemma \ref{lem:fibrations-Sigma-chi} the horizontal part of $D$ consists of two $1$-sections, $H_1$ and $H_2$, and every fiber of $\bar \rho$ has exactly one component not contained in $D$. Since $\kappa(X\setminus D)\geq 0$, by \cite[4.9(C)]{Palka-classification2_Qhp} $\beta_D(H_1)=\beta_D(H_2)>2$ (note that in this case $\beta_D(H_j)=1+n$, where $n$ is the number of 'columnar' fibers; see loc.\ cit.). Because all components of $D-E$ have $\beta_D\leq 3$, we get $\beta_D(H_1)=\beta_D(H_2)=3$. Suppose some $C_j$ is vertical. Since it is a branching $(-1)$-curve, at least one of the components of $D-C_j$ meeting it is horizontal. Then $\mu(C_j)=1$, so $C_j$ is a tip of the fiber containing it, hence $C_j\cdot H_1=C_j\cdot H_2=1$. Since $D$ contains no loop, there is at most one vertical $C_j$ and since all $C_j$'s are disjoint, we get in fact $c=1$. However, because at least one of the components of $Q_1$ meeting $C_1$ is non-branching, we have $H_1=E$ and hence $c=\beta_D(H_1)=3$; a contradiction.

Therefore, $C_1+\ldots+C_c$ is horizontal, and hence $c\leq 2$. Since $\beta_D(E)=c$, we infer that $E$ is vertical. Let $F_E$ be the fiber containing $E$. Assume $c=2$. Then $C_1+C_2$ is horizontal, and does not meet $D_v-E$. Due to the connectedness of $D$ we obtain $D_v\cap F_E=E$, so $F_E=E+L=[1,1]$ for some $(-1)$-curve $L$. But $L$ does not meet $D-E$, so its image on $\P^2$ is a $(-1)$-curve; a contradiction. Hence $c=1$ and the HN-type of $\bar E$ is one of $\ORa$, $\ORb$. For these HN-types $E^2=-2$ and there are at most two branching components of $D$, namely $C_1$ and one of the components of $D-E$ meeting it, call it $B$. The horizontal components of $D$ are branching, so $C_1+B$ is horizontal. Since $D$ is snc and connected, the fiber containing the point $C_1\cap B$ is smooth. By \cite[7.6(2)]{Fujita-noncomplete_surfaces} the reduction of $F_E$ is $[E,1,E^{*}]$. Since $E^2=-2$, we have $E^*=[2]$, so B meets a $(-2)$-tip. But the latter is false for HN-types $\ORa$, $\ORb$; a contradiction.
\end{proof}
	
In order to prove Proposition \ref{prop:Cstst_fibr_exists_and_k=2}, it remains to show the following
\begin{lem}[Complements are $\C^{**}$-fibered]\label{lem:existence_of_fibrations}
Let $\bar{E}\subseteq \P^{2}$ be a rational cuspidal curve of one of the HN-types listed in Theorem \ref{thm:possible_HN-types}. Then $\P^{2}\setminus \bar{E}$ is $\C^{**}$-fibered.
\end{lem}	

\begin{proof}
If $\bar{E}\subseteq \P^{2}$ is of HN-type $\FZa$ then the multiplicity of $q_{2}\in \bar{E}$ equals $\deg \bar{E}-2$, so the pencil of lines through $q_{2}$ gives a $\C^{**}$-fibration of $\P^{2}\setminus \bar{E}$.

Assume that $\bar{E}\subseteq \P^{2}$ is of one of the HN-types $\cA$ - $\cD$. Using Lemma \ref{lem:Hn_gives_multiplicities} we see that the multiplicity sequence of $q_{1}\in \bar{E}$ is of the form $(\mu,(\mu')_{\gamma},\dots)$ with $\mu'=\deg \bar E-\mu-1<\mu$. Because $\kappa(\P^{2}\setminus \bar{E})=2$ by Lemma \ref{lem:log_general_type}, Lemma \ref{lem:Qhp_has_no_lines} implies that the line $\ell$ tangent to $\bar{E}$ at $q_{1}$ meets $\bar{E}\setminus \{q_{1}\}$ at least once, hence $\ell\cdot \bar{E}\geq \mu+\mu'+1=\deg \bar{E}$, so by the Bezout theorem we have an equality (cf.\ Section \ref{ssec:AD}). We check that $\mu\geq \gamma\mu'$, which implies that after $\gamma+1$ blowups over $q_{1}\in \bar E$ the reduced total transform of $\ell$ is a chain $[\gamma+1,1,(2)_{\gamma-1},1]$ (the last component is the proper transform of $\ell$), the proper transform of $\bar{E}$ meets the second component of this chain and it does not meet the third one. Hence there is a twig $T=[(2)_{\gamma-1}]$ of $Q_{1}$ such that the proper transform $L$ of $\ell$ on $X$ meets $D$ only in $\ftip{T}$ and $E$, each once and transversally. The chain $E+L+T=[\gamma,1,(2)_{\gamma-1}]$ supports a fiber of a $\P^{1}$-fibration of $X$. It is met by the remaining part of $D$ once in $\ltip{T}$ and twice in $E$ ($\bar E$ is bicuspidal), so the $\P^{1}$-fibration restricts to a $\C^{**}$-fibration of $X\setminus D=\P^{2}\setminus \bar{E}$.

Assume now that $\bar{E}\subseteq \P^{2}$ is of HN-type $\cE$ or $\cF$. The last HN-pair of $q_{1}\in \bar{E}$ is $\binom{2}{1}$, which by Lemma \ref{lem:HN_for_chains} means that there is a twig $T=[2]$ of $Q_{1}$ meeting the remaining part of $D$ in $C_{1}$. Since $\gamma=2$, the linear system $|T+2C_{1}+E|$ induces a $\P^{1}$-fibration of $X$, which restricts to a $\C^{**}$-fibration of $\P^{2}\setminus \bar{E}$.

For HN-type $\cG$ the multiplicity sequence of $q_{1}\in \bar{E}$ equals $(\gamma-1)_{4}$ and we have $\deg \bar{E}=2\gamma-1$, so the line $\ell$ tangent to $\bar{E}$ at $q_{1}$ satisfies $(\ell\cdot \bar{E})_{q_{1}}=2(\gamma-1)$ and meets $\bar{E}\setminus \{q_{1}\}$ in one point, transversally. Denote by $L$ be the proper transform of $\ell$ on $X$. We have $Q_{1}=[(2)_{3},\gamma,1,(2)_{\gamma-2}]$. Denote by $T_{i}$ the $i$-th component of $Q_{1}$. Since $(\ell\cdot \bar{E})_{q_{1}}=2(\gamma-1)$, $L$ is a $(-1)$-curve meeting $Q_{1}$ only in $T_{2}$, transversally. Then $|2L+2T_{2}+T_{1}+T_{3}|$ supports a fiber of a $\P^{1}$-fibration of $X$. Since $L$ meets $E$ once and transversally, this $\P^{1}$-fibration restricts to a $\C^{**}$-fibration of $\P^{2}\setminus \bar{E}$.

Finally, assume that $\bar{E}\subseteq \P^{2}$ is of HN-type $\ORa$ or $\ORb$. Then $E^{2}=-2$. Because $Q_{1}$ contracts to a smooth point, there is a $(-2)$-curve $C$ in $Q_{1}$ meeting $C_{1}$ and satisfying $\beta_{D}(C)\leq 2$. The linear system $|C+2C_{1}+E|$ induces a $\P^{1}$-fibration of $X$, which restricts to a $\C^{(n*)}$-fibration of $\P^{2}\setminus \bar{E}$, where $n=\beta_{D}(C)\leq 2$. Because $\kappa(\P^{2}\setminus \bar{E})=2$ by Lemma \ref{lem:log_general_type}, the Iitaka Easy Addition theorem implies that we actually have an equality, which finishes the proof.
\end{proof}

To prove Theorem \ref{thm:possible_HN-types} it remains to show the following Proposition.
\begin{prop}\label{prop:existence_and_uniqueness}
	Each HN-type from the list in Theorem \ref{thm:possible_HN-types} is realized by a rational cuspidal curve which is unique up to a projective equivalence.
\end{prop}

The existence and uniqueness of the Flenner-Zaidenberg curves of the first kind, $\FZa(d,k)$, $d\geq 4$, $k\geq 1$, is shown in \cite[3.5]{FLZa-_class_of_cusp} by a direct computation of power series. The Orevkov curves $\ORa(k), \ORb(k)$, $k\geq 1$ are constructed in \cite[\S 6]{OrevkovCurves} by an inductive application of some (uniquely determined) Cremona transformations and \cite[Lemma 15]{Tono-on_Orevkov_curves} shows that any curve of HN-type $\ORa(k)$, $\ORb(k)$ can be constructed that way. It is easy to see from this construction that such curves are unique up to a projective equivalence.

Therefore, it remains to show the existence and uniqueness of curves of HN-types $\cA$ - $\cD$; $\cE$ - $\cF$ and $\cG$. This is done in Lemmas \ref{lem:type_AD}, \ref{lem:type_EF} and \ref{lem:type_G}, respectively.

\subsection{HN-types $\cA$ - $\cD$: closures of $\C^*$-embeddings with a good asymptote.}
\label{ssec:AD}

Let $\bar{E}\subseteq \P^2$ be a rational cuspidal curve of one of the HN-types $\cA$ - $\cD$. The multiplicities of the cusps of $\bar{E}$ add up to $\deg\bar{E}$, so by the Bezout theorem the line $\ell$ joining them does not meet $\bar{E}\setminus \ell$. Therefore, $\bar{E}\setminus \ell$ is the image of some smooth embedding $\C^{*}\into \C^{2}$. 

Denoting by $t_1$ the line tangent to $\bar{E}$ at $q_{1}$ we see that for the HN-types $\cA$ - $\cD$ this affine line is a so-called \emph{good asymptote} for this embedding, that is, $t_1\setminus \ell$ is isomorphic to $\C^1$ and meets $\bar{E}\setminus \ell$ at most once and transversally. Indeed, since the multiplicity sequence of $q_{1}$ is of the form $(\mu,\mu',\dots)$ with $\mu+\mu'=\deg\bar{E}-1$, the Bezout theorem implies that the latter holds. Such $\C^*$-embeddings into $\C^2$ have been classified in \cite[Thm.\ 8.2]{CKR-Cstar_good_asymptote}. In fact only part (ii) of the latter theorem is of interest for us, because the embeddings from part (i) either have only one place at infinity, and hence their closure on $\P^2$ is not cuspidal (8.2(i.2) for $a>b(k+1)$, $b>1$, see 6.8.1.3 loc.\ cit.) or they have two places at infinity but they admit a so-called \emph{very good asymptote}, which is an affine line contained in $\P^2\setminus \bar E$, and hence by \ref{lem:Qhp_has_no_lines} $\kappa(\P^{2}\setminus\bar{E})\leq 1$ (remaining cases of 8.2(i), see 6.8 and 7.4 loc.\ cit.). The authors give equations and sequences of HN-pairs computed with respect to the line at infinity, that is, the line gives the starting smooth germ $Z_1$, see Sec.\ \ref{sec:cusps}. As always, the multiplicity sequences can be computed using \eqref{lem:Hn_gives_multiplicities}. Those pairs can be also easily converted into HN-pairs in a standard form using \eqref{eq:HN-equivalence}. We describe the results. To avoid a conflict of notation we denote the parameters $s$ and $p$ from \cite[Thm.\ 8.2(ii)]{CKR-Cstar_good_asymptote} by $s'$ and $p'$ respectively.
\begin{enumerate}
	\item[$\cA(\gamma,p,s)$] By 6.9.1 loc.\ cit.\ this HN-type is realized by closures of embeddings $(\mathbf{gga+})(k=\gamma,p'=p,s'=s)$ given by 8.2(ii.1).
	
	\item[$\cB(\gamma,p,s)$] By 6.9.2 loc.\ cit, this HN-type is realized by closures of embeddings $(\mathbf{gga-})(k=\gamma,p'=p,s'=s)$ given by 8.2(ii.2).
	
	\item[$\cC(\gamma,p,s)$] If $s\geq 2$ then by 6.10.1 loc.\ cit.\ this HN-type is realized by closures of embeddings $(\mathbf{bga+})(k=\gamma+1,p'=p,s'=s-1)$ given by 8.2(ii.4). The HN-type $\cC(\gamma,p,1)$ is by 6.9.2 loc.\ cit realized by closures of embeddings $(\mathbf{gga-})(k=\gamma+1,p'=1,s'=p+1)$ given by 8.2(ii.2).
	
	\item[$\cD(\gamma,p,s)$] If $s\geq 2$ then by 6.10.2 loc.\ cit.\ this HN-type is realized by closures of embeddings
	$(\mathbf{bga-})(k=\gamma+1,p'=p,s'=s-1)$  given by 8.2(ii.5). The HN-type $\cD(\gamma,p,1)$ for $p\geq 3$ is by 6.10.1 loc.\ cit.\ realized by closures of embeddings $(\mathbf{bga+}) (k=\gamma+2, p'=1, s'=p-2)$ given by 8.2(ii.4). The HN-type $\cD(\gamma,2,1)$, again by 6.10.2 loc.\ cit, is realized by closures of embeddings $(\mathbf{bga-})(k=\gamma+2,p'=1,s'=1)$.
\end{enumerate}

This comparison shows that curves of each HN-type $\cA$ - $\cD$ do exist. 

\begin{lem}[Uniqueness for HN-types $\cA$ - $\cD$]\label{lem:type_AD} Up to a projective equivalence, for every admissible choice of parameters $\gamma$, $p$, $s$ as in Theorem \ref{thm:possible_HN-types} there exists a unique rational cuspidal curve of each of the HN-types $\cA(\gamma,p,s)-\cD(\gamma,p,s)$.
\end{lem}
\begin{proof}	
	Let $\bar{E}_j$, $j=1,2$ be rational cuspidal curves of the same HN-type, one of $\cA$ - $\cD$. Let $q_{1,j}$, $q_{2,j}$ be the cusps of $\bar{E}_j$. Recall  that their multiplicities add up to $\deg\bar{E_j}$, so by the Bezout theorem the line $\ell_j$ joining them does not meet $\bar{E_j}\setminus \ell_j$ and we have $(\ell_j\cdot \bar{E_j})_{q_{i,j}}=\mu(q_{i,j})$, $i,j\in \{1,2\}$. In particular, $\bar{E}_j\setminus \ell_j$ is the image of one of the $\C^*$-embeddings listed above. Thus \cite[8.2.(iii)]{CKR-Cstar_good_asymptote} implies that we may choose coordinates on $\C^2=\P^{2}\setminus \ell_j$ so that $\bar{E}_j\setminus \ell_j$ is given by one of the equations listed in loc.\ cit. Moreover, closures of different cases on those lists have different HN-types, so we have an isomorphism
	\begin{equation*}
	\psi\colon (\P^{2}\setminus \ell_1, \bar{E}_1\setminus \ell_1)\to (\P^{2}\setminus \ell_2, \bar{E}_2\setminus \ell_2).
	\end{equation*}
	
	Let $\pi_{j}\colon (X_{j},D_{j})\to (\P^{2},\bar{E}_{j})$ be the minimal log resolution and let $L_{j}$ and $E_{j}$ be the proper transforms on $X_j$ of $\ell_j$ and $\bar{E}_{j}$ respectively. The weighted dual graphs of $D_{1}$ and $D_{2}$, being determined by the HN-types of $\bar E_1$ and $\bar E_2$, are isomorphic and this isomorphism identifies vertices corresponding to $E_{1}$ and $E_{2}$. Moreover, since $\ell_j$ meets $\bar{E}_{j}$ only in cusps with the least possible multiplicity, $L_{j}$ meets $D_{j}$ only in the first exceptional curve over each cusps, transversally, and hence it is a $(-1)$-curve. In particular, $D_{j}+L_{j}$ is snc and the weighted dual graphs of $D_{1}+L_{1}$ and $D_{2}+L_{2}$ are isomorphic via an isomorphism mapping $L_{1}$ and $E_{1}$ to $L_{2}$ and $E_{2}$ respectively. 
	
	Let $\phi_{j}\colon X_{j}\to X_{j}'$ be some snc-minimization of $D_{j}+L_{j}$. Put $D_{j}'=(\phi_{j})_{*}(D_{j}+L_{j})$. By definition $\phi_j$ contracts successively non-branching $(-1)$-curves in $D_{j}+L_{j}$ and its images. The only non-branching $(-1)$-curves in $D_{j}+L_{j}$ are $L_{j}$ and possibly $E_{j}$, so we can assume that the latter is contracted first, in case it is a $(-1)$-curve. The curves $L_{j}$ and $E_{j}$ are not tips of $D_{j}+L_{j}$, so each blowdown within $\phi_{j}$ is inner with respect to the image of $D_{j}+L_{j}$. Suppose that after some number of blowdowns within $\phi_{j}$ the image of $D_{j}$ contains two non-branching $(-1)$-curves $\Gamma$, $\Gamma'$ which are not disjoint. Because $E_{j}$ meets only (two) branching $(-1)$-curves in $D_{j}$, the contraction of $E_{j}$ produces no such, so $\Gamma$ and $\Gamma'$ meet at the image of $L_{j}$. Then $|\Gamma+\Gamma'|$ induces a $\P^{1}$-fibration whose pullback to $X$ restricts to a $\C^{1}$- or a $\C^{*}$-fibration of $\P^{2}\setminus \bar{E}$. But then by Iitaka's Easy Addition Theorem the latter surface is not of log general type; a contradiction. 
	
	It follows that $\varphi_j$ is unique, determined by the weighted dual graph of $D_j$ and that in $D_j'$ there are no non-branching components $C$ with $C^2\geq 0$. On the other hand, if the base locus of $\psi$, treated as a map from $X_1'$ to $X_2'$, is nonempty, then denoting by $C$ the last exceptional curve in a minimal resolution of base points of $\psi$ we observe that by minimality the image of $C$ on $X_2'$ is a non-branching curve with non-negative self-intersection. Thus $\psi$ has no base points on $X_1'$ and hence gives an isomorphism of pairs $$\psi\: (X_{1}',D_{1}')\to (X_{2}',D_{2}').$$ The universal property of blowing up implies that $\psi$ lifts to an isomorphism of pairs $\psi\:(X_{1},D_{1}+L_{1})\to (X_{2},D_{2}+L_{2})$ which maps $L_{1}$ and $E_1$ to $L_{2}$ and $E_2$ respectively. Therefore, $\psi$ descends to an isomorphism of pairs $(\P^{2},\bar{E}_{1}+\ell)\to (\P^{2},\bar{E}_{2}+\ell)$, which means that $\bar{E}_{1}$ and $\bar{E}_{2}$ are projectively equivalent.
\end{proof}

\begin{rem}[Remaining cases from \cite{CKR-Cstar_good_asymptote}]
	The case \cite[Thm.\ 8.2(ii.3)]{CKR-Cstar_good_asymptote} does not appear in the above comparison, despite the fact that its closure is a bicuspidal curve with a complement of log general type. This is because in this exceptional case the complement of the cuspidal curve is not $\C^{**}$-fibered.
	
	An attentive reader will note also that the domain of parameters in \cite[Thm.\ 8.2(ii)]{CKR-Cstar_good_asymptote} is slightly bigger than the domain for corresponding cases in our Theorem \ref{thm:possible_HN-types}. Firstly, for curves $\cA$ - $\cD$ we do not allow $\gamma=0$, because the complement is then not of log-general type (cf. \cite[Thm.\ 7.5]{CKR-Cstar_good_asymptote}). Secondly, some of the cases listed in \cite[Thm.\ 8.2(ii)]{CKR-Cstar_good_asymptote} are not really distinct and our choice of parameters solves this problem. The differences are as follows. We do not allow $p=1$ because such HN-types have at the end one pair $\binom{1}{1}$ describing a blowup which is not a part of a minimal log resolution. The curve $\bar{E}$ obtained this way is of HN-type $\cC(\gamma-1,s,1)$ 
	in case $\cA$, of HN-type $\cC(\gamma-1,s-1,1)$ in case $\cB$, of HN-type $\cD(\gamma-1,s+1,1)$ in case $\cC$ and of HN-type $\cD(\gamma-1,s,1)$ in case $\cD$. We do not allow cases $(\gamma,p)=(1,2)$ for HN-types $\cA$, $\cB$ because of $\cA(1,2,s)$ is $\cD(1,2,s)$ and $\cB(1,2,s)$ is $\cC(1,2,s-1)$. Finally, we do not allow $s=1$ for the HN-type $\cB$, because $\cB(\gamma,p,1)$ is $\cA(\gamma,p-1,1)$. 
\end{rem}
	
\subsection{HN-types $\cE$ - $\cF$: closures of sporadic smooth $\C^*$-embeddings.}\label{ssec:EF}

Recall that if $\bar{E}\subseteq \P^{2}$ is of HN-type $\cE$ or $\cF$ then the multiplicities of its cusps add up to $\deg\bar{E}$, so the line $\ell$ joining these cusps does not meet $\bar{E}$ in any other point. Therefore, like in the case of HN-types $\cA$ - $\cD$ considered in the previous section, $\bar{E}\setminus \ell$ is the image of some embedding $\C^{*}\into \C^{2}$. However, such embeddings have a different geometry than the previous ones, namely they do not admit good asymptotes. They appear in the conditional classification \cite{BoZo-annuli}, see Remark \ref{rem:BoZo} and they will be classified in a forthcoming article \cite{KoPa-SporadicCstar2}. To show their existence consider the parametrization
\begin{equation*}
\C^{*}\ni t\mapsto (t^{2n}(t^{2}+t+\tfrac{1}{2}),t^{-2n-4}(t^{2}-t+\tfrac{1}{2}))\in \C^{2},
\end{equation*}
where $n$ is a positive integer. In \cite{BoZo-annuli} this is the rescaled family (s). A computation as in Example \ref{ex:BoZo_HN_computation} shows that the closure of the image of this embedding has HN-type $\cE(n/2)$ if $n$ is even and $\cF((n+1)/2)$ if $n$ is odd (another method is to compute the weighted dual graph of the log resolution and to check that it is as in Fig.\ \ref{fig:E} or \ref{fig:F}).
%

In order to prove their uniqueness we borrow the argument from loc.\ cit. The key step is a construction of a uniquely determined morphism $\psi\colon X\to \P^{2}$ such that $\psi_{*}D$ is a configuration of lines and conics. 
\begin{lem}[Existence and uniqueness of curves of HN-types $\cE$ - $\cF$]\label{lem:type_EF} Up to a projective equivalence, for every $k\geq 1$ there exists a unique curve of HN-type $\cE(k)$ and $\cF(k)$.
\end{lem}
\begin{proof}
The existence part was proved above. Let $\bar{E}\subseteq \P^{2}$ of HN-type $\cE(k)$ or $\cF(k)$ for some $k\geq 1$ and, as before, let $\pi\colon (X,D)\to (\P^{2},\bar{E})$ be a minimal log resolution. Using Lemma \ref{lem:HN_for_chains} one computes that the graph of $D$ is as in Figures \ref{fig:E}-\ref{fig:F}.In particular,  $E^{2}=-2$ (see Lemma \ref{lem:HN-equations}) and the resolution $q_{1}$ of $Q_{1}$ contains a $(-2)$-tip, say, $T$ meeting the $(-1)$ curve $C_{1}$. We now show the existence of $(-1)$-curves $\Gamma_{1}$, $\Gamma_{2}$ and $L_{F_{A}}$ on $X$ meeting $D$ as in those figures. They will be found in some degenerate fibers of a $\P^{1}$-fibration induced by $F_{\infty}\de E+2C_{1}+T$.
	
Recall that the line joining the cusps of $\bar{E}$ meets $\bar{E}$ only at its cusps and with the least possible multiplicity. It follows that the proper transform $\Gamma_{2}$ of this line has the required properties, i.\ e.\ $\Gamma_{2}\cdot D=2$ and $\Gamma_{2}$ meets both divisors $Q_{j}$ in the curve which is contracted last by $\pi$.
	
The horizontal part of $D$ is $D_{h}=H+C_{2}$, where $H$ is a $2$-section and $C_{2}$ is a $1$-section. Let $(D+\Gamma_{2})_{v}$ be the vertical part of $D+\Gamma_{2}$. Denote by $R$ its connected component containing $\Gamma_{2}$ and denote by $F$ the fiber containing $R$. Using the description of the graph of $D$ (see Fig.\ \ref{fig:E}-\ref{fig:F}) we see that $R$ is a chain $[2,(k+2),2,1,3,(2)_{k-1},5,(k+1)]$ if $\bar{E}$ is of HN-type $\cE(k)$; it is a chain $[3,(2)_{k-2},3,1,2,k+1,5,(2)_{k-1}]$ in case $\cF(k)$ with $k\geq 2$ and $[4,1,2,2,5]$ in case $\cF(1)$. It meets $H$ in the first tip and $C_{1}$ in the $(-5)$-curve. Let $\phi$ be the contraction of $\Gamma_{2}$ followed by the snc-minimization of the image of $D$. This map does not touch $D-R$ and $\phi_{*}R$ equals $[2,k+1]$ if $\bar{E}$ is of HN-type $\cE(k)$ and $[3,(2)_{k-1}]$ otherwise. This chain meets both components of $\phi_{*}D_{h}$ in the first tip. It follows that the fiber $\phi_{*}F$ containing it has a $(-1)$-curve $\Gamma$. Because $D_{v}-(F_{\infty})\redd$ contains no such, $\Gamma\not\subseteq \phi_{*}D$. Moreover, by Lemma \ref{lem:fibrations-Sigma-chi}, $\Gamma$ is the unique component of $\phi_{*}F$ not contained in $\phi_{*}D$. Therefore, $\Gamma$ is the unique $(-1)$-curve in $\phi_{*}F$, so $\mu(\Gamma)\geq 2$ by Lemma \ref{lem:singular_P1-fibers}. The fiber $\phi_{*}F$ meets $\phi_{*}D_{h}$ in two components of multiplicity $1$ which lie in different connected components of $(\phi_{*}F)\redd-\phi_{*}\Gamma_{1}$. Thus by Lemma \ref{lem:singular_P1-fibers} $\phi_{*}F$ is a chain meeting $\phi_{*}D_{h}$ in tips. It follows that the proper transform $\Gamma_{1}$ of $\Gamma$ on $X$ has the required properties, i.\ e.\ $\Gamma_{1}\cdot D=2$, it meets $Q_{1}$ in a tip which is contracted second-to-last by $\pi$, and it meets $Q_{2}$ in the $(-5)$-curve in case $\cF(1)$ and in the second-to-last tip otherwise.

Now denote by $A$ the connected component of $D_{v}$ not contained in $F+F_{\infty}$ and by $F_{A}$ the fiber containing it. Then $\sigma(F_{A})=1$ by Lemma \ref{lem:fibrations-Sigma-chi}, so $(F_{A})\redd=A+L_{F_{A}}$, where $L_{F_{A}}$ is the unique $(-1)$-curve in $F_{A}$. Because $A=[2,2,2]$, we see that $F_{A}$ contracts to a $0$-curve if and only if $L_{F_{A}}$ meets $A$ in its middle component. Then $\mu(L_{F_{A}})=2$, so the equality $F_{A}\cdot D_{h}=3$ implies that $L_{F_{A}}$ meets $D_{h}$ only in $H$, transversally. 

Thus we have shown the existence of $\Gamma_{1}$, $\Gamma_{2}$ and $L_{F_{A}}$ as indicated in Fig.\ \ref{fig:E}-\ref{fig:F}. We now use them to find a contraction onto $\P^2$ mentioned in the introduction. Let $A_{2}$ be the tip of $A$ meeting $C_{2}$. Define $\psi\colon X\to X'$ as the contraction of $D-(H+E+C_{2}+A_{2})$ and put $D'\de \psi_{*}D$. Now $X'$ is a smooth surface with Picard rank $\rho(X')=\rho(X)-\#\Exc \psi=\#D-(3+\#D-\#D')=1$, where in the second equality we used the fact that $\rho(X)=\#D$. Thus $X'\cong \P^{2}$. Computing the changes of self-intersection numbers of the components of $D'$ we infer that the images of $F$, $F_{\infty}$ and $F_{A}$ are lines meeting at $\psi(C_{1})$ and $\psi_{*}H$ is a conic which is tangent to both $\psi_{*}F_{\infty}$ and $\psi_{*}F_{A}$ and meets $\psi_{*}F$ in two distinct points, one of which is $\psi(\Gamma_{1})$. Thus to each curve $\bar{E}$ of HN-type $\cE(k)$ or $\cF(k)$, $k\geq 1$, we can associate a pair $(\psi_{*}\bar{E},(\psi(C_{1}),\psi(L_{F_{A}}),\psi(\Gamma_{1})))$ consisting of a conic and a triple of distinct points on it. Such pair is unique up to a projective equivalence.

The morphism $\psi$ consists of the contraction of superfluous $(-1)$-curves in the reduced total transform of $D'$ followed by the contraction of the image of $H$. We can assume that $\psi$ first contracts the components of $F$, then of $F_{\infty}$ and then of $F_{A}$. If the image $F_{0}$ of some fiber contains two $(-1)$-curves which are not disjoint then $F_{0}=[1,1]$, so since $F_{0}'\cdot D_{h}=3$ for the total transform $F_{0}'$ of $F_{0}$ on $X$, one of these $(-1)$-curves is not superfluous. It follows that the center of each blowup constituting $\psi^{-1}$ is uniquely determined. Therefore, the above correspondence shows that up to a projective equivalence for each $k\geq 1$ there exist unique rational cuspidal curves of HN-type $\cE(k)$ and $\cF(k)$.
\end{proof}

\subsection{Type $\cG$ (new bicuspidal curves).} \label{sec:construction_(h)}

Below we construct rational bicuspidal curves of HN-type $\cG$. The construction shows in particular that given a closed $\C^*$-embedding into $\C^2$, by choosing different embeddings $\C^{2}\into \P^{2}$ or, equivalently, different coordinates on $\C^{2}$, we may obtain bicuspidal closures which are not projectively equivalent. 

Our classification results have been presented during the talk of the first author in 2015 in Bangalore, India (conference \emph{Algebraic surfaces and related topics}). In 2016 we were told by M. Zaidenberg that, although the curves of HN-type $\cG$ do not appear in literature, they were known to T.\ tom Dieck, who listed their multiplicity sequences in a private correspondence with H.\ Flenner in 1995. When we were preparing the final version of our article another proof of their existence, but not of uniqueness, was given by Bodnar (who credits A. Némethi) in \cite[3.1]{Bodnar_type_G_and_J}: the family (b) in loc.\ cit.\ is of HN-type $\cG(k-1)$. 

\begin{lem}[Existence and uniqueness of curves of HN-type $\cG$]\label{lem:type_G} Up to a projective equivalence, for every $\gamma\geq 3$ there exists a unique rational cuspidal curve of HN-type $\cG(\gamma)$.
\end{lem}

\begin{proof}
The curves of HN-type $\cG(\gamma)$, $\gamma\geq 3$ are constructed by applying certain quadratic Cremona map to the curves of HN-type $\cD(\gamma-3,2,1)$. By a curve of HN-type $\cD(0,2,1)$ we mean, by definition, a rational bicuspidal curve with cusps $q_1$, $q_2$ having standard HN-pairs $\binom{3}{2}$ and $\binom{5}{2}$ respectively. The existence of such curves for $\gamma>3$ was shown in a previous section. For $\gamma=3$ we can take the quartic parametrized by
\begin{equation}\label{eq:param_D0}
\P^{1}\ni [t:s]\mapsto [t^{2}s^{2}:s^{3}(s-t):t^{4}]\in \P^{2},
\end{equation}
with $q_1=[0:0:1]$ and $q_{2}=[0:1:0]$. By Lemma \ref{lem:log_general_type} for $\gamma\neq 3$ we have $\kappa(\P^{2}\setminus \bar{C})=2$. For $\gamma=3$, the complement contains the affine line $\{[x,1,0]:x\neq 0\}$, so by Lemma \ref{lem:Qhp_has_no_lines} it is not of log general type.

Let then $\bar C$ be a curve of HN-type $\cD(\gamma-3,2,1)$ for $\gamma\geq 3$ and let $t_1$ be the line tangent to $\bar C$ at $q_1$. The cusps $q_{1}$, $q_{2}$ have multiplicity sequences
\begin{equation*}
(\gamma-1)\text{\ \ and\ \ } (2)_{\gamma-1}.
\end{equation*}
By Lemma \ref{lem:HN-equations} $\deg\bar{C}=\gamma+1$. Since $\gamma-1=\mu(q_{1})<(t_{1}\cdot \bar{C})_{q_{1}}<\deg \bar C$, we get  
\begin{equation*}
t_1\cdot (\bar C\setminus \{q_1\})=1.
\end{equation*}
The Bezout theorem implies that a general line passing through $q_{1}$ meets $\bar{C}\setminus \{q_{1}\}$ twice, so the pencil of lines through $q_{1}$ gives a $\C^{**}$-fibration of $\P^{2}\setminus \bar{C}$. Reductions of at least two of its fibers are isomorphic to $\C^{*}$: namely, the line tangent to $q_{1}$ and the line passing through $q_{2}$. By Lemma \ref{lem:Suzuki} there is another fiber whose reduction is isomorphic to $\C^{*}$, so there is a line $\ell\subseteq \P^{2}$ passing through $q_{1}$ which meets $\bar{C}\setminus \{q_{1}\}$ at some smooth point $p\in \bar{C}$ with multiplicity two. This point is uniquely determined by $\bar C$. We infer that $\bar C\setminus \ell \subseteq \P^{2}\setminus \ell$ is the image of a singular embedding $\phi\colon \C^{*}\to \C^{2}$ (it is the case \cite[(b)]{BoZo-annuli} for $m=k=\gamma-1$). 

\begin{figure}[htbp]
	\centering
	\begin{subfigure}{0.3\textwidth}
		\centering
		\includegraphics[scale=0.2]{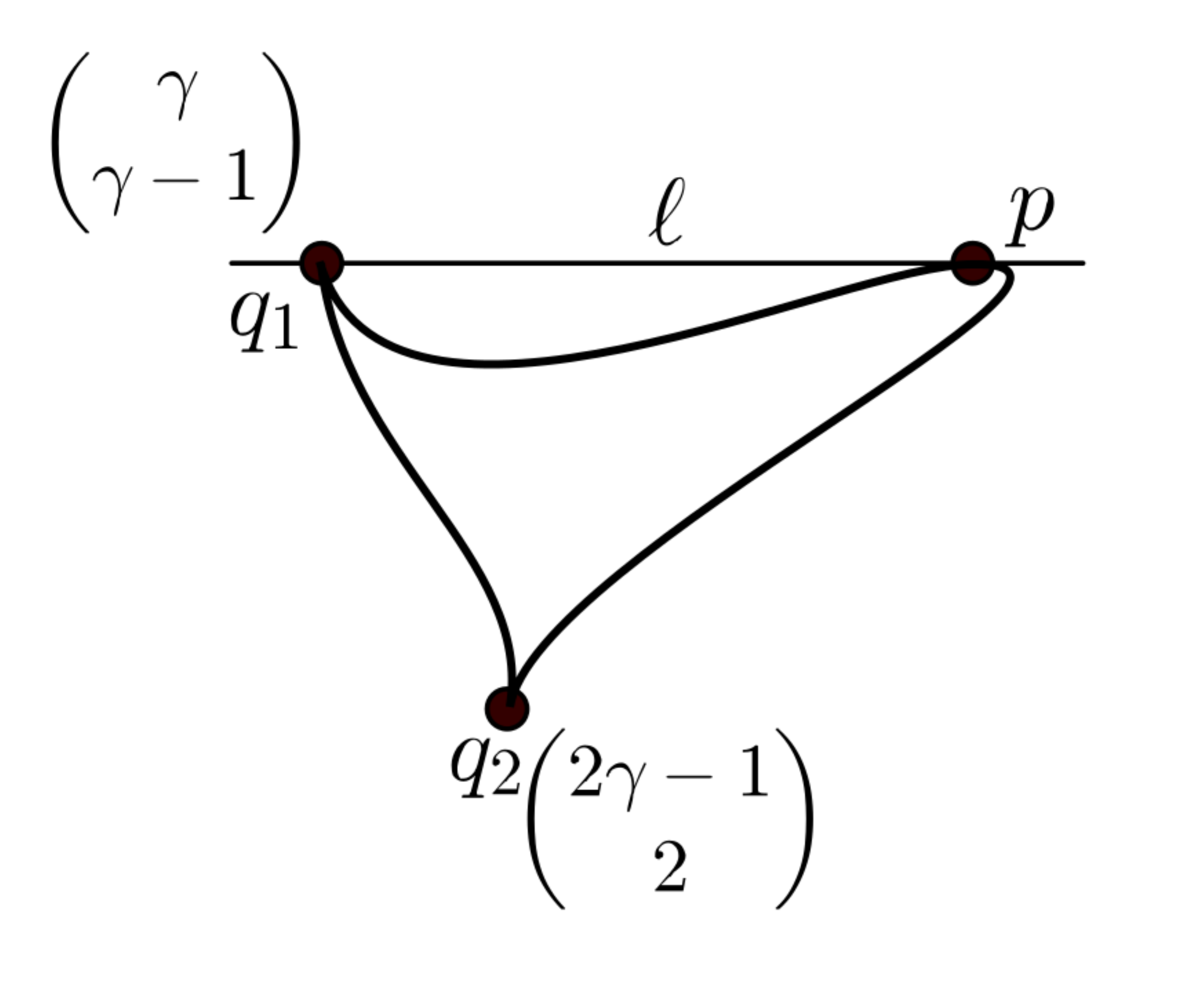}
		\caption{$\cD(\gamma-3,2,1)$.}
	\end{subfigure}$\longleftarrow$\hfill
	\begin{subfigure}{0.3\textwidth}
		\centering
		\includegraphics[scale=0.2]{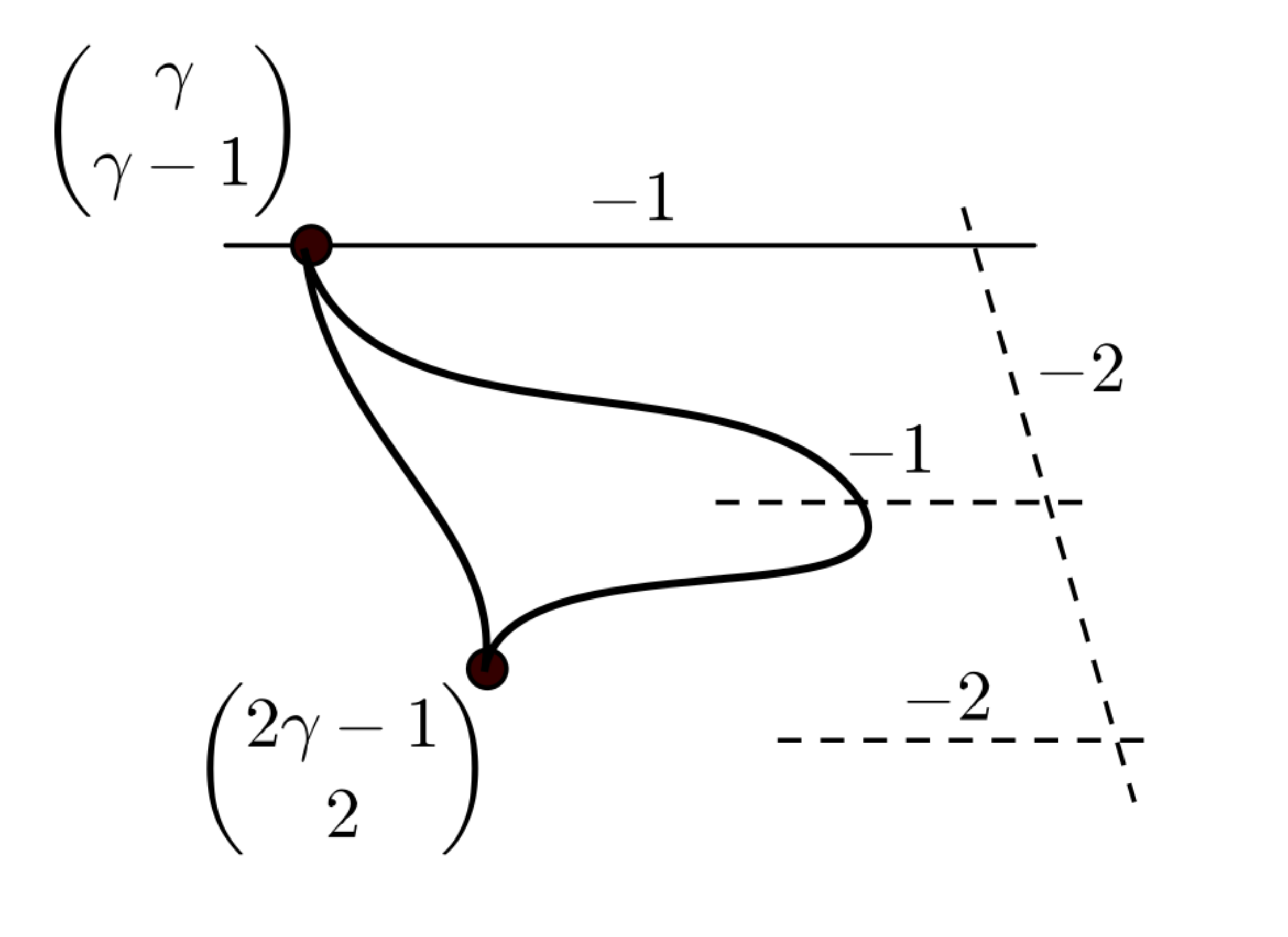}
	\end{subfigure}$\longrightarrow$\hfill
	\begin{subfigure}{0.3\textwidth}
		\centering
		\includegraphics[scale=0.2]{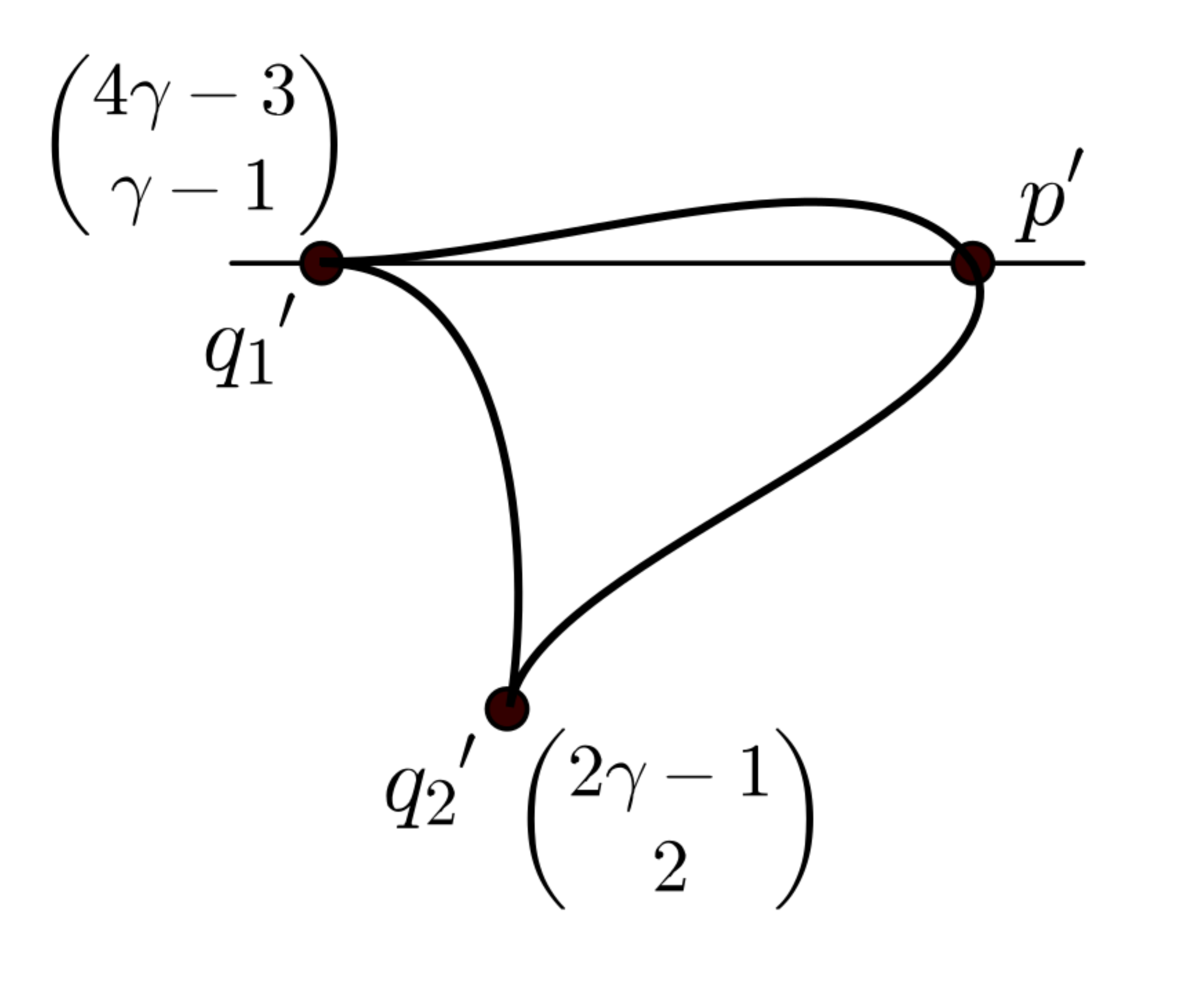}
		\caption{$\cG(\gamma)$.}
	\end{subfigure}
\caption{A quadratic transformation of $\P^{2}$ mapping a curve of HN-type $\cD(\gamma-3,2,1)$ onto a curve of HN-type $\cG(\gamma)$.}
\label{fig:differentCst}
\end{figure}

Let $\sigma\colon \P^{2}\dashrightarrow \P^{2}$ be the composition $\sigma=\sigma_{2}\circ\sigma_{1}^{-1}$, where (see Fig.\ \ref{fig:differentCst}) $\sigma_{1}$ consists of blowups at $p$ and its two infinitely near points on the proper transforms of $C$ and $\sigma_2$ contracts the proper transforms of $\ell$ and the proper transforms of the first and second exceptional curve of $\sigma_{1}$. This (quadratic) Cremona transformation has one proper base point and can be written as a composition of four standard quadratic transformations (see \cite[5.3.5]{Moe-cuspidal_MSc}). Put $\bar E=\sigma_*C$. Since $\sigma$ is an isomorphism off $\ell$, it restricts to an automorphism of $\C^{2}=\P^{2}\setminus \ell$. At infinity it contracts subsequently three curves meeting the germ of $C$ at $q_{1}$ with multiplicity $\gamma-1$, hence the multiplicity sequence of the cusp $q_1$ is augmented by the sequence $(\gamma-1)_{3}$. The germ at $p$ is touched only by $\sigma_{1}$, so it remains smooth and meets the total transform of $\ell$ transversally. Thus the multiplicity sequences of $q_1, q_2\in \bar E$ are
$$(\gamma-1)_{4} \text{\ \ and\ \ } (2)_{\gamma-1}.$$
By Lemma \ref{lem:Hn_gives_multiplicities} $\bar E\subset \P^2$ is of HN-type $\cG(\gamma)$. 

For the proof of uniqueness let $\bar E$ be of HN-type $\cG(\gamma)$ with $\gamma\geq 3$ and let $q_1$, $q_2$ be the cusps. The degree of $\bar E$ is $2\gamma-1$. Denote by $t_1$ the line tangent to $\bar E$ at $q_{1}$. Since the multiplicity sequence of $q_1$  is $(\gamma-1)_4$, the intersection of the proper transforms of $\bar E$ and $t_1$ after two blowups over $q_1$ equals $\deg \bar E-2(\gamma-1)=1<\gamma-1$, and hence it is necessarily zero. Then $(t_1\cdot \bar{E})_{q_{1}}=2\gamma-2$, so by the Bezout theorem $t_1$ meets $\bar{E}\setminus \{q_{1}\}$ in a unique point $p'$, transversally. Let $\tilde{\sigma}\colon \P^{2}\dashrightarrow \P^{2}$ be a Cremona transformation with one proper base point as above, with $(q_1,t_1,\bar E)$ replacing $(p,\ell,C)$. Put $\bar C=\tilde{\sigma}_*\bar E$, and let $p$, $q_{2}'$, $q_{1}'$ be the images of $t_{1}$, $q_{2}$ and of the common point of $C$ and the image of the last exceptional curve over $q_{1}$. Then the curve $\bar{C}$ with cusps $q_{1}'$, $q_{2}'$ is of HN-type $\cD(\gamma-3,2,1)$ and the point $p\in \bar{C}$ has the same property as before, i.\ e.\ the line joining $q_{1}'$ and $p$ meets $\bar{C}$ only in $q_{1}'$ and $p'$ with multiplicities $\gamma-1$ and $2$, respectively. We see that $\tilde{\sigma}$ is inverse to $\sigma$ up to a projective equivalence. Since the centers of blowups and blowdowns constituting $\sigma$ and $\tilde{\sigma}$ are uniquely determined by $\bar{C}$ and  $\bar{E}$ respectively, by the universal property of blowups the transformations are well defined on the classes of projective equivalence of rational cuspidal curves of HN-types $\cD(\gamma-3,2,1)$ and $\cG(\gamma)$, respectively.

Thus it remains to show that, up to a projective equivalence, for every integer $\gamma\geq 3$ there is only one curve of HN-type $\cD(\gamma-3,2,1)$. For $\gamma>3$ this was done in Lemma \ref{lem:type_AD}. So assume $\bar C$ is of HN-type $\cD(0,2,1)$. Let $\phi\colon \P^{1}\to \P^{2}$ be a parametrization of $\bar{C}$. Put $\phi=[\phi_{x}:\phi_{y}:\phi_{z}]$, where $\phi_{x}$, $\phi_{y}$, $\phi_{z}\in \C[s,t]$ are homogeneous polynomials of degree $4$. Up to a choice of coordinates on $\P^2$ we can assume that $\{x=0\}$ is the line joining $q_{1}=\phi[1:0]$ with $q_{2}=\phi[0:1]$ and the lines $\{y=0\}$, $\{z=0\}$ are tangent to $\bar{C}$ at $q_{1}$ and $q_{2}$, respectively. As noticed above, $\{y=0\}$ meets $\bar{C}\setminus \{q_{1}\}$ once and transversally. We may assume the point of intersection is $\phi[1:1]$. Now $\{x=0\}$ meets $\bar{C}$ at $t=0$ and $s=0$ with multiplicity $2$, so $\phi_{x}=at^{2}s^{2}$ for some $a\in \C^{*}$. Similarly, the line $\{z=0\}$ meets $\bar{C}$ with multiplicity $4$ at $t=0$, so $\phi_{z}=bt^{4}$ for some $b\in \C^{*}$. Finally, the line $\{y=0\}$ meets $\bar{C}$ at $s=0$ with multiplicity $3$ and at $t=s$ with multiplicity $1$, so $\phi_{z}=cs^{3}(s-t)$ for some $c\in \C^{*}$. Thus, composing with a diagonal automorphism of $\P^{2}$ we can assume that $\phi$ is given by \eqref{eq:param_D0}. 
\end{proof}

\begin{rem}[Equivalent $\C^*$-embeddings with projectively non-equivalent closures] The construction shows that equivalent $\C^*$-embeddings into $\C^2$ may, by choosing different embeddings $\C^{2}\to \P^{2}$ and taking closures, give projectively non-equivalent bicuspidal curves. In the above case those are $\cG(\gamma)$ and  $\cD(\gamma-3,2,1)$. Their degrees are $2\gamma-1$ and $\gamma+1$ respectively. For $\gamma>3$ both closures have complement of log general type, however, for $\gamma=3$ the complement of the first closure is of log general type while of the second one is not.
\end{rem}

Our next example shows that, conversely, non-equivalent injective morphisms $\C^{*}\to \C^{2}$ can have projectively equivalent closures.

\begin{prz}[Non-equivalent $\C^*$-embeddings with projectively equivalent closures]	\label{ex:different_Cst2}
Let $(\bar{C}, (q_1,q_2,p))$ be a pair as in the proof of Lemma \ref{lem:type_G}, i.\ e.\ $\bar C$ is of HN-type $\cD(\gamma-3,2,1)$ with cusps $q_1$, $q_{2}$ and the line $\ell$ joining $q_{1}$ and $p$ is tangent to $\bar{C}$ at $p$ with multiplicity $2$. The line $\ell'$ joining the cusps $q_{1}$ and $q_{2}$ meets $\bar{C}$ only at these cusps. Hence $\bar{C}\setminus \ell'\subseteq \P^{2}\setminus \ell'$ is the image of a closed smooth embedding $\C^{*}\to \C^{2}$ and $\bar{C}\setminus \ell\subseteq \P^{2}\setminus \ell$ is the image of a singular embedding $\C^{*}\to \C^{2}$. Therefore, $\bar{C}$ can be realized as a closure of images of two non-equivalent embeddings $\C^{*}\to \C^{2}$.
\end{prz}

\subsection{Comparison with other results from the literature.}

Some families or special subfamilies of rational cuspidal curves listed in \ref{thm:possible_HN-types} appear also in other places in literature. Here we collect the examples which we are aware of.

\begin{rem}[Borodzik-\.{Z}o\l\c{a}dek list]	\label{rem:BoZo}
	In \cite[Main Theorem]{BoZo-annuli} Borodzik and \.{Z}o\l\c{a}dek classified possible injective morphisms $\C^*\to\C^2$ up to a change of coordinates, assuming certain regularity conditions stated in 2.40 loc.\ cit.\ (for definitions of $ext$, $Curv$ and $\overline{Curv}$ see respectively 2.39, 2.37 and 2.34 loc.\ cit). The Commentary following the Main Theorem of loc.\ cit.\ provides singularity types of closures of these curves in terms of (slightly modified) Puiseux expansions. The list enables one to compute the Puiseux characteristic sequence, and thus, using Lemma \ref{lem:HN_vs_Puiseux}, to compare them with our list in Theorem \ref{thm:possible_HN-types}. Let us remind the reader (see Section \ref{sec:construction_(h)}) that sometimes from one embedding $\C^{*}\to \C^{2}$ one can obtain non-equivalent cuspidal projective closures by taking different embeddings $\C^{2}\to \P^{2}$. In what follows we consider the cuspidal curves obtained as closures of the embeddings listed in loc.\ cit.\ via the standard embedding $\C^{2}\ni (x,y)\mapsto [x:y:1]\in \P^{2}$.  
	
	The closures of $\C^*$-embeddings (k), (o)-(r), (u)-(v) from \cite[Main Theorem]{BoZo-annuli} have unique points at infinity, so are not cuspidal. The embeddings (a) have bicuspidal closures but they admit a very good asymptote, namely the line $\{x=0\}$ tangent to the first cusp at infinity (see Section \ref{ssec:AD}), and hence the complement of its closure is not of log general type. This is also the case if $|l-n|\in \{0,1\}$ in (l) or (n) or if $|l-n|\in \{0,1,p\}$ in (m) (the line tangent to the cusp contained in the affine part is the very good asymptote). In other cases of (l)-(n) the closure is not cuspidal. We now compare closures in $\P^2$ of the remaining families with the list in Theorem \ref{thm:possible_HN-types}. We give below an explicit computation in case (c).

	\begin{enumerate}
		\item[(c)-(f)] The closures are respectively of HN-types $\cC(k,n,m)$, $\cD(k,n,m)$, $\cA(k,n,m)$, $\cB(k,n,m)$.
		\item[(s)] If $2\mid n$ then the closure is of HN-type $\cE(n/2)$. If $2\nmid n$ then the closure is of HN-type $\cF((n+1)/2)$.
		\item[(b)] If $k>m$ then the closure is of HN-type $\FZa(k+2,m)$ and if $k=m$ then it is of HN-type $\cD(m-2,2,1)$. Otherwise its has only one place at infinity, so it is not cuspidal.
		\item[(j), (w)] The closures are of HN-type $\FZa(m+n+2,n)$ and $\FZa(4,2)$ respectively.
		\item[(g)-(i), (t)] The complements of the closures are of log general type, but do not admit $\C^{**}$-fibrations.
	\end{enumerate}
	Cuspidal curves of the remaining HN-type $\cG(\gamma)$ can be realized  as a closures of curves (b) from loc.\ cit.\ for $m=k=\gamma-1$ via a non-standard embedding $\C^{2}\to \P^{2}$ described in the proof of Lemma \ref{lem:type_G}. 
\end{rem}

\begin{prz}[\cite{BoZo-annuli}(c) is HN-type $\cC$]\label{ex:BoZo_HN_computation}
	Let us make an explicit comparison between the case (c) of \cite{BoZo-annuli} and our HN-type $\cC$, leaving an analogous analysis in the remaining cases to the reader. By definition, a $\C^*$-embedding (c) is the image of the morphism $\C^{*}=\Spec \C[t,t^{-1}]\to \Spec \C[x,y]=\C^{2}$ given by:
		\begin{equation*}
		x = t^{mn}(t-1),\ y=S_{k}(t^{-1})(1-t), \mbox{ where } 
		 S_{0}(u)=u^{n},\ S_{k+1}(u)=(S_{k}(u)-S_{k}(1))\tfrac{u^{mn+1}}{u-1}.
		\end{equation*}
	for $k,m \geq 1$, $n\geq 2$. (We will disregard the additional condition $(m,n)\neq (1,2)$, which is imposed to assure the embedding is different than the ones in (e).) The image is a smooth curve isomorphic to $\C^*$. In order to compute the HN-type of its closure we use the local description provided by the Commentary to the Main Theorem in loc.\ cit. For $t_0\in \P^1=\C^1\cup \{\8\}$ and for two Laurent series $f,g$ in variable $t$ we write $f\sim_{t_0} g$ if $\lim_{t\to t_0}f(t)/g(t)=c$ for some $c\in \C^{*}$.
	
	The local description of $q_{1}\in \bar{E}$, which corresponds to $t\to 0$, is 
	\begin{equation*}
	x\sim_{0}t^{mn},\quad y\sim_{0} x^{-k-(1/m)}(1+\alpha x^{1/mn})\quad \mbox{for some } \alpha\in \C^{*}.
	\end{equation*}
	We have $(x,y)=(X/Z,Y/Z)$, were $[X:Y:Z]$ are homogeneous coordinates on $\P^{2}$. Choose local coordinates $(u,v)=(Z/Y,X/Y)$ near $q_{1}=[0:1:0]\in \P^{2}$. The first condition above means that $X/Z=t^{mn}\phi_{v}(t)$, where $\phi_{v}(t)$ is some invertible power series. Substituting it to the second condition we get $u=t^{kmn+n}(1+\alpha t)^{-1}\phi_{u}(t)$ for some other invertible power series $\phi_{u}$. Using the implicit function theorem we treat $t$ as an analytic function of $u^{1/(kmn+n)}$
	Therefore, we can write a Puiseux series like \eqref{eq:Zariski_pairs_definition}:
	\begin{equation*}
	v=(X/Z)\cdot u=u^{\frac{km+m+1}{km+1}}
	\cdot (1+\alpha u^{\frac{1}{(km+1)n}})^{-1}\cdot \phi(u^{\frac{1}{(km+1)n}}).
	\end{equation*}
	for some invertible power series $\phi$.
	Therefore, the Zariski pairs (see Section \ref{sec:appendix}) of $q_{1}\in \bar{E}$ are $(km+1,km+m+1)$, $(1,n)$, so by Lemma \ref{lem:HN_vs_Puiseux}, $q_{1}\in \bar{E}$ is of HN-type $\binom{n(km+m+1)}{n(km+1)}\binom{n}{1}$. 
	
	The local description of $q_{2}\in \bar{E}$, which corresponds to $t\to \infty$, is 
	\begin{equation*}
	x\sim_{\8} t^{mn+1},\ \  y\sim_{\8} \alpha_{1}x^{-1}+\alpha_{2}x^{-k-n/(mn+1)} \text{\ \ \ for some \ } \alpha_{1},\alpha_{2}\in \C^{*}.
	\end{equation*}
	Choose local parameters $(u,v)=(Z/X,Y/X)$ near $[1:0:0]\in \P^{2}$. As before, we use the first condition and the implicit function theorem to treat the local parameter $t^{-1}$ near $\infty\in \P^{1}$ as an analytic function of $u^{1/(mn+1)}$. Now the second condition implies that
	\begin{equation*}
	v=u^{\frac{(k+1)(mn+1)+n}{mn+1}}\cdot \phi(u^{\frac{1}{mn+1}}) \quad \mbox{ for some invertible power series } \phi,
	\end{equation*} 
	so $q_{2}\in \bar{E}$ has one Zariski pair $(mn+1,(k+1)(mn+1)+n)$ and thus by Lemma \ref{lem:HN_vs_Puiseux} it is of HN-type $\binom{(k+1)(mn+1)+n}{mn+1}$. Therefore, $\bar{E}$ is of HN-type $\cC(k,n,m)$.
\end{prz}

\begin{rem}	\label{rem:Fenske-Tono}
	Independently from the study of planar $\C^*$-embeddings, rational cuspidal curves of HN-types $\cB$ - $\cD$ in case $s=1$ were constructed by Fenske (cases respectively 4, 5 and 6 in \cite[Thm.\ 1.1]{Fenske_1and2-cuspidal_curves}; case 8 there is a subcase of 6). Similarly, curves of HN-types $\cA$ - $\cD$ in case $\gamma=1$ were constructed by Tono (see the proof of Lemma \ref{lem:Tono_E2=-1,-2}).
\end{rem}

\begin{rem}	\label{rem:Liu}
	The HN-types $\ORa(1)$ and $\ORb(1)$ written in a standard way are $\binom{22}{3}$ and $\binom{43}{6}$, respectively. The main result of \cite{FLMN_one_pair} asserts that the Orevkov curves $C_{4}$ and $C_{4}^{*}$, which realize these HN-types, are the only unicuspidal planar curves with complement of log general type and with one HN-pair. Using the bounds obtained by Borodzik and Livingstone in \cite{BoroLivi-HeegaardFloer_and_cusps} via Heegard-Floer homology methods, Liu  \cite[Thm.\ 1.1]{Liu-thesis} extended this result by describing possible types of singularities of unicuspidal curves under the assumption that they have at most two Puiseux pairs (see Section \ref{sec:cusps} for definitions). Those which are realized by curves with complements of log general type correspond exactly to the HN-types $\ORa$ - $\ORb$. Unfortunately, the assumption that the number of Puiseux pairs is at most two is still strong and it seems unlikely one could go much further this way.
\end{rem}

\subsection{Proof of the implication \ref{thm:possible_HN-types}$\implies$\ref{thm:geometric}.}\label{sec:1.2=>1.1}

\begin{proof}[Proof of \ref{thm:possible_HN-types}$\implies$\ref{thm:geometric}]
	
	Assume $\bar{E}\subseteq \P^{2}$ has at least three cusps. By Theorem \ref{thm:possible_HN-types} it is of HN-type $\FZa$, hence by Proposition \ref{prop:existence_and_uniqueness}
	it is projectively equivalent to one of the Flenner-Zaidenberg curves of the first kind. The Bezout theorem implies that a line $\ell$ joining the unique non semi-ordinary cusp of $\bar E$ with a semi-ordinary one does not meet $\bar E$ in any other point, so $\bar E\setminus \ell$ is the image of some singular embedding of $\C^{*}$ into $\C^{2}$, with one semi-ordinary cusp. This proves Theorem \ref{thm:geometric}(d).
	
	Assume $\bar{E}\subseteq \P^{2}$ is bicuspidal. By Theorem \ref{thm:possible_HN-types} $\bar{E}$ is of one of the HN-types $\cA$ - $\cG$. Let $\mu_{j}$ denote the multiplicity of the cusp $q_{j}\in \bar{E}$. If $\bar{E}\subseteq \P^{2}$ is not of HN-type $\cG$ then $\deg \bar{E}=\mu_{1}+\mu_{2}$, so the line joining $q_{1}$ and $q_{2}$ is the line $\ell$ from Theorem \ref{thm:geometric}(b). Assume $\bar{E}\subseteq \P^{2}$ is of HN-type $\cG(\gamma)$. Denote by $\ell$ the line tangent to $\bar{E}$ at $q_{1}$. The multiplicity sequence of $q_{1}\in \bar{E}$ is $(\gamma-1)_{4}$, so $2\gamma-1=\deg \bar{E}\geq (\bar{E}\cdot \ell)_{q_{1}}\geq 2(\gamma-1)$. If the last inequality is strict then $(\bar{E}\cdot \ell)_{q_{1}}\geq 3(\gamma-1)$, which is impossible. Thus $(\bar{E}\cdot \ell)_{q_{1}}=2(\gamma-1)$ and hence $\ell$ meets $\bar{E}\setminus \{q_{1}\}$ transversally in one smooth point, which gives Theorem \ref{thm:geometric}(c).
	
	Finally, assume $\bar{E}\subseteq \P^{2}$ is unicuspidal. By Theorem \ref{thm:possible_HN-types} it is of HN-type $\ORa(k)$ or $\ORb(k)$ for some $k\geq 1$, so by Proposition \ref{prop:existence_and_uniqueness} it is projectively equivalent to one of the Orevkov curves $C_{4k}$, $C_{4k}^{*}$. This proves Theorem \ref{thm:geometric}(a).
\end{proof}

\bigskip
\section{Appendix: Comparison of other numerical characteristics of cusps}\label{sec:appendix}

As before let $\bar E\subseteq \P^2$ denote a rational cuspidal curve and let $q\in \bar{E}$ be a cusp. In our approach we relied on the description of cusps in terms of Hamburger-Noether pairs or multiplicity sequences. Lemma \ref{lem:Hn_gives_multiplicities} shows how to compute one from another. We now explain how to compute other numerical characteristics and how they are related to HN-pairs.

There is a classical notion of the \emph{Puiseux characteristic sequence} $(\beta_{0};\beta_{1},\dots, \beta_{g})$ (see \cite[3.1]{Wall_singular_curves}), also known as a sequence of \emph{characteristic exponents} \cite[p.\ 204]{GLS_Intro_to_singularities} defined as follows. By the Newton-Puiseux theorem \cite[2.1]{Wall_singular_curves}, one can take local analytic coordinates $(x,y)$ at $q\in \P^{2}$, such that the germ of $\bar{E}$ at $q$ is described by 
\begin{equation}\label{eq:Puiseux_param}
x=t^{\beta_{0}},\quad y=\sum_{k=\beta_0}^{\infty} d_{k}t^{k}.
\end{equation}
Put $e_0\de\beta_0$ and note that $e_0$ is the multiplicity of the cusp. Now for we $i>0$ define inductively:
\begin{equation} \label{eq:charseq}
\beta_{i}\de \min \{k: d_{k}\neq 0\mbox{ and } e_{i-1}\nmid k \},\quad e_{i}=\gcd(e_{i-1},\beta_{i}), \quad  i=1,\dots g,
\end{equation}
where $g$ is the smallest integer such that $e_{g}=1$. Therefore, $\beta_i$ is the least exponent in the series which does not belong to the additive group generated the preceding $\beta_j$.

Instead of the characteristic sequence one can work with \emph{Puiseux pairs} $(m_i,n_i)=(\frac{\beta_{i}}{e_{i}},\frac{e_{i-1}}{e_{i}})$, $i\geq 1$ (see \cite[3.7]{Wall_singular_curves}). Letting $m_0=1$, the characteristic sequence can be recovered from them as 
\begin{equation*}
\beta_{i}=m_{i}(n_{i+1}n_{i+2}\dots n_{g}), \quad i=0,\ldots, g.
\end{equation*}
By an analytic change of coordinates $y\mapsto y-\sum_{\beta_{0}|k}d_{k}x^{k/\beta_{0}}$ we may assume that in the expansion \eqref{eq:Puiseux_param} $d_k=0$ if $\beta_0\mid k$. Then we can write $y$ as a Puiseux series:
	\begin{equation}\label{eq:Zariski_pairs_definition}
	y=x^{\frac{a_{1}}{b_{1}}}\left(
	d_{\beta_{1}}'+x^{\frac{a_{2}}{b_{1}b_{2}}}\left(
	d_{\beta_{2}}'+ x ^{\frac{a_{3}}{b_{1}b_{2}b_{3}}}\left(
	\ldots\left(
	d_{\beta_{g-1}}' +d_{\beta_{g}}'x^{\frac{a_{g}}{b_{1}b_{2}\dots b_{g}}} 
	\right)\cdots
	\right)
	\right)
	\right),
	\end{equation}
	where for $i=1\dots, g$, $\gcd (b_{i},a_{i})=1$ and $d_{\beta_{i}}'$ is an invertible power series in the variable $x^{a_{i}/(b_{1}\dots b_{i})}$.
The pairs $(b_{1},a_{1}),\dots, (b_{g},a_{g})$ are called \emph{Zariski pairs} \cite[3.7]{Wall_singular_curves} or \emph{Newton pairs} \cite[p.\ 49]{EN_Newton_pairs}. They are determined by, and can be recovered from, the Puiseux pairs $(m_{i},n_{i})$ by means of the formulas
\begin{equation*}
b_{i}=n_{i}\mbox{ for }i=1,\dots g\quad \mbox{ and } \quad a_{1}=m_{1},\ a_{i}=m_{i}-n_{i}m_{i-1}\mbox{ for }i=2,\dots, g.
\end{equation*}

The Zariski pairs arise naturally in the course of the Newton-Puiseux algorithm \cite[2.1]{Wall_singular_curves} (see also \cite[I.3.6]{GLS_Intro_to_singularities}) used to obtain the expansion \eqref{eq:Zariski_pairs_definition} from the equation of $\bar{E}$, say, $f(x,y)=0$ ($q$ has $(x,y)$-coordinates $(0,0)$). This algorithm defines a sequence $f_0,f_{1}\ldots $ of polynomials in variables $x,y$ starting with $f_{0}=f$. Having $f_i\in \C[x,y]$ for some $i\geq 0$ let $-b_{i}/a_{i}$ be the slope of the steepest facet of the Newton polygon $\Gamma$ of $f_i$. Let $f^0_i$ be the sum of monomials corresponding to the points of $\Gamma$ lying on this facet, let $r_i=\deg f_i^0(x^{b_i},y^{a_i})$ and let $u_i$ be some root of $f_i^{0}(1,y)=0$. Put
\begin{equation*}
f_{i+1}(x,y)=f_{i}(x^{a_i},x^{b_i}(u_i+y))/x^{r_i}, 
\end{equation*}
Then $f_i\in \C[x,y]$. One can show that the Puiseux series $x^{a_{1}/b_{1}}(u_{1} +x^{a_{2}/b_{2}}(u_{2}+\dots ))$ obtained this way has the form \eqref{eq:Zariski_pairs_definition} (that is, the denominators of the exponents are bounded) and converges to the solution of $f(x,y)=0$.

Note that the same slope $-(b_{i}/a_{i})$ can occur as the steepest slope of several subsequent $f_{i}$'s, which corresponds to the fact that the coefficients $d_{\beta_i}'$ in \eqref{eq:Zariski_pairs_definition} are non-constant (but for $i\neq g$ they are polynomials). Equivalently, it corresponds to the fact that after multiplying out \eqref{eq:Zariski_pairs_definition} to get \eqref{eq:Puiseux_param}, some coefficients $d_{k}$ in \eqref{eq:Puiseux_param} are nonzero for $k\neq \beta_{1},\dots, \beta_{g}$. Those multiple occurrences of $x^{a_{i}/(b_{1}\dots b_{i})}$ can be ignored if one is interested only in the topology of singularities. More precisely, by a \emph{topological type} of a cusp $q\in \bar{E}$ we mean the homotopy type of a pair $(S,S\cap \bar{E})$, where $S$ is a small $3$-sphere in $\P^{2}$ around $q$. The knot $S\cap \bar{E}$ is called the \emph{link} of $q\in \bar{E}$. The topological type determines and is determined by the Zariski pairs, see \cite[Appendix to ch.\ I]{EN_Newton_pairs}. In other words, a cusp  with Zariski pairs $(b_{1},a_{1}),\dots, (b_{g},a_{g})$ has the same topological type as the one locally described by
\begin{equation}\label{eq:Zariski_topologically}
y=x^{\frac{a_{1}}{b_{1}}}\left(
1+x^{\frac{a_{2}}{b_{1}b_{2}}}\left(
1+x^{\frac{a_{3}}{b_{1}b_{2}b_{3}}}\left(
\ldots\left(
1+x^{\frac{a_{g}}{b_{1}b_{2}\dots b_{g}}}
\right)\cdots
\right)
\right)
\right).
\end{equation}

\smallskip
\begin{uw}
	The above definitions of numerical invariants seem to be the most common ones. However, one may encounter other definitions in the literature. For example, \cite[p.\ 204]{GLS_Intro_to_singularities} uses the term \emph{Puiseux pairs} for what we have defined as Zariski pairs.
\end{uw}

\begin{lem}[Puiseux characteristic sequence from HN-pairs]\label{lem:HN_vs_Puiseux}
	Let $q\in \bar{E}$ be a cusp of standard HN-type $\binom{c_{1}}{p_{1}},\dots, \binom{c_{h}}{p_{h}}$. If $(\beta_0,\beta_1,\ldots,\beta_g)$ denotes its Puiseux characteristic sequence then $g=h$ and
	\begin{equation}\label{eq:char_seq_using_HN}
	\beta_0=p_1, \text{\ \ and \ \ } \beta_i=c_1+p_2+\ldots+p_i,\ \ i=1,\ldots,h.
	\end{equation}
	Therefore, the Puiseux pairs of $q\in \bar E$ are:
	$\left(c_{1}/c_{2},p_{1}/c_{2}\right),\ ((c_{1}+p_2+\ldots+p_k)/c_{k+1},c_{k}/c_{k+1})_{k=2}^{h}$
	 and the Zariski pairs are
$\left(p_{k}/c_{k+1},c_{k}/c_{k+1}\right)_{k=1}^{h}$, where $c_{h+1}=1$.
\end{lem}

\smallskip
Note that it follows that the Zariski pairs are the standard HN-pairs whose terms are divided by their greatest common divisor and written in the opposite order.

\begin{proof}
	Lemma \ref{lem:Hn_gives_multiplicities} shows how to compute the multiplicity sequence from the HN-type, and  \ref{lem:mult_seq} below shows how to compute the characteristic sequence from the multiplicity sequence. Therefore,  this result can be obtained by formally combining those two. However, we provide here a direct argument.
	
	Let $(\beta_{j})_{j=0}^{g}$ be the characteristic sequence of $q\in \bar{E}$ and let $(x,y)=(t^{\beta_0},\sum_{k=\beta_{1}}^{\infty} d_{k}t^{k})$ be a corresponding parametrization. The first two terms of \eqref{eq:char_seq_using_HN} and of the Puiseux characteristic sequence of $q\in \bar E$ are equal by definitions. To prove the equality of the remaining terms we argue by induction on the minimal number of blowups needed to resolve $(\chi,q)$, where $\chi$ is the germ of $\bar E$ at $q$.
	
	Assume that $c_{1}-p_{1}\nmid p_{1}$. After one blowup at $q$ the first HN-pair of the proper transform of the cusp is  $\binom{c_{1}-p_{1}}{p_{1}}$ or $\binom{p_{1}}{c_{1}-p_{1}}$, and the other ones remain unchanged. Computing the parametrization of the proper transform we get (see \cite[3.5.5]{Wall_singular_curves}) that the characteristic sequence of its cusp is $(p_{1};c_{1}-p_{1},(\beta_{j}-p_{1})_{j=2}^{g})$ in the first case  and  $(c_{1}-p_{1};p_{1},(\beta_{j}-(c_{1}-p_{1}))_{j=2}^{g})$ in the second case. Therefore, in these cases the inductive hypothesis gives the desired formula.
	
	Assume that $c_{1}-p_{1}|p_{1}$. Then $c_{1}-p_{1}=\gcd (c_{1},p_{1})$, because $p_{1}\nmid c_{1}$. Again, a computation  shows (see loc.\ cit.) that the characteristic sequence of the proper transform is $(c_{1}-p_{1};(\beta_{j}-(c_{1}-p_{1}))_{j=2}^{g})$. Choose a smooth germ $Y$ at this cusp such that its intersection multiplicity with the proper transform of $\chi$ is maximal possible; by the definition of the characteristic sequence, this multiplicity equals $\beta_{2}-(c_{1}-p_{1})$. After $p_{1}/(c_{1}-p_{1})$ further blowups, the multiplicity of the cusp drops to $m=\min(c_{2},p_{2})$. The proper transform $T'$ of $T$ meets the cusp, and since $T$ is smooth, $T'$ is transversal to the last exceptional curve. Put $\mu=(T'\cdot \chi')_{q'}$, where $(\chi',q')$ is the proper transform of $(\chi,q)$. We have $\mu>m$, that is, $T'$ is tangent to $\chi'$ if and only if the last exceptional curve is not, which happens if and only if $p_{2}>c_{2}$. In this case, by the definition of $T'$ the number $(T'\cdot \chi')_{q'}$ is maximal possible, that is, $\mu=p_{2}$. Otherwise, $\mu=m$, which in this case equals $p_{2}$ by definition. Thus $\mu=p_{2}$. We have $\beta_{2}-(c_{1}-p_{1})=p_{1}+\mu$, and combining this with the equality $c_{1}=\beta_{1}$ we obtain the desired formula for $\beta_{2}-\beta_{1}$. The formulas for $\beta_{j}-\beta_{j-1}$ for $j\geq 3$ are given by the inductive hypothesis.
	
	This ends the computation of the Puiseux sequence. The remaining formulas follow from the discussion preceding the Lemma and the fact that by \eqref{eq:HN_c_j+1} for standard HN-pairs we have $c_{j}=\gcd(c_{j-1},p_{j-1})$, $j=2,\dots, h+1$.
\end{proof}

The \emph{multiplicity sequence} of $q_{j}\in \bar E$, that is, the sequence of multiplicities of proper transforms of $q_{j}\in \bar{E}$ under consecutive blowups, can be easily computed from the characteristic sequence using the Euclidean algorithm. 

\begin{lem}[Multiplicity sequence vs characteristic sequence]\label{lem:mult_seq}
	Let $q\in \bar{E}$ be a cusp with the characteristic Puiseux sequence $(\beta_{0};\beta_{1},\dots, \beta_{g})$. Put $m_{1,1}\de \beta_{0}$, $m_{i,0}=\beta_i-\beta_{i-1}$ and subsequently for each $i=1,\dots, g$ define $s_{i,k}$, $m_{i,k}$ and $n_{i}$ by induction on $k$ as the unique positive integers satisfying:
	\begin{equation*}\begin{split}
	\quad m_{i,k}&=s_{i,k+1}m_{i,k+1}+m_{i,k+2}\text{\ \ and\ \ }  m_{i,k+2}<m_{i,k+1} \text{\ \ for\ \ } k=0,\dots, n_{i}-2,  \\
	m_{i,n_{i}-1}&=s_{i,n_{i}}m_{i,n_{i}}, \\
	m_{i+1,1}&=m_{i,n_{i}}.
	\end{split}\end{equation*}
	Then the multiplicity sequence of $q\in \bar{E}$ equals
	\begin{equation*}
	(m_{1,1},(m_{1,1})_{s_{1,1}},(m_{1,2})_{s_{1,2}},\dots, (m_{g,n_{g}})_{s_{g,n_{g}}}).
	\end{equation*}
\end{lem}
\begin{proof}
	The result follows by applying \cite[3.5.6]{Wall_singular_curves} inductively. See also \cite[2.2]{Tono_nie_bicuspidal} and references there.
\end{proof}

Another important invariant of a cusp $q\in\bar E$ is the semigroup $G(q)$ of (a germ of) $\bar{E}$ at $q$: it is the set of positive integers $(\bar E\cdot C)_{q}$ for all germs $C$ of curves passing through $q$. It is a sub-semigroup of $(\N,0,+)$ and it can be easily computed from the characteristic sequence. Namely, the minimal set of generators $\bar{\beta}_{0},\dots, \bar{\beta}_{n}$ of this semigroup is given by (see \cite[4.3.5, (4.5), (4.4)]{Wall_singular_curves}):
\begin{equation*}\label{eq:semigroups}
\bar{\beta}_{0}=\beta_{0}, \quad \bar{\beta}_{l+1}=\frac{1}{e_{l}}(\beta_{0}\beta_{1}+\sum_{i=1}^{l}e_{i}(\beta_{i+1}-\beta_{i})) \text{\ \  for\ \  } l=0,\dots n-1,
\end{equation*}
where $\beta_{l}$, $e_{l}$ are given by \eqref{eq:charseq}. Let us recall that the Alexander polynomial of the link of $q\in \bar{E}$ is given by $\Delta(t)=1+(t-1)\cdot \sum_{k\not\in G(q)}t^{k}$ (see \cite[2.3]{BoroLivi-HeegaardFloer_and_cusps}).

All the above numerical invariants can be computed one from another (see \cite[4.3.8]{Wall_singular_curves}). We note that if a cusp has only one HN-pair then this HN-pair is equal to the Puiseux pair, to the characteristic sequence and to the minimal set of generators of the semigroup.

\begin{prz}[Computations of characteristics from a parametrization]\label{ex:bicuspidal_(e)}We now compute the Puiseux characteristic sequence for the planar bicuspidal curve $\bar{E}$ of degree $7$ which is the image of  $\phi\colon\P^1\to \P^2$ given by
	\begin{equation*}
	[s:t]\mapsto [s^{3}t^{4}:2s^{6}t+s^{7}:t^{6}s-t^{7}].
	\end{equation*}
	One checks easily that this parametrization is injective and that the rank of the matrix $[\frac{\d \phi}{\d s},\frac{\d \phi}{\d t}]$ is not maximal exactly at $\{[1:0], [0:1]\}$. The singular locus of $\bar E$ consists of two points, $q_{1}\de \phi[1:0]=[0:1:0]$ and $q_{2}\de \phi[0:1]=[0:0:1]$, which are cusps of multiplicities $4$ and $3$, respectively. Choose local analytic parameters $\xi=t/(2t+1)^{1/4}$, $\eta=s/(s-1)^{1/3}$ near, respectively, $[1:0]\in \P^{1}$ and $[0:1]\in \P^{1}$. In the affine charts $\{y\neq 0\}$ and $\{z\neq 0\}$ of $\P^{2}$ choose the standard coordinates $(\frac{x}{y},\frac{z}{y})$ and $(\frac{x}{z},\frac{y}{z})$, respectively. Then one can write $\phi$ near $[1:0]\in \P^{1}$ as $\xi \mapsto (\xi^{4},\xi^{6}(2t+1)^{1/2}-\xi^{7}(2t+1)^{3/4})$, where, using the inverse function theorem, we treat $2t+1$ as an analytic function of $\xi$. A similar formula can be found also for the other cusp. One computes this way that the cusps $q_{1},q_{2}\in \bar{E}$ can be locally analytically parametrized as:
	\begin{equation*}
	(\xi^{4},\xi^{6}-\tfrac{3}{2}\xi^{8}-\tfrac{9}{8}\xi^{9}+\dots) \text{\ \ and\ \ } (\eta^{3},-2\eta^{6}+\eta^{7}+\dots).
	\end{equation*}
	For $q_{1}\in \bar{E}$, this parametrization is as in the definition of the Puiseux characteristic sequence. Using the notation introduced there, for $q_{1}\in \bar{E}$ we have $\beta_{0}=4$, $\beta_{1}=6$, and hence $e_{1}=2$. Because $2|8$, the next term does not contribute to this sequence, so $\beta_{2}=9$. Thus $e_{2}=1$ and therefore the Puiseux sequence has length $3$ and equals $(4;6,9)$. Near $q_{2}\in \bar{E}$ perform a local change of coordinates $(x,y)\mapsto (x,y+2x^{2})$. Then $q_{2}$ is given by $(\eta^{3},\eta^{7}+\dots)$, so $\beta_{0}=3$, $\beta_{1}=7$, hence $e_{1}=1$ and thus the characteristic sequence of $q_{2}\in \bar{E}$ equals $(3;7)$. 
	
	\begin{figure}[h]
		\centering 	\hspace{-1cm} 
		\begin{subfigure}{0.45\textwidth} \includegraphics[scale=0.35]{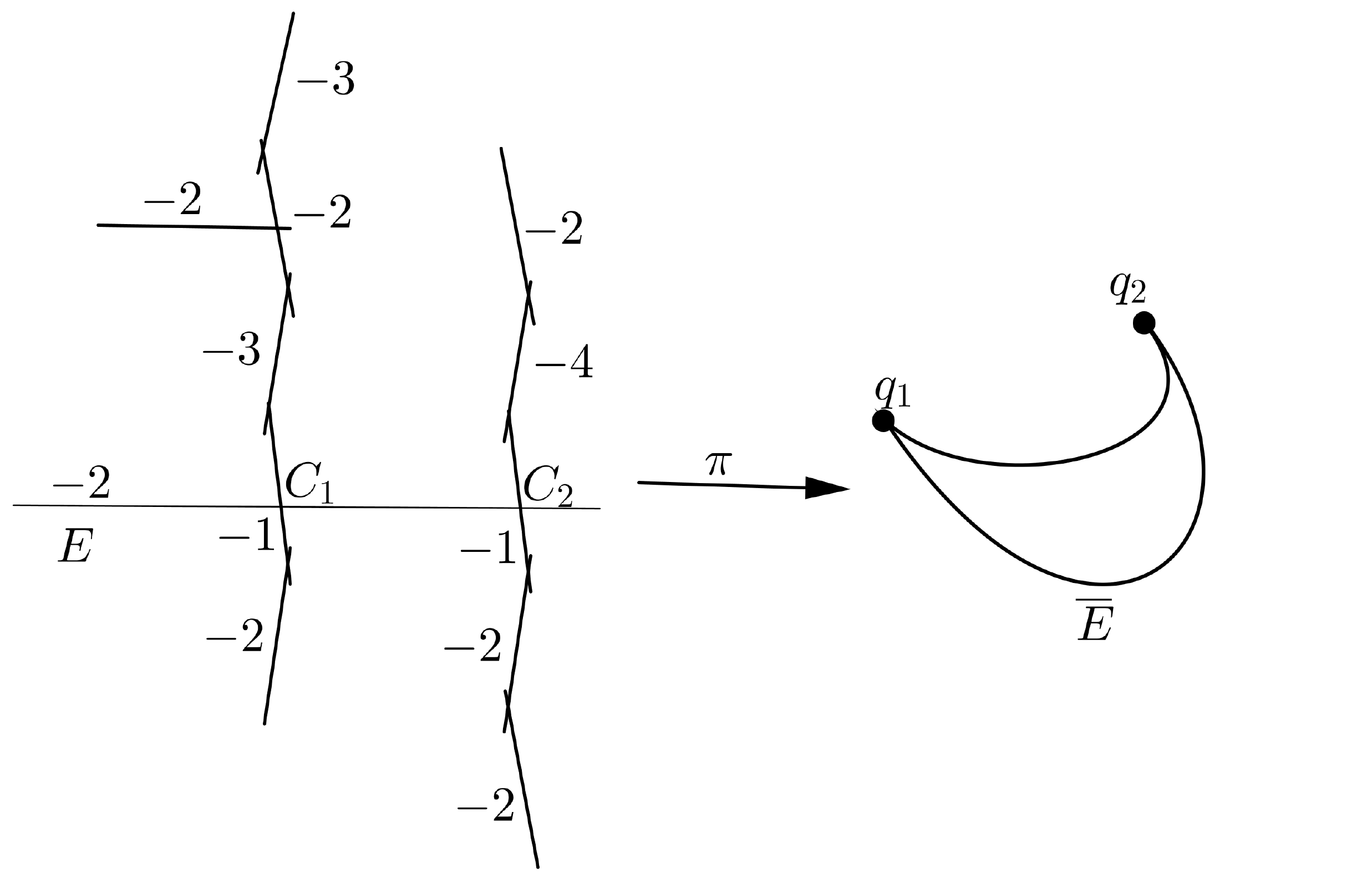} 
		\end{subfigure}
		\hspace{1.6cm}
		\begin{subfigure}{0.3\textwidth} 
			{\footnotesize {\renewcommand{\arraystretch}{1.5}
					\begin{tabular}{r | >{$}c<{$} | >{$}c<{$} }
						& q_{1} & q_{2} \\ \hline 
						Multiplicity seq.    &  (4,(2)_{3}) & ((3)_{2})  \\ \hline 
						HN pairs           &  \binom{6}{4}\binom{2}{3} & \binom{7}{3}  \\ \hline 
						Puiseux seq. &  (4; 6,9)  & (3; 7)   \\ \hline 
						Puiseux pairs & (3,2), (9,2)  &  (3,7)   \\ \hline
						Zariski pairs & (2,3), (2,3)  &  (7,3)   \\ \hline 
						Semigroups     &  \langle 4,6,15 \rangle & \langle 3, 7 \rangle    \\ \hline
					\end{tabular}
				}}
			\end{subfigure}
			\caption{The curve from Example \ref{ex:bicuspidal_(e)}.}	\label{fig:example}
		\end{figure}
		
		The Lemma \ref{lem:HN_vs_Puiseux} implies that $\bar{E}$ is of HN-type $\cA(2,2,1)$. Other invariants and the graphs of the resolution can be computed easily using the results above. Their values are given in the table in Fig.\ \ref{fig:example}.
		
		Looking more carefully at the resolution of $q_1$ we see that the first three blowups (the second one is outer and the third one is inner with respect to the reduced preimage of $q_{1}$) are described by the HN pair $\binom{6}{4}$. The next blowup is outer. After that the proper transform of $\bar{E}$ is smooth, but we need another outer and inner blowup to make the total transform snc. These three blowups are described by the second HN-pair $\binom{2}{3}$. Note that after the first three blowups, the cusp meets the last exceptional curve with multiplicity $2$. Therefore (see the discussion preceding \eqref{eq:HN-equivalence}), if we used HN-pairs as in \cite{KoPaRa-SporadicCstar1} instead of the standard ones, namely, if we required that $p_{2}\leq c_{2}$ and thus chose some general smooth germ $Y_{2}$ transversal to the last exceptional curve instead of the maximally tangent one, then the HN-type would be $\binom{6}{4}\binom{2}{2}\binom{2}{1}$. The fact that we have two Puiseux pairs (equivalently, that the Puiseux sequence is of length $3$) corresponds to the fact that the exceptional divisor over $q_1$ has one branching component.
	\end{prz}

\begin{prz}[Computation of other characteristics from HN-pairs]\label{ex:invariants_from_HN} Let $\gamma,s\geq 1$, $p\geq 2$. Consider a rational cuspidal curve of HN-type $\cC(\gamma,p,s)$. Its standard HN-pairs are
	\begin{equation*}
	\binom{p(\gamma  s+s+1)}{p(\gamma s+1)}\binom{p}{1}
	\text{\ \ and\ \ }\binom{(\gamma +1)(ps+1)+p}{ps+1}.
	\end{equation*}
	Lemma \ref{lem:HN_vs_Puiseux} gives Puiseux sequences
	\begin{equation*}
	(p(\gamma s+1);p(\gamma s+s+1),p(\gamma s +s+1)+1)  \text{\ \ and\ \ } (ps+1;(\gamma+1)(ps+1)+p),
	\end{equation*}
	Puiseux pairs
	\begin{equation*}
	(\gamma s +s+1,\gamma s+1),(p(\gamma s+s+1)+1,p) \text{\ \ and\ \ } ((\gamma+1)(ps+1)+p, ps+1)
	\end{equation*}
	and Zariski pairs
	\begin{equation*}
	(\gamma s +1,\gamma  s+s+1),(1,p) \text{\ \ and\ \ } (ps+1,(\gamma+1)(ps+1)+p).
	\end{equation*}
	Topologically the cusps of $\bar{E}$ are the same as the ones with Puiseux parametrization
	\begin{equation*}
	y=x^{\frac{\gamma s+s+1}{\gamma s +1}}\left(1+x^{\frac{p}{\gamma s +1}} \right)
	\quad \mbox{ and }\quad
	y=x^{\frac{(\gamma+1)(ps+1)+p}{ps+1}}.
	\end{equation*}
	
	To compute multiplicity sequences we either apply Lemma \ref{lem:Hn_gives_multiplicities} or we simply follow the Euclidean algorithm using the fact when passing from a cusp to its proper transform under a blowup the initial HN-pair $\binom{c}{p}$ is replaced by $\binom{\max\{c-p, p\}}{\min \{c-p,p\}}$. Let us make the computation. The multiplicity of $q_{1}\in \bar{E}$ equals $p(\gamma s+1)$. The first blowup replaces the first HN-pair with $\binom{p(\gamma s+1)}{p}$. The next $\gamma s+1$ blowups are centered at a point of multiplicity $p$ on the proper transforms of $\bar{E}$. After they are performed, the first HN-pair becomes $\binom{p}{0}$, which means that the cusp of the proper transform of the germ of the curve is not at the intersection point of proper transforms of local coordinates used to define this HN-pair, so we move on to the next HN-pair $\binom{p}{1}$. The latter describes $p$ blowups at a point of multiplicity $1$ on the proper transform of $\bar{E}$. Therefore, the multiplicity sequence of $q_{1}\in \bar{E}$ is
	\begin{equation*}
	(p(\gamma s+1),(p)_{\gamma s+1},(1)_{p}).
	\end{equation*}
	
	For $q_2\in \bar E$ the first $\gamma+1$ blowups are centered at a point of multiplicity $ps+1$ on the proper transforms of $\bar{E}$, and they replace the HN-pair by $\binom{p}{1}$. As we have seen before, this HN-pair describes $p$ blowups at smooth points of the proper transforms of $\bar{E}$. Thus the multiplicity sequence of $q_{2}\in \bar{E}$ is
	\begin{equation*}
	((ps+1)_{\gamma+1},(p)_{s},(1)_{p}).
	\end{equation*}

\end{prz}

\begin{landscape}
\begin{tiny}
\begin{table}
\addcontentsline{toc}{section}{TABLE: Classification}
{\renewcommand{\arraystretch}{2.5}
\begin{tabular}{>{$}r<{$}|>{$}c<{$}|>{$}c<{$}|>{$}c<{$}|>{$}c<{$}|>{$}c<{$}|>{$}c<{$}|c}

   & c & \mbox{degree} & -E^{2} & \parbox{5.5 cm}{\centering HN pairs (standard, except ($^{*}\!$) for $s=1$, $k=1$ or ($^{\dagger}\!$) for $\gamma=1$, see Rem.\ \ref{rem:special_HN_cases})} & \mbox{Multiplicity sequences} & \mbox{Parameters} & References \\  \hline 

\FZa(d,k) & 3 & d & d-2 & \binom{2k+1}{2},\quad \binom{d-1}{d-2}, \quad \binom{2(d-2-k)+1}{2} & ((2)_{k}),\quad (d-2),\quad ((2)_{d-2-k}) & d-3\geq k \geq \frac{1}{2}d-1 \geq 1 & \cite{FLZa-_class_of_cusp} \\ \hline 

\multirow{2}{*}{$\cA(\gamma,p,s)$} 
& \multirow{2}{*}{$2$} 
& \multirow{2}{*}{$(\gamma +1)ps+1$}  
& \multirow{2}{*}{$\gamma $} 
& \binom{(\gamma +1)ps}{\gamma ps}\binom{ps}{p}\binom{p}{1} ^{*\dagger}
& (\gamma ps,(ps)_{\gamma },(p)_{s})
& p\geq 2,\ \gamma,s\geq 1
& \cite[(ii.1)]{CKR-Cstar_good_asymptote}, \cite[(e)]{BoZo-annuli}  \\

&&&& \binom{\gamma (ps+1)+p(s-1)+1}{ps+1} 
& ((ps+1)_{\gamma },p(s-1)+1,(p)_{s-1}) 
& \mbox{and } (p,\gamma)\neq (2,1)
& \cite[(3)]{Tono_nie_bicuspidal} \\ \hline

\multirow{2}{*}{$\cB(\gamma,p,s)$} 
& \multirow{2}{*}{$2$} 
& \multirow{2}{*}{$(\gamma+1)ps-\gamma$} 
& \multirow{2}{*}{$\gamma$} 
& \binom{(ps-1)(\gamma +1)}{(ps-1)\gamma }\binom{ps-1}{p} ^{\dagger}
& (\gamma ps-\gamma ,(ps-1)_{\gamma },(p)_{s-1},p-1) 
& p,s\geq 2,\ \gamma\geq 1
& \cite[(ii.2)]{CKR-Cstar_good_asymptote}, \cite[(f)]{BoZo-annuli} \\
&&&& \binom{p(\gamma s+s-1)}{ps}\binom{p}{1} 
& ((ps)_{\gamma },p(s-1),(p)_{s-1})  
& \mbox{and } (p,\gamma)\neq (2,1)
& \cite[(4)]{Fenske_1and2-cuspidal_curves}, \cite[(4)]{Tono_nie_bicuspidal}  \\ \hline 

\multirow{2}{*}{$\cC(\gamma,p,s)$} 
& \multirow{2}{*}{$2$} 
& \multirow{2}{*}{$(\gamma s+s+1)p+1$} 
& \multirow{2}{*}{$\gamma $} 
& \binom{p(\gamma  s+s+1)}{p(\gamma s+1)}\binom{p}{1}, 
&  (\gamma ps+p,(ps)_{\gamma },(p)_{s}) 
& \multirow{4}{*}{$p\geq 2,\ \gamma,s\geq 1$}
&  \cite[(ii.4)]{CKR-Cstar_good_asymptote}, \cite[(c)]{BoZo-annuli}  \\ 
&&&& \binom{(\gamma +1)(ps+1)+p}{ps+1} 
& ((ps+1)_{\gamma +1},(p)_{s}) 
&
&  \cite[(5)]{Fenske_1and2-cuspidal_curves}, \cite[(2)]{Tono_nie_bicuspidal} \\ \cline{1-6}\cline{8-8}
 
\multirow{2}{*}{$\cD(\gamma,p,s)$} 
& \multirow{2}{*}{$2$} 
& \multirow{2}{*}{$(\gamma s+s+1)p-\gamma $} 
& \multirow{2}{*}{$\gamma $} 
& \binom{(\gamma +1)(ps-1)+p)}{\gamma (ps-1)+p}, 
& (\gamma (ps-1)+p,(ps-1)_{\gamma },(p)_{s-1},p-1)  
&
& \cite[(ii.4)]{CKR-Cstar_good_asymptote}, \cite[(d)]{BoZo-annuli}  \\ 
&&&& \binom{p(\gamma s+s+1)}{ps}\binom{p}{1} ^{*}
& ((ps)_{\gamma +1},(p)_{s}) 
&
&  \cite[(6)]{Fenske_1and2-cuspidal_curves}, \cite[(1)]{Tono_nie_bicuspidal} \\ \hline 

\multirow{2}{*}{$\cE(k)$} 
& \multirow{2}{*}{$2$} 
& \multirow{2}{*}{$8k+6$} 
& \multirow{2}{*}{$2$} 
& \binom{8k+8}{4k+2}\binom{2}{1} 
& ((4k+2)_{2},(4)_{k},(2)_{2})
& \multirow{4}{*}{$k\geq 1$}
& \cite[(s)]{BoZo-annuli},  \\
&&&& \binom{8k+4}{4k+4}\binom{4}{1}
& (4k+4,4k,(4)_{k}) 
& 
& \cite[(a)]{KoPa-SporadicCstar2}  \\ \cline{1-6}\cline{8-8}

\multirow{2}{*}{$\cF(k)$} 
& \multirow{2}{*}{$2$} 
& \multirow{2}{*}{$8k+2$} 
& \multirow{2}{*}{$2$} 
& \binom{8k}{4k+2}\binom{2}{1} 
& (4k+2,4k-2,(4)_{k-1},(2)_{2})
& 
& \cite[(s)]{BoZo-annuli},  \\
&&&& \binom{8k+4}{4k}\binom{4}{1} ^{*}
& ((4k)_{2},(4)_{k}) 
& 
&  \cite[(b)]{KoPa-SporadicCstar2}   \\ \hline

\cG(\gamma) & 2 & 2\gamma-1 &  \gamma & \binom{4\gamma-3}{\gamma-1},\quad   \binom{2\gamma-1}{2} & ((\gamma-1)_{4}), ((2)_{\gamma-1}) & \gamma\geq 3 &  \cite[VII]{tomDieck_letter}, \cite[(b)]{Bodnar_type_G_and_J} \\ \hline

\ORa(k) & 1 & F_{4k+2} & 2 & \binom{F_{4k+4}}{F_{4k}}\binom{3}{1}^{*} &  (F_{4k},((F_{4l})_{5},F_{4l}-F_{4l-4})_{l=k}^{1}) & \multirow{2}{*}{$k\geq 1$} &\cite{OrevkovCurves} \\ \cline{1-6}\cline{8-8}

\ORb(k) & 1 & 2F_{4k+2} & 2 & \binom{2F_{4k+4}}{2F_{4k}}\binom{6}{1}^{*} &  (2F_{4k},((2F_{4l})_{5},(2F_{4l}-2F_{4l-4})_{l=k}^{1}) & & \cite{OrevkovCurves} \\ \hline

\multicolumn{8}{c}{where $F_{j}$ are the Fibonacci numbers defined by $F_{0}=0$, $F_{1}=1$, $F_{j}=F_{j-1}+F_{j-2}$.} \\ \hline
\end{tabular}
}
\vspace{2em}
\caption{Numerical data for rational cuspidal curves $\bar{E}\subseteq \P^{2}$ such that $\kappa(\P^{2}\setminus \bar{E})= 2$ and $\P^{2}\setminus \bar{E}$ is $\C^{**}$-fibered.}
\label{table:fibrations}
\end{table}
\end{tiny}
\end{landscape}

\bibliographystyle{amsalpha}
\bibliography{bibl2016}

\providecommand{\bysame}{\leavevmode\hbox to3em{\hrulefill}\thinspace}
\providecommand{\MR}{\relax\ifhmode\unskip\space\fi MR }
\providecommand{\MRhref}[2]{%
  \href{http://www.ams.org/mathscinet-getitem?mr=#1}{#2}
}
\providecommand{\href}[2]{#2}
\begin{thebibliography}{FdBLMHN07b}

\bibitem[BHPVdV04]{BHPV_complex_surfaces}
Wolf~P. Barth, Klaus Hulek, Chris A.~M. Peters, and Antonius Van~de Ven,
  \emph{Compact complex surfaces}, second ed., Results in Mathematics and
  Related Areas. 3rd Series. A Series of Modern Surveys in Mathematics, vol.~4,
  Springer-Verlag, Berlin, 2004.

\bibitem[BL14]{BoroLivi-HeegaardFloer_and_cusps}
Maciej Borodzik and Charles Livingstone, \emph{Heegaard {F}loer homology and
  rational cuspidal curves}, Forum Math. Sigma \textbf{2} (2014), e1, 28,
  \arxiv{1304.1062}.

\bibitem[Bod16]{Bodnar_type_G_and_J}
József Bodnár, \emph{Construction of bicuspidal rational complex projective
  plane curves}, \arxiv{1608.02921}, 2016.

\bibitem[Bre93]{Bredon_top_and_geom}
Glen~E. Bredon, \emph{Topology and geometry}, Graduate Texts in Mathematics,
  Springer New York, 1993.

\bibitem[BZ10]{BoZo-annuli}
Maciej Borodzik and Henryk \.{Z}oł\c{a}dek, \emph{Complex algebraic plane
  curves via {P}oincar\'e-{H}opf formula. {II}. {A}nnuli}, Israel J. Math.
  \textbf{175} (2010), 301--347,
  \href{http://arxiv.org/abs/0708.1661}{arXiv:0708.1661}.

\bibitem[CNKR09]{CKR-Cstar_good_asymptote}
Pierrette Cassou-Nogues, Mariusz Koras, and Peter Russell, \emph{Closed
  embeddings of {$\mathbb{C}^*$} in {$\mathbb{C}^2$}. {I}}, J. Algebra
  \textbf{322} (2009), no.~9, 2950--3002.

\bibitem[EN85]{EN_Newton_pairs}
David Eisenbud and Walter Neumann, \emph{Three-dimensional link theory and
  invariants of plane curve singularities}, Annals of Mathematics Studies, vol.
  110, Princeton University Press, Princeton, NJ, 1985.

\bibitem[FdBLMHN07a]{FLMN_cusps_and_open_surfaces}
J.~Fern{\'a}ndez~de Bobadilla, I.~Luengo, A.~Melle-Hern{\'a}ndez, and
  A.~N{\'e}methi, \emph{On rational cuspidal plane curves, open surfaces and
  local singularities}, Singularity theory, World Sci. Publ., Hackensack, NJ,
  2007, pp.~411--442.

\bibitem[FdBLMHN07b]{FLMN_one_pair}
Javier Fern{\'a}ndez~de Bobadilla, Ignacio Luengo, Alejandro
  Melle~Hern{\'a}ndez, and Andras N{\'e}methi, \emph{Classification of rational
  unicuspidal projective curves whose singularities have one {P}uiseux pair},
  Real and complex singularities, Trends Math., Birkh\"auser, Basel, 2007,
  pp.~31--45.

\bibitem[Fen99]{Fenske_1and2-cuspidal_curves}
Torsten Fenske, \emph{Rational 1- and 2-cuspidal plane curves.}, Beiträge zur
  Algebra und Geometrie \textbf{40} (1999), no.~2, 309--329.

\bibitem[Fuj82]{Fujita-noncomplete_surfaces}
Takao Fujita, \emph{On the topology of noncomplete algebraic surfaces}, J. Fac.
  Sci. Univ. Tokyo Sect. IA Math. \textbf{29} (1982), no.~3, 503--566.

\bibitem[FZ94]{FZ-deformations}
Hubert Flenner and Mikhail Zaidenberg, \emph{{$\mathbb{Q}$}-acyclic surfaces
  and their deformations}, Classification of algebraic varieties ({L}'{A}quila,
  1992), Contemp. Math., vol. 162, Amer. Math. Soc., Providence, RI, 1994,
  pp.~143--208.

\bibitem[FZ96]{FLZa-_class_of_cusp}
\bysame, \emph{On a class of rational cuspidal plane curves}, Manuscripta Math.
  \textbf{89} (1996), no.~4, 439--459.

\bibitem[FZ03]{FlZa-rational_curves_and_singularities}
\bysame, \emph{Rational curves and rational singularities}, Math. Z.
  \textbf{244} (2003), no.~3, 549--575.

\bibitem[GLS07]{GLS_Intro_to_singularities}
G.-M. Greuel, C.~Lossen, and E.~Shustin, \emph{Introduction to singularities
  and deformations}, Springer Monographs in Mathematics, Springer, Berlin,
  2007.

\bibitem[GP99]{GuPrad-rationality_3(smooth_II)}
R.~V. Gurjar and C.~R. Pradeep, \emph{{${\bf Q}$}-homology planes are rational.
  {III}}, Osaka J. Math. \textbf{36} (1999), no.~2, 259--335.

\bibitem[Har77]{Hartshorne_AG}
Robin Hartshorne, \emph{Algebraic geometry}, Springer-Verlag, New York, 1977,
  Graduate Texts in Mathematics, No. 52.

\bibitem[Iit82]{Iitaka_AG}
Shigeru Iitaka, \emph{Algebraic geometry}, Graduate Texts in Mathematics,
  vol.~76, Springer-Verlag, New York, 1982, An introduction to birational
  geometry of algebraic varieties, North-Holland Mathematical Library, 24.

\bibitem[Kas87]{Kashiwara}
Hiroko Kashiwara, \emph{Fonctions rationnelles de type $(0,1)$ sur le plan
  projectif complexe}, Osaka J. Math. \textbf{24} (1987), no.~3, 521--577.

\bibitem[KP15]{KoPa-CooligeNagata2}
Mariusz Koras and Karol Palka, \emph{The {C}oolidge--{N}agata conjecture},
  \arxiv{1502.07149}, 2015.

\bibitem[KP16]{KoPa-SporadicCstar2}
\bysame, \emph{Classification of sporadic {$\mathbb{C^*}$}-embeddings into
  {$\mathbb{C}^2$}}, preprint, 2016.

\bibitem[KPR16]{KoPaRa-SporadicCstar1}
Mariusz Koras, Karol Palka, and Peter Russell, \emph{The geometry of sporadic
  {$\Bbb{C}^*$}-embeddings into {$\Bbb{C}^2$}}, J. Algebra \textbf{456} (2016),
  207--249, \arxiv{1405.6872}.

\bibitem[KR99]{KR-C*_actions_on_C3}
Mariusz Koras and Peter Russell, \emph{{${\bf C}\sp *$}-actions on {${\bf C}\sp
  3$}: the smooth locus of the quotient is not of hyperbolic type}, J.
  Algebraic Geom. \textbf{8} (1999), no.~4, 603--694.

\bibitem[Liu14]{Liu-thesis}
Tiankai Liu, \emph{On planar rational cuspidal curves}, Ph.\ D.\ thesis,
  \href{http://hdl.handle.net/1721.1/90190}{http://hdl.handle.net/1721.1/90190},
  2014.

\bibitem[Miy01]{Miyan-OpenSurf}
Masayoshi Miyanishi, \emph{Open algebraic surfaces}, CRM Monograph Series,
  vol.~12, American Mathematical Society, Providence, RI, 2001.

\bibitem[Moe08]{Moe-cuspidal_MSc}
Karoline~Torgunn Moe, \emph{Rational cuspidal curves}, M.Sc. thesis,
  \href{http://folk.uio.no/torgunnk/MOE_MASTER.pdf}{http://folk.uio.no/torgunnk/MOE\_MASTER.pdf},
  2008.

\bibitem[MS81]{MiySu-curves_with_negative_kod_of_complement}
Masayoshi Miyanishi and Tohru Sugie, \emph{On a projective plane curve whose
  complement has logarithmic kodaira dimension $-\infty $}, Osaka J. Math.
  \textbf{18} (1981), no.~1, 1--11.

\bibitem[MS91]{MiySu-Cstst_fibrations_on_Qhp}
\bysame, \emph{${\Q}$-homology planes with ${\C^{**}}$-fibrations}, Osaka J.
  Math. \textbf{28} (1991), no.~1, 1--26.

\bibitem[MT92]{MiTs-lines_on_qhp}
M.~Miyanishi and S.~Tsunoda, \emph{Absence of the affine lines on the homology
  planes of general type}, J. Math. Kyoto Univ. \textbf{32} (1992), no.~3,
  443--450.

\bibitem[Ore02]{OrevkovCurves}
Stephan~Yu. Orevkov, \emph{On rational cuspidal curves}, Math. Ann.
  \textbf{324} (2002), no.~4, 657--673.

\bibitem[Pal12]{Palka-classification2_Qhp}
Karol Palka, \emph{Classification of singular {$\Bbb Q$}-homology planes {II}:
  {$\Bbb C^1$}- and {$\Bbb C^*$}-rulings}, Pacific J. Math. \textbf{258}
  (2012), no.~2, 421--457, \arxiv{1201.2463}.

\bibitem[Pal14]{Palka-Coolidge_Nagata1}
\bysame, \emph{The {C}oolidge--{N}agata conjecture, part {I}}, Adv. Math.
  \textbf{267} (2014), 1--43,
  \href{http://arxiv.org/abs/1405.5917}{arXiv:1405.5917}.

\bibitem[Pal15]{Palka-AMS_LZ}
\bysame, \emph{A new proof of the theorems of {L}in-{Z}aidenberg and
  {A}bhyankar-{M}oh-{S}uzuki}, J. Algebra Appl. \textbf{14} (2015), No. 8,
  \href{http://arxiv.org/abs/1405.5391}{arXiv:1405.5391}.

\bibitem[Pal16]{Palka-minimal_models}
\bysame, \emph{Cuspidal curves, minimal models and {Z}aidenberg's finiteness
  conjecture}, Journal für die reine und angewandte Mathematik (Crelles
  Journal) (2016), \href{http://arxiv.org/abs/1405.5346}{arXiv:1405.5346}.

\bibitem[Rus80]{Russell_HN_pairs}
Peter Russell, \emph{Hamburger-{N}oether expansions and approximate roots of
  polynomials}, Manuscripta Math. \textbf{31} (1980), no.~1-3, 25--95.

\bibitem[Suz77]{Suzuki-chi_for_fibrations}
Masakazu Suzuki, \emph{Sur les op\'erations holomorphes du groupe additif
  complexe sur l'espace de deux variables complexes}, Ann. Sci. \'Ecole Norm.
  Sup. (4) \textbf{10} (1977), no.~4, 517--546.

\bibitem[tD95]{tomDieck_letter}
Tammo tom Dieck, \emph{Letter to {H}. {F}lenner}, 1995.

\bibitem[Ton00]{Tono_doctoral_thesis}
Keita Tono, \emph{Defining equations of certain rational cuspidal plane
  curves}, Doctoral thesis, Saitama University, 2000.

\bibitem[Ton01]{Tono_1cusp_with_kod_1}
\bysame, \emph{Rational unicuspidal plane curves with $\bar{\kappa}=1$}, Newton
  polyhedrons and Singularities, vol. 1233, RIMS Kokyuroku, 2001, pp.~82--89.

\bibitem[Ton12a]{Tono_nie_bicuspidal}
\bysame, \emph{On a new class of rational cuspidal plane curves with two
  cusps}, \href{http://arxiv.org/abs/1205.1248}{arXiv:1205.1248}, 2012.

\bibitem[Ton12b]{Tono-on_Orevkov_curves}
\bysame, \emph{On {O}revkov's rational cuspidal plane curves}, J. Math. Soc.
  Japan \textbf{64} (2012), no.~2, 365--385.

\bibitem[Ton12c]{Tono-cusps_self-intersections}
\bysame, \emph{On the self-intersection number of the nonsingular models of
  rational cuspidal plane curves},
  \href{http://arxiv.org/abs/1408.0190}{arXiv:1408.0190}, 2012.

\bibitem[Tsu81]{Tsunoda_cusps_complement_kod_0}
Shuichiro Tsunoda, \emph{The complements of projective plane curves}, vol. 446,
  RIMS Kokyuroku, 1981, pp.~48--56.

\bibitem[Wal04]{Wall_singular_curves}
Charles Terence~Clegg Wall, \emph{Singular points of plane curves}, London
  Mathematical Society Student Texts, Cambridge University Press, 2004.

\bibitem[Zai87]{Zaid_isotrivial_curves_on_surfaces}
Mikhail Zaidenberg, \emph{Isotrivial families of curves on affine surfaces, and
  the characterization of the affine plane}, Izv. Akad. Nauk SSSR Ser. Mat.
  \textbf{51} (1987), no.~3, 534--567, 688.

\end{thebibliography}

\end{document}